\documentclass[a4paper,reqno]{amsart}

\usepackage{amsmath, amssymb, amsthm,mathtools,stmaryrd}
\usepackage{thmtools}
\usepackage{cancel}
\usepackage[normalem]{ulem}
\usepackage{hyperref}
\usepackage{cleveref}
\usepackage[english]{babel}
\usepackage{enumitem}

\usepackage[all]{xy}
\usepackage{graphicx}
\usepackage{mathtools}
\usepackage{mathrsfs}

\usepackage{verbatim}

\usepackage{color}
\usepackage{tikz-cd}
\tikzcdset{arrow style=tikz, diagrams={>=stealth}}

\usepackage[most]{tcolorbox}
  {
\begin{tcolorbox}[breakable,
 enhanced jigsaw,
 opacityback=0,
 sharp corners,
 parbox=false,
 boxrule=0mm,
 top=0mm,bottom=0pt,left=0pt,right=0pt,
 boxsep=0pt,
 frame hidden,
 parbox=false,
 before upper=\indent\strut,
 before=\par,after=\par,
 finish={\draw[thick,red] ([xshift=-0.4\textwidth]frame.north)--([xshift=0.4\textwidth]frame.north)--([xshift=-0.4\textwidth]frame.south)--([xshift=0.4\textwidth]frame.south);}
]}
{\end{tcolorbox}}

\usepackage{tikz-cd}
\tikzcdset{arrow style=tikz, diagrams={>=stealth}}

\newcommand\leer{\text{---}}

\allowdisplaybreaks

\DeclareMathOperator{\Aut}{Aut}
\DeclareMathOperator{\Ad}{Ad}  %
\DeclareMathOperator{\Hom}{Hom}
\newcommand*{\id}{\textup{id}}      %

\numberwithin{equation}{section}

\theoremstyle{plain}

\newtheorem{thm}{Theorem}[section]
\newtheorem{lem}[thm]{Lemma}
\newtheorem{prop}[thm]{Proposition}
 \newtheorem{cor}[thm]{Corollary}
\newtheorem{defi}[thm]{Definition}
\newtheorem{conj}[thm]{Conjecture}
\theoremstyle{remark}

\newtheorem{rem}[thm]{Remark}

\newcommand{\ltau}{\overset\leftharpoonup\tau}
\newcommand{\rtau}{\overset\rightharpoonup\tau}
\newcommand{\lcan}{\overset\leftharpoonup\can}
\newcommand{\rcan}{\overset\rightharpoonup\can}
\numberwithin{equation}{section}
\newcommand\inv{^{-1}}
\newcommand{\nn}{\nonumber}
\newcommand{\ot}{\otimes}

\newcommand{\beq}{\begin{equation}}
\newcommand{\eeq}{\end{equation}}

\newcommand{\Bi}{\textup{BiGal}}

\newcommand{\cL}{\mathcal{L}}

\newcommand{\BB}{\overline{B}}
\newcommand{\BG}{{B^{\Gamma}}}

\newcommand{\M}{\mathcal{M}}

\newcommand{\CG}{{\cdot_{\Gamma}}}
\newcommand{\GH}{{\Gamma^{\#}}}

\newcommand{\one}[1]{{#1}{}_{\scriptscriptstyle{(1)}}}
\newcommand{\two}[1]{{#1}{}_{\scriptscriptstyle{(2)}}}

\newcommand{\tuno}[1]{{#1}{}{}^{\scriptscriptstyle{<1>}}}
\newcommand{\tdue}[1]{{#1}{}{}^{\scriptscriptstyle{<2>}}}

\newcommand{\yi}[1]{{#1}{}{}^{\scriptscriptstyle{[1]}}}
\newcommand{\er}[1]{{#1}{}{}^{\scriptscriptstyle{[2]}}}

\newcommand{\teins}[1]{{#1}{}{}^{\scriptscriptstyle{(1)}}}
\newcommand{\tzwei}[1]{{#1}{}{}^{\scriptscriptstyle{(2)}}}

\newcommand{\sw}[1]{{}_{(#1)}}
\newcommand{\swm}[1]{{}_{(-#1)}}

\newcommand\co{{\operatorname{co}}}

\newcommand{\lbiprod}{{>\!\!\!\triangleleft\kern-.33em\cdot}}
\newcommand{\rbiprod}{{\cdot\kern-.33em\triangleright\!\!\!<}}

\newcommand{\z}{{}_{\scriptscriptstyle{(0)}}}
\newcommand{\rz}{{}_{\scriptscriptstyle{[0]}}}
\newcommand{\tz}{{}_{\scriptscriptstyle{\{0\}}}}
\renewcommand{\o}{{}_{\scriptscriptstyle{(1)}}}
\newcommand{\ro}{{}_{\scriptscriptstyle{[1]}}}
\newcommand{\mo}{{}_{\scriptscriptstyle{(-1)}}}
\newcommand{\rmo}
{{}_{\scriptscriptstyle{[-1]}}}
\newcommand{\tone}{{}_{\scriptscriptstyle{\{1\}}}}
\newcommand{\tmo}{{}_{\scriptscriptstyle{\{-1\}}}}
\renewcommand{\t}{{}_{\scriptscriptstyle{(2)}}}

\newcommand{\mt}{{}_{\scriptscriptstyle{(-2)}}}

\newcommand{\ttwo}{{}_{\scriptscriptstyle{\{2\}}}}
\renewcommand{\th}{{}_{\scriptscriptstyle{(3)}}}

\newcommand{\mth}{{}_{\scriptscriptstyle{(-3)}}}

\newcommand{\mfo}{{}_{\scriptscriptstyle{(-4)}}}

\newcommand{\di}{{\diamond_{B}}}
\newcommand{\la}{{\triangleright}}
\newcommand{\ra}{{\triangleleft}}
\newcommand{\bla}{{\blacktriangleright}}

\DeclareMathOperator{\tens}{\otimes}

\newcommand{\CH}{{\mathcal{H}}}
\newcommand{\CM}{\mathcal{M}}
\newcommand{\CL}{\mathcal{L}}

\newcommand\HMod[4]{{^{#1}_{#3}\mathcal M^{#2}_{#4}}}
\newcommand\LMod[1]{{_{#1}\mathcal M}}
\newcommand\LComod[1]{{^{#1}\mathcal M}}
\newcommand\RComod[1]{\HMod{}{#1}{}{}}
\newcommand\RMod[1]{\mathcal M_{#1}}

\newcommand\BiMod[1]{{_{#1}\mathcal M_{#1}}}
\newcommand\Bimod[2]{{_{#1}\mathcal M_{#2}}}
\newcommand\vect{\underline{\operatorname{Vect}}}
\newcommand{\ou}[1]{\underset{{#1}}{\otimes}}
\newcommand\ol{\overline}

\newcommand{\can}{{\rm can}}

\newcommand\cop{{\operatorname{cop}}}
\newcommand{\op}{{\operatorname{op}}}
\newcommand\sym{{\operatorname{sym}}}

\begin{document}
\thanks{For the purpose of Open Access, a CC-BY public copyright licence has been applied by the author to the present document and will be applied to all subsequent versions up to the Author Accepted Manuscript arising from this submission}
\author{Xiao Han}
\address{Queen Mary University of London}
\email{x.h.han@qmul.ac.uk}

\author{Peter Schauenburg}
\address{Université Bourgogne Europe, CNRS, IMB UMR 5584, 21000 Dijon, France}
\email{peter.schauenburg@ube.fr}


\title{Hopf BiGalois Theory for Hopf Algebroids}

\begin{abstract}We develop a theory of Hopf BiGalois extensions for Hopf algebroids. We understand these to be left bialgebroids (whose left module categories are monoidal categories) fulfilling a condition that is equivalent to being Hopf in the case of ordinary bialgebras, but does not entail the existence of an antipode map.  

The immediate obstacle to developing a full biGalois theory for such Hopf algebroids is simple: The condition to be a left Hopf Galois extension can be defined in complete analogy to the Hopf case, but the Galois map for a right comodule algebra is not a well defined map. We find that this obstacle can be circumvented using bialgebroids fulfilling a condition that still does not entail the existence of an antipode, but is equivalent, for ordinary bialgebras, to being Hopf with bijective antipode. 

The key technical tool is a result of Chemla \cite{C20} allowing to switch left and right comodule structures under flatness conditions much like one would do using an antipode. Using this, we arrive at a left-right symmetric theory of biGalois extensions, including the construction of an Ehresmann Hopf algebroid making a one-sided Hopf-Galois extension into a biGalois one.

Moreover, we apply a more general 2-cocycle twist theory\cite{HM25} to Ehresmann Hopf algebroids. As the 2-cocycle is only left $\BB$-linear the base algebra is also twisted. We also study the Ehresmann Hopf algebroids of quantum Hopf fibrations and quantum homogeneous space as examples.

 \end{abstract}
\maketitle

\section{Introduction}

Similarly to the relation between Hopf algebras and groups, Hopf algebroids can be viewed as a quantization of groupoids. There are several different definitions of Hopf algebroids, in this paper, we will consider the Hopf algebroids introduced in \cite{schau1} and \cite{schau3}. %
In particular, our Hopf algebroids do not necessarily have antipodes, but we require the bijectivity of certain maps that would be equivalent to the existence of a bijective antipode in the case of bialgebras over a commutative ring. Generalizations of these maps also define the notion of Hopf-Galois extension.

For a left bialgebroid $\mathcal{L}$ over a noncommutative ring $B$, a left $\mathcal{L}$-Galois extension can be  defined in a natural way: The definition requires bijectivity of the "Galois map" which can be written down by the same formula as in the ordinary Hopf case, although now not only the source but also the target is a tensor product over a possibly noncommutative algebra rather than the base field. However, there is no obvious definition of a  Hopf Galois extension on the right-hand side; the direct analog of the canonical map in the definition of "Galois" for the ordinary Hopf case is not well defined in the Hopf algebroid case. This makes it impossible to arrive at a naive direct generalization of the notion of Hopf biGalois extension as in \cite{schau5} to the Hopf algebroid setting.

To deal with this problem we study anti-right $\mathcal{L}$-Galois extensions. A short definition of this notion is that $P$ is an anti-right $\mathcal{L}$-Galois extension iff it is a left $\mathcal L^\cop$-Galois extension over the co-opposite bialgebroid. If $\mathcal L$ is an ordinary Hopf algebra with bijective antipode, it was already observed in \cite{schneider} that the Galois and anti-Galois conditions are equivalent; in the Hopf algebroid case however, only one of them is defined for left comodule algebras, and only the other one is defined for right comodule algebras.

A $B$-bialgeroid $\CL$ is a left Hopf algebroid iff the obvious left $\CL$-comodule algebra $\CL$ is a left $\CL$-Galois extension of $\BB$. In the case of ordinary bialgebras this condition is equivalent to the existence of an antipode. Similarly, $\CL$ is an anti-left Hopf algebroid  if and only if  $B\subseteq \mathcal{L}$ is an anti-right $\cL$-Galois extension. In the ordinary bialgebra case this is equivalent to the existence of a skew antipode. In particular the condition on a bialgebroid to be both a left Hopf algebroid and an anti-left Hopf algebroid is equivalent, in the ordinary bialgebra case, to having a bijective antipode. We will also see that under some faithful flatness conditions, a bialgebroid admitting a left Galois (resp.\ anti-right Galois) extension is necessarily (anti-)Hopf, similarly to the classical case.

For a classical Hopf algebra $H$ with bijective antipode and a left $H$-comodule $P$, one can always make $P$ into a right $H$ comodule, by defining the right coaction $\delta:p\to p\z\ot S^{-1}(p\mo)$. In fact this construction can be generalized to Hopf algebroids in absence of an antipode (and under some mild flatness conditions). This was shown in \cite{CGK16}; the general result best suited to our purposes is due to Chemla \cite{C20}.  

 A consistent use of this side-switching trick will allow us to establish a full and left-right symmetric theory of Hopf biGalois extensions, always assuming our Hopf algebroids fulfill the analogous condition to having a bijective antipode (while not having an antipode), and suitable (faithful) flatness conditions over the base. 
 
 For the rest of this introduction we will not mention these important conditions, nor the (faithful) flatness conditions needed frequently on our extensions and Hopf algebroids to make things work.

The facts on one-sided Hopf Galois extensions that we need will be established in \cref{sec:onesided}, including a study of the analogs of Schneider's structure theorems on Hopf modules, and an extensive discussion of the technical details of switching sides as alluded to above. 

\Cref{sec:biGalois} contains the main general results on biGalois theory. In particular, following a result on the structure of Hopf bimodules based on the structure theorems obtained in \cref{sec:onesided}, we proceed to construct an Ehresmann Hopf algebroid for any one-sided Galois extension over a Hopf algebroid. The Ehresmann Hopf algebroid makes the one-sided Galois extension into a biGalois extension (which we define) in the essentially unique way. We prove that biGalois extensions form a groupoid under cotensor product, with biGalois extensions giving rise to monoidal equivalences between comodule categories over Hopf algebroids. We conjecture that all monoidal equivalences between comodule categories are given by biGalois extensions, but can only prove a partial result in this direction. 

In \cref{sec:structhms} we prove several structure theorems on Hopf bimodules based on the generalizations of Schneider's structure theorem. In particular, Hopf modules with two module structures over a Galois extension are classified (as a monoidal category) as modules over the Ehresmann Hopf algebroid, and Hopf modules with two module structures over a biGalois extension and two comodule structures over the left and right Hopf algebroids correspond to Yetter-Drinfeld modules. This generalizes the structure theorems for Hopf modules over a Hopf algebroid that appeared in \cite{CK25} while the authors were finalizing the present paper. The paper \cite{CK25} also discusses the braidings we discuss here.

In \Cref{Applications and examples}, we apply a new 2-cocycle twist theory\cite{HM25} to the Ehresmann Hopf algebroid. More precisely, compared with the twist theory in \cite{Boehm} and \cite{HM22}, a 2-cocycle in the new theory is only left $\BB$-linear (the 2-cocycle in \cite{Boehm} is left $B^e$-linear). This results in a deformed Hopf algebroid whose base is also deformed. Moreover, the coproduct and product are simultaneously deformed. Here we show Hopf Galois extensions can be twisted by a 2-cocycle with a deformed coaction and product. In particular, we show for any Hopf algebra $H$, if a $H$-Galois extension $B\subseteq P$ has an extra symmetry (left coaction which cocommutes with the right $H$-coaction) associated to a Hopf algebra $K$, then a 2-cocycle $\gamma$ on $K$ induces a 2-cocycle $\Gamma$ on the corresponding Ehresmann Hopf algebroid $L(P, H)$. Moreover $L(P, H)^{\Gamma}\cong L({}_{\gamma}P, H)$, where ${}_{\gamma}B\subseteq{}_{\gamma}P$ is another $H$-Galois extension invented in \cite{ppca} with a deformed base. Finally, we present two examples of this setup. The first one is a deformation of Hopf fibration $S^7\to S^4$, the second one is quantum homogeneous spaces.

\subsection*{Acknowledgements} Xiao Han was supported by the European Union's Horizon 2020 research and innovation program under the Marie Sklodowska-Curie grant agreement No 101027463 and Leverhulme Trust project grant RPG-2024-177. This article is also based upon work from COST Action CaLISTA CA21109 supported by COST (European Cooperation in Science and Technology). Xiao Han is grateful to the Institut de Mathematiques de Bourgogne
(Dijon) for hospitality.

\section{Preliminaries} \label{sec2}

In this section, we will recall some definitions and notation. Let $B$ be an unital algebra over a field $k$. We denote the opposite algebra by $\BB$ and let $B\to \BB$, $b\mapsto\Bar{b}$ for any $b\in B$ be the obvious $k$-algebra antiisomorphism. Define $B^{e}:=B\ot \BB$, so $B$ and $\BB$ are obvious subalgebras of $B^{e}$. Let $M, N$ be $B^{e}$-bimodules. We define
\begin{align*}
    M\di N:=\int_{b} {}_{\Bar{b}}M\ot {}_{b}N:=&M\ot N/\langle \Bar{b}m\ot n-m\ot bn|b\in B, m\in M, n\in N\rangle\\
    M\ot_{B} N:=\int_{b} M_{b}\ot {}_{b}N:=&M\ot N/\langle mb\ot n-m\ot bn|b\in B, m\in M, n\in N\rangle\\
    M\ot_{\BB} N:=\int_{b} M_{\Bar{b}}\ot {}_{\Bar{b}}N:=&M\ot N/\langle m\Bar{b}\ot n-m\ot \Bar{b}n|b\in B, m\in M, n\in N\rangle\\
\end{align*}
For convenience, we also define $N\ot^{B}M=\int_{b} {}_{b}N\ot M_{b}$ and  $N\ot^{\BB}M=\int_{b} {}_{\Bar{b}}N\ot M_{\Bar{b}}$. Moreover, we define 
\begin{align*}
    \int^{b}M_{\Bar{b}}\ot N_{b}:=\left\{\left.\sum_{i}m_{i}\ot n_{i}\in M\ot N\right|\sum m_{i}\Bar{b}\ot n_{i}=\sum m_{i}\ot n_{i}b, \forall b\in B\right\}.
\end{align*}
Although a subspace in a tensor product like the one above is usually not spanned by simple tensors (so that the element in question was written as a sum above), we will usually abuse notations and write such an element in the form $m\ot n$ just the same, although it is not the simple tensor coming from single elements $m$ and $n$.

The symbols $\int^{b}$ and $\int^{c}$ commute, also $\int_{b}$ and $\int_{c}$ commute. However, in general, the symbols $\int^{b}$ and $\int_{c}$ do not commute. For any $\BB$-bimodule $M$ and any $B$-bimodule $N$, we also define
\begin{align*}
    M\times_{B}N:=\int^{a}\int_{b} {_{\Bar{b}}M_{\Bar{a}}}\ot {_bN_a}.
\end{align*}
$M\times_{B}N$ is called Takeuchi product of $M$ and $N$. If $P$ is a $B^{e}$-bimodule, then $P\times_{B}N$ is a $B$-bimodule with $B$ acting on $P$. Similarly, $M\times_{B}P$ is a $\BB$-bimodule with $\BB$ acting on $P$. If both $M$ and $N$ are $B^{e}$-bimodule, then $M\times N$ is also a $B^{e}$-bimodule. However, the product $\times_{B}$ is neither associative nor unital on the category of $B^{e}$-bimodules. For any $M,N, P\in {}_{B^{e}}\M_{B^{e}}$, we can define
\begin{align*}
    M\times_{B}P\times_{B}N:=\int^{a,b}\int_{c,d} {}_{\Bar{c}}M_{\Bar{a}}\ot {}_{c,\Bar{d}}P_{a, \Bar{b}}\ot {}_{d}N_{b},
\end{align*}
where $\int^{a,b}:=\int^{a}\int^{b}$ and $\int_{c,d}:=\int_{c}\int_{d}$. There are maps
\begin{align*}
    &\alpha:(M\times_{B}P)\times_{B} N\to M\times_{B}P\times_{B}N,\quad m\ot p\ot n\mapsto m\ot p\ot n,\\
    &\alpha':M\times_{B}(P\times_B N)\to M\times_{B}P\times_{B}N,\quad m\ot p\ot n\mapsto m\ot p\ot n.\\
\end{align*}
Notice that neither $\alpha$ nor $\alpha'$ are isomorphisms  in general. We note that $\alpha$ is an isomorphism if $N$ is flat as a left $B$-module, and $\alpha'$ is an isomorphism if $M$ is flat as a left $\ol B$-module. In many instances in the paper, some of the spaces we use in Takeuchi products will be faithfully flat as left modules over $B$ and $\ol B$ (either by assumption or because we can prove that this is the case), which makes iterated Takeuchi products somewhat well behaved.

\subsection{Hopf algebroids}

Here we recall the definitions (cf. \cite{BW}, \cite{Boehm}). Let $B$ be a unital algebra over a field $k$.
A {\em $B$-ring}   means a unital algebra in the monoidal category ${}_B\CM_B$ of $B$-bimodules. Likewise,  a  {\em $B$-coring} is a coalgebra in ${}_B\CM_B$. Morphisms of $B$-(co)rings are defined to be morphisms of (co)algebras, but in the category ${}_B\CM_B$.

Specifying a unital $B$-ring $\cL$ is equivalent to specifying a unital algebra $\cL$ (over $k$) and an algebra map $\eta:B\to \cL$. Left and right multiplication in $\cL$ pull back to left and right $B$-actions as a bimodule (so $bXc=\eta(b)X\eta(c)$ for all $b,c\in B$ and $X\in \cL$) and the product descends to the product $\mu_B:\cL\tens_B\cL\to \cL$ with $\eta$ the unit map. Conversely, given
$\mu_B$ we can pull back to an associative product on $\cL$ with unit $\eta(1)$.

Now suppose that $s:B\to \cL$ and $t:\BB\to \cL$ are algebra maps with images that commute. Then $\eta(b\tens c)=s(b)t(c)$ is an algebra map $\eta: B^e\to \cL$, where $B^e=B\tens \BB$, and is equivalent to making $\cL$ a $B^e$-ring. The left $B^e$-action part of this is equivalent to a $B$-bimodule structure
\begin{equation}\label{eq:rbgd.bimod}
b.X.c=b\Bar{c}X:= s(b) t(c)X
\end{equation}
for all $b,c\in B$ and $X\in \cL$. Similarly, the right $B^{e}$-action part of this is equivalent to another $B$-bimodule structure
\begin{equation}\label{eq:rbgd.bimod1}
c^{.}X^{.}b=Xb\Bar{c}:= Xs(b) t(c).
\end{equation}

If $\cL$ and $\mathcal R$ are two $B^e$-rings, then the Takeuchi product $\cL\times_B\mathcal R$ is an algebra with multiplication given by $(\ell\otimes r)(\ell'\otimes r')=\ell\ell' \otimes rr'$ for $\ell\ot r,\ell'\ot r'\in \cL\times_B\mathcal R$.

\begin{defi}\label{def:left.bgd} Let $B$ be a unital algebra. A left $B$-bialgebroid (or left bialgebroid over $B$) is an algebra $\cL$ and (`source' and `target') commuting algebra maps $s:B\to \cL$ and $t:\BB\to \cL$ (thus making $\cL$ a $B^e$-ring), such that
\begin{itemize}
\item[(i)]  There are two left $B^{e}$-module maps, the coproduct $\Delta:\cL\to \cL\times_{B} \cL$ and counit $\varepsilon:\cL\to B$ satisfying
\begin{align*}
&\alpha\circ(\Delta\times_{B}\id)\circ\Delta=\alpha'\circ(\id\times_{B}\Delta)\circ\Delta:\cL\to \cL\times_{B}\cL\times_{B}\cL\\    &\varepsilon(X\o)X\t=\overline{\varepsilon(X\t)}X\o=X,
\end{align*}
for any $X\in\cL$, where we use the sumless Sweedler index to denote the coproduct, namely, $\Delta(X)=X\o\ot X\t$.
\item[(ii)] The coproduct is an algebra map. The counit $\varepsilon$ is a left character in the following sense:
\begin{equation*}\varepsilon(1_{\cL})=1_{B}, \quad \varepsilon(X\varepsilon(Y))=\varepsilon(XY)=\varepsilon(X\overline{\varepsilon(Y)})\end{equation*}
for all $X,Y\in \cL$ and $a\in B$.
\end{itemize}
\end{defi}
\begin{rem}
   Condition (i) implies $\cL$ is a $B$-coring with the $B$-bimodule structure given by \eqref{eq:rbgd.bimod}. By condition (ii), we can define an algebra map $\hat{\varepsilon}: \cL\to \operatorname{End}(B)$ by $\hat{\varepsilon}(X)(b):=\varepsilon(X b)$. Moreover, the image of $\alpha\circ(\Delta\times_{B}\id)\circ\Delta=\alpha'\circ(\id\times_{B}\Delta)\circ\Delta$ could be also denoted by the Sweedler index, namely, $\alpha\circ(\Delta\times_{B}\id)\circ\Delta(X)=\alpha'\circ(\id\times_{B}\Delta)\circ\Delta(X)=X\o\ot X\t\ot X\th\in \cL\times_{B}\cL\times_{B}\cL$. Unlike the bialgebra case, the elements $X\o\o\ot X\o\t\ot X\t$, $X\o\ot X\t\o\ot X\t\t$ and $X\o\ot X\t\ot X\th$ are not the same, since they belong to three different vector spaces, even if in many interesting cases both $\alpha$ and $\alpha'$ are isomorphisms. 
   Morphisms between left $B$-bialgebroids are $B$-coring maps which are also $B^e$-ring map maps. 
\end{rem}

\begin{defi}\label{defHopf}
A left bialgebroid $\cL$ over $B$ is a left Hopf algebroid (\cite{schau1}, Thm and Def 3.5.) if
\[\lambda: \cL\ot_{\BB}\cL\to \cL\di\cL,\quad
    \lambda(X\ot Y)=\one{X}\ot \two{X}Y\]
is invertible.\\
Similarly, A left bialgebroid $\cL$ over $B$ is a anti-left Hopf algebroid if 
\[\mu: \cL\ot_{B}\cL\to \cL\di \cL,\quad
    \mu(X\ot Y)=\one{X}Y\ot \two{X}\]
is invertible. A left $B$-bialgebroid is a $B$-Hopf algebroid if it is a left Hopf algebroid and anti-left Hopf algebroid.
\end{defi}
In the following, we will always use the balanced tensor product explained as above. If $B=k$ then this reduces to the map $\cL\tens\cL\to \cL\tens\cL$ given by $h\tens g\mapsto h\o\tens h\t g$ which for a usual bialgebra has an inverse, namely $h\tens g\mapsto h\o\tens (Sh\t)g$ if and only if there is an antipode. We adopt the shorthand
\begin{equation}\label{X+-} X_{+}\ot_{\BB}X_{-}:=\lambda^{-1}(X\di 1)\end{equation}
\begin{equation}\label{X[+][-]} X_{[+]}\ot_{B}X_{[-]}:=\mu^{-1}(1\di X).\end{equation}
So for a Hopf algebra $H$, $h_{+}\ot h_{-}=h\o\ot S(h\t)$ and $h_{[-]}\ot h_{[+]}= S^{-1}(h\o)\ot h\t$ for any $h\in H$.
We recall from \cite[Prop.~3.7]{schau1} that for a left Hopf algebroid, and any $X, Y\in \cL$ and $a,a',b,b'\in B$,
\begin{align}
    \one{X_{+}}\di{}\two{X_{+}}X_{-}&=X\di{}1\label{equ. inverse lambda 1};\\
    \one{X}{}_{+}\ot_{\BB}\one{X}{}_{-}\two{X}&=X\ot_{\BB}1\label{equ. inverse lambda 2};\\
    (XY)_{+}\ot_{\BB}(XY)_{-}&=X_{+}Y_{+}\ot_{\BB}Y_{-}X_{-}\label{equ. inverse lambda 3};\\
    1_{+}\ot_{\BB}1_{-}&=1\ot_{\BB}1\label{equ. inverse lambda 4};\\
    \one{X_{+}}\di{}\two{X_{+}}\ot_{\BB}X_{-}&=\one{X}\di{}\two{X}{}_{+}\ot_{\BB}\two{X}{}_{-}\label{equ. inverse lambda 5};\\
    X_{+}\ot\one{X_{-}}\ot{}\two{X_{-}}&=X_{++}\ot X_{-}\ot{}X_{+-}\in\int^{a,b}\int_{c,d}{}_{\Bar{a}}\cL_{\Bar{c}}\ot {}_{\Bar{d}}\cL_{\Bar{b}}\ot {}_{\Bar{c},d}\cL_{b,\Bar{a}}
    \label{equ. inverse lambda 6};\\
    X&=X_{+}\overline{\varepsilon(X_{-})}\label{equ. inverse lambda 7};\\
    X_{+}X_{-}&=\varepsilon(X)\label{equ. inverse lambda 8};\\
    (a\Bar{a'}Xb\Bar{b'})_{+}\ot_{\BB}(a\Bar{a'}Xb\Bar{b'})_{-}&=aX_{+}b\ot_{\BB}b'X_{-}a'\label{equ. inverse lambda 9};\\
    \Bar{b}X_{+}\ot_{\BB}X_{-}&=X_{+}\ot_{\BB}X_{-}\Bar{b}\label{equ. inverse lambda 10}.
\end{align}
Similarly, for an anti-left Hopf algebroid, from \cite{BS} we have
\begin{align}
    \one{X_{[+]}}X_{[-]}\di{}\two{X_{[+]}}&=1\di{}X\label{equ. inverse mu 1};\\
   \two{X}{}_{[+]}\ot_{B}\two{X}{}_{[-]}\one{X}&=X\ot_{B}1\label{equ. inverse mu 2};\\
    (XY)_{[+]}\ot_{B}(XY)_{[-]}&=X_{[+]}Y_{[+]}\ot_{B}Y_{[-]}X_{[-]}\label{equ. inverse mu 3};\\
    1_{[+]}\ot_{B}1_{[-]}&=1\ot_{B}1\label{equ. inverse mu 4};\\
    (\one{X_{[+]}}\di{}\two{X_{[+]}})\ot_{B}X_{[-]}&=(\one{X}{}_{[+]}\di{}\two{X})\ot_{B}\one{X}{}_{[-]}\label{equ. inverse mu 5};\\
    \one{X_{[-]}}\ot{}\two{X_{[-]}}\ot X_{[+]}&=X_{[+][-]}\ot{}X_{[-]}\ot X_{[+][+]}\in \int^{a,b}\int_{c,d}{}_{c,\Bar{d}}\cL_{a,\Bar{b}}\ot {}_{d}\cL_{b}\ot {}_{a}\cL_{c}\label{equ. inverse mu 6};\\
    X&=X_{[+]}\varepsilon(X_{[-]})\label{equ. inverse mu 7};\\
    X_{[+]}X_{[-]}&=\overline{\varepsilon(X)}\label{equ. inverse mu 8};\\
    (a\Bar{a'}Xb\Bar{b'})_{[+]}\ot_{B}(a\Bar{a'}Xb\Bar{b'})_{[-]}&=\Bar{a'}X_{[+]}\Bar{b'}\ot_{B}\Bar{b}X_{[-]}\Bar{a}\label{equ. inverse mu 9};\\
    X_{[+]}\ot_{B}X_{[-]}b&=bX_{[+]}\ot_{B}X_{[-]}\label{equ. inverse mu 10}.
\end{align}

\begin{prop}\label{prop. anti left and left Galois maps}
    If $\cL$ is a left and anti-left Hopf algebroid, then we have
    \begin{align*}
        X_{[+]}\ot X_{[-]+}\ot X_{[-]-}&=X\t{}_{[+]}\ot X\t{}_{[-]}\ot X\o\in \cL\ot_{B}\cL\ot_{\BB}\cL;\\
        X_{+}\ot X_{-[+]}\ot X_{-[-]}&=X\o{}_{+}\ot X\o{}_{-}\ot X\t\in \cL\ot_{\BB}\cL\ot_{B}\cL;\\
        X_{[+]+}\ot X_{[-]}\ot X_{[+]-}&=X_{+[+]}\ot X_{+[-]}\ot X_{-}\in \int^{c,d}\int_{a,b} {}_{c, \Bar{d}}\cL_{\Bar{a}, b}\ot {}_{b}\cL_{c}\ot {}_{\Bar{a}}\cL_{\Bar{d}},
    \end{align*}
   for any $X\in\cL$. 
\end{prop}
\begin{proof}
    The first equality can be proved by comparing the result of $\id_{\cL}\ot \lambda $ on both sides. The second 
 equality can be proved by comparing the result of $\id_{\cL}\ot \mu$ on both sides. The third 
 equality can be proved by comparing the result of $\mu\ot \id_{\cL}$ on both sides. 
\end{proof}

\begin{prop}\label{rightff implies leftff}
    Let $\CH$ be a $B$-Hopf algebroid. If $\CH$ is faithfully flat as a right $B$-module and as a right $\ol B$-module, then $\CH$ is faithfully flat as a left $B$-module and left $\ol B$-module.
\end{prop}
\begin{proof}
    By tensoring $\lambda\colon\CH\ou{\ol B}\CH\to\CH\di\CH$ with a left $\CH$-module $M$ we obtain a similar isomorphism 
    \[\lambda_M\colon\CH\ou{\ol B}M\to\CH\di M;\quad X\ot m\mapsto X\o\ot X\t m.\]
    It follows that the functor $\CH\di \leer$ is faithfully exact on the category of left $\CH$-modules if $\CH$ is faithfully flat as a right $\ol B$-module. Similarly, by
    \[\mu_M\colon\CH\ou{B}M\to M\di \CH;\quad X\ot m\mapsto X\o m\ot X\t ,\]
    $\leer\di\CH$ is faithfully exact on the category of left $\CH$-modules if $\CH$ is faithfully flat as a right $B$-module.

    Note that $\lambda\colon\CH\ou{\ol B}\CH\to\CH\di\CH$ is a $B$-module map between the structures given by left multiplication by $b\in B$ on the right tensor factor on the source, and right multiplication by $\ol b$ on the left tensor factor on the target, namely $\lambda(X\ot bY)=X\o \ot X\t bY=X\o\ol b\ot X\t Y.$
    Consider
    \begin{align*}
        F\colon\RMod B&\to\LMod{\ol B}&M&\mapsto M\ou B\CH\\
        G\colon\LMod{\ol B}&\to\vect&N&\mapsto \CH\ou{\ol B}N\\
        H\colon\RMod B=\LMod{\ol B}&\to 
\LMod\CH& M&\mapsto \CH \ou{\ol B}M \\
        J\colon\LMod\CH&\to \vect&\Gamma&\mapsto \Gamma\di\CH    .
\end{align*}
Then we have $GF=JH.$  Since $G,H,$ and $J$ are faithfully exact, so is $F$. This means that $\CH$ is left $B$-faithfully flat. Similarly $\CH$ is left $\ol B$-faithfully flat.
\end{proof}

\section{Hopf Galois extensions of Hopf algebroids}\label{sec:onesided}

In this section we will collect some results on (one-sided) Hopf-Galois extensions over Hopf algebroids, mostly in preparation for our work on biGalois theory. An important topic will be the question as to how one can switch between left and right Galois extensions. In the classical (Hopf algebra) theory this would simply be done by using the antipode. Here, this is not available, and also, the notion of a right Galois extension is not immediately well-defined.

We circumvent the latter problem by working with what we call "anti-right" Galois extensions. As to the problem of "switching sides" we lean heavily on an observation originating in \cite{CGK16}; the result we will use is in \cite{C20}. Namely, there is a way to switch between left and right comodule structures that are "regular" which is analogous to switching using a bijective antipode in the Hopf case. Moreover, under mild conditions, if our Hopf algebroid fulfills a condition analogous to having bijective antipode (namely, the bialgebroid and its coopposite need to be Hopf), then comodules are always regular.

\subsection{(Co)modules of left bialgebroids and regularity}

The following observation was made in \cite{schau1}; with some technical modification it in fact characterizes bialgebroids: 
\begin{rem}
     Let $\cL$ be a left bialgebroid over $B$. A left $\cL$-module $M$ is a $B$-bimodule $M$ by 
     \[bm:=s_{L}(b)\la m,\qquad mb:=t_{L}(b)\la m\]
     for $b\in B$  and $m\in M$. 

     Thus, we have a functor $\LMod \cL\to\Bimod{B}{B}$. The category $\LMod\cL$ is monoidal with tensor product of $M,N\in\LMod\cL$ given by $M\ou BN$ with module structure 
     \[ X\la\,(m\,n)=(X\o\la m)\ot (X\t\la n)\]
      for any $m\in M,n\in N$.
\end{rem}
Now the following definition of module algebra is natural.
\begin{defi}
     Let $\cL$ be a left bialgebroid over $B$. A left $\cL$-module algebra is a is an algebra in the monoidal category $\LMod\cL$. Thus, it is a $B$-ring with a $\cL$-module structure such that the $B$-bimodule structure of th $B$-ring coincides with that of the $\cL$-module, and 
     \[X\la b=\varepsilon(X\,b),\qquad X\la\,(p\,q)=(X\o\la p)(X\t\la q)\]
      for any $p, q\in P$.
\end{defi}
\emph{One} (see remark below) suitable definition of comodules over a bialgebroid is given in \cite{schau1}.
\begin{defi}
    Let $\cL$ be a left bialgebroid over $B$. A left $\cL$ comodule is a $B$-bimodule $\Gamma$, together with a $B$-bimodule map ${}_{L}\delta: \Gamma\to \cL\times_{B}\Gamma\subseteq\cL\di \Gamma$, written ${}_{L}\delta(p)=p\mo\ot p\z$ (${}_{L}\delta$ is a $B$-bimodule map in the sense that
${}_{L}\delta(bpb')=bp\mo b'\ot p\z$), such that
\begin{align*}
    &\alpha\circ (\Delta\times_{B}\id_{\Gamma})\circ {}_{L}\delta=\alpha'\circ(\id_{\cL}\times_{B}{}_{L}\delta)\circ {}_{L}\delta:\Gamma\to \cL\times_{B} \cL \times_{B} \Gamma, p\mapsto p\mt\ot p\mo\ot p\z\\ 
    &\varepsilon(p\mo)p\z=p,
\end{align*}
The category $^\cL\M$ of left comodules is a monoidal category under the tensor product over $B$ with the codiagonal comodule structure
\[\Gamma\ot_B\Lambda \ni p\ot q\mapsto p\mo q\mo\ot p\z \ot q\z\in \cL\diamond_B(\Gamma\ou B\Lambda)\]
A left comodule algebra over $\CL$ is by definition an algebra in the category of left $\CL$-comodules; that is, an algebra $P$ with a left $\CL$-comodule structure such that 
\[\delta(pq)=p\swm 1q\swm 1\ot p\sw 0q\sw 0\in \CL\di P.\]
As in the case of ordinary Hopf algebras this can also be read as $\delta$ being an algebra map $P\to \CL\times_BP$.

A right $\cL$-comodule is a $\BB$-bimodule $\Gamma$, together with a $\BB$-bimodule map $\delta_{L}: \Gamma\to \Gamma\times_{B}\cL\subseteq\Gamma\di \cL$, written $\delta_{L}(p)=p \z\di p\o$ ($\delta_{L}$ is a $\BB$-bimodule map in the sense that $\delta_{L}(\Bar{b}p\Bar{b'})=p\z\ot \Bar{b}p\o\Bar{b'}$), such that
    \begin{align*}
        &\alpha\circ (\delta_{L}\times_{B}\id_{\cL})\circ \delta_{L}=\alpha'\circ (\id_{\Gamma}\times_{B} \Delta)\circ \delta_{L}:\Gamma\to \Gamma\times_{B} \cL \times_{B} \cL, p\mapsto p\z\ot p\o\ot p\t\\ 
    &\overline{\varepsilon(p\o)}p\z=p,
    \end{align*}
for any $p\in \Gamma$.
\end{defi}
\begin{rem}\label{simpler comodule def}
    Since a left bialgebroid $\cL$ is in particular a coring over $B$, there is already a notion of comodule over $\cL$. By definition, such a comodule over the coring $\cL$ has just one $B$-module structure. As proved in [Section 1.4 of \cite{HHP}], the two notions coincide: A left comodule over $\cL$ in the sense of the definition above is quite obviously a left comodule over the coring $\cL$, but also given a comodule over the coring $\cL$, there is a unique second $B$-module structure which makes it a comodule in the bialgebroid sense. More precisely, given a left comodule $M$ of a coring $\cL$, the right $B$-module structure on $M$ is given by $mb:=\varepsilon(m\mo\,b)m\z$ for any $m\in M$. By direct computation one can check this is a right $B$-module. Moreover, we can also check the image of the coproduct belongs to Takeuchi product with respect to the right $B$-module structure. Indeed,
    \begin{align*}
        m\mo\ot m\z\,b=&m\mo\ot \varepsilon(m\z\mo\,b)m\z=m\mt\ot \varepsilon(m\mo\,b)m\z\\
        =&m\mt\overline{b}\ot \varepsilon(m\mo)m\z=m\o\overline{b}\ot m\z.
    \end{align*}
Also, the coaction is  $B$-linear by the calculation
\begin{align*}
    (mb)\mo\ot (mb)\z=&\varepsilon(m\mt b)m\mo\ot m\z=\varepsilon(m\mt )m\mo\, b\ot m\z\\
    =&m\mo\,b\ot m\z
\end{align*}
As a result, given a left $B$-bialgebroid $\cL$, left comodules of the coring $\cL$ are  the same as  left comodules of the left bialgebroid $\cL$.

Note that the category of left $\CL$-comodules is abelian if $\CL$ is a flat left $\ol B$-module (the condition is needed to define the comodule structure on a kernel). Similarly the category of right $\CL$-comodules is abelian if $\CL$ is a flat left $B$-module.
\end{rem}

\begin{rem}
For $\Gamma\in{^\cL\M}$ we have \[^{co\cL}\Gamma:=\{p\in\Gamma\,|\,\delta(p)=1\ot p\}\subseteq\{p\in\Gamma\,|\,bp=pb\,\forall b\in B\}\]. 
\end{rem}

\begin{rem}\label{rem:codiagonal_as_composition}
    For an ordinary Hopf algebra $H$, the codiagonal comodule structure on the tensor product of two (right) comodules $V,W$ can easily be written as a composition 
    $$V\ot W\xrightarrow{\delta\ot\delta} V\ot H\ot W\ot H\xrightarrow{\tau_{23}}V\ot W\ot H\ot H\xrightarrow{1\ot\nabla}V\ot W\ot H.$$
    We need a suitable analog for later use but due to the various relative tensor products the definitions are a bit more involved.

    For $\Gamma,\Lambda\in\RComod\CL$ we note that the map
    $\tau_{23}\colon (\Gamma\times_B\CL)\ou{\ol B}(\Lambda\diamond\CL)\to (\Gamma\ou{\ol B}\Lambda)\diamond (\CL\ou{B^e}\CL)$ induced by flipping the second and third tensor factors is indeed well-defined (for this both the Takeuchi product instead of the diamond product and the tensor product over $B^e$ instead of $B$ are essential). With this, we can write the comodule structure on the tensor product again as
    $$\Gamma\ou{\ol B}\Lambda\xrightarrow{\delta\ot\delta}(\Gamma\times_B\CL)\ou{\ol B}(\Lambda\diamond\CL)\xrightarrow{\tau_{23}}(\Gamma\ou{\ol B}\Lambda)\diamond (\CL\ou{B^e}\CL)\xrightarrow{\id\ot\nabla}(\Gamma\ou{\ol B}\Lambda)\diamond \CL .$$
    This can be done for left comodules in a similar fashion.
\end{rem}

\begin{rem}\label{antiremark} For a left $B$-bialgebroid $\cL$, the coopposite $\cL^\cop$ is a left $\ol B$-bialgebroid (while the opposite algebra of $\cL$ does not have a left bialgebroid structure in general). The monoidal category of right (left) comodules over $\cL$ is naturally isomorphic to the category of left (right) comodules over $\cL^\cop$. In this way any notion or fact on left coactions of a left bialgebroid gives rise to a corresponding notion with equivalent properties for right coactions, and vice versa. For example the category $\M^\cL$ of right comodules is monoidal with the tensor product of comodules taken over $\BB$. Note also that the category of left $\cL^\cop$-modules is monoidally  equivalent to the category of $\cL$-modules with opposite tensor product: $_{(\cL^\cop)}\M\cong(_\cL\M)^\sym$.
\end{rem}

\begin{defi}
    Let $\cL$ be a left bialgebroid over $B$,  A right $\cL$-comodule $\Gamma$ is a \emph{regular} right $\cL$-comodule [\cite{C20} Definition 3.0.1], if 
\[\psi: \Gamma\ot_{\BB} \cL\to \Gamma\di \cL, \quad \psi(p\ot_{\BB} X)=p\z \di p\o X\]
is invertible. We adopt the shorthand
\begin{align}
    p\rz\ot_{\BB} p\rmo=\psi^{-1}(p \di 1). 
\end{align}
so that $\psi\inv(p\di X)=p\rz\ot p\rmo X.$

We will call a left $\cL$-comodule regular if it is regular as a right $\cL^\cop$-comodule. 
\end{defi}
Thus a left comodule $\Gamma$ is regular if 
\begin{align}\label{equ. definition of phi}
   \phi: \cL\ot^{B} \Gamma\to \cL\di \Gamma, \quad \phi(X\ot p)=p\mo X\ot p\z 
\end{align}
is bijective. We write
\begin{align}
    p\ro\ot^{B} p\rz=\phi^{-1}(1\di p). 
\end{align}

\begin{rem}
   The definition of a regular right comodule is such that $\cL$ is  regular as a right comodule over itself if and only if it is a left Hopf algebroid. If $B=k$ and $\cL$ is a Hopf algebra, then  any right $\cL$-comodule $\Gamma$ is  regular and $\forall p\in \Gamma$ we have
   $p\rz\ot p\rmo=p\z\ot S(p\o)$. By results of Chemla this can be maintained in absence of an antipode in the Hopf algebroid case under mild assumptions on the module structures of $\cL$.
\end{rem}
In the following proposition, we will see the properties of the inverse of the map $\phi$ generalize the properties of the inverse of the anti-Galois map. 

\begin{prop}[\cite{C20} Proposition 3.0.6]
    Let $\Gamma$ be a regular right comodule of a left $B$-bialgebroid $\cL$, then we have
    \begin{align}
        \label{equ. regular map 1}
        p\rz\di{} p\rz\o p\rmo=&1\di{}p,\\
        \label{equ. regular map 2}
        p\z\rz\ot_{\BB} p\z\rmo p\o=&p\ot_{\BB}1,\\
        \label{equ. regular map 4}
        (\Bar{b'}p\Bar{b})\rz\ot_{\BB} (\Bar{b'}p\Bar{b})\rmo=&p\rz\ot_{\BB}bp\rmo b',\\
        \label{equ. regular map 5}
        \Bar{b}p\rz\ot_{\BB}p\rmo=&p\rz\ot_{\BB}p\rmo \Bar{b},
    \end{align}
    Moreover, $\Gamma$ is a left comodule of $\cL$, with $B$-bimodule structure  $b'pb:=\Bar{b}p\Bar{b'}$ and  coaction ${}_{L}\delta(p):=p\rmo\ot{}p\rz$. 
We will refer to this comodule structure as the \emph{reverse} of the given one. 
In addition, if $\Gamma$ is a regular right comodule algebra then
\begin{align}\label{equ. skew regular map 3}
    (pq)\rmo\ot (pq)\rz=q\rmo\,p\rmo\ot p\rz\,q\rz \in \cL\di \Gamma.
\end{align}
That is to say, $\Gamma^\op$ with the $B$-bimodule structure as above is a left $\CL$-comodule algebra.

More generally, taking reverse comodule structures defines a functor $\mathcal M^\CL\to{}^\CL\mathcal M$; tensor product in the source is over $\ol B$ and in the target over $B$, and the functor is anti-monoidal. In this way we have a functor $\mathcal M^\CL\to{}^\CL\mathcal M$ which is monoidal. That is to say, the reverse comodule structure of $\Gamma\ot\Lambda$ for $\Gamma,\Lambda\in\mathcal M^\CL$ is given by
 \[(p\ot q)\rmo\ot (p\ot q)\rz=q\rmo p\rmo\ot p\rz\ot q\rz.\]
    In particular if $\CL$ is a Hopf algebroid, its reverse left comodule structure is $X\rmo\ot X\rz = X_-\ot X_+$.
\end{prop}
\begin{rem}\label{bij between left and right regular com}
If we identify, as indicated, the $\BB$-bimodule structure of a right $\cL$-comodule with a $B$-bimodule structure, then \eqref{equ. regular map 2} can be written
\begin{equation}\label{symmetric reverse formula}
    p\z\rmo p\o\di p\z\rz =1\di p\in \cL\di\Gamma.
\end{equation}

 As usual, there is a version of the above for regular left comodules. In the form \eqref{equ. regular map 1}. \eqref{symmetric reverse formula} of the condition that defines a regular right comodule and its reverse, it is obvious %
that the same condition describes regularity of the left comodule, and its reverse would be the structure we started from. Thus we have a bijection between right regular comodule structures and left regular comodule structures on the same space, given by taking the respective reverse. 
\end{rem}

We only note the rules for passing ``from left to right'' for later use:

\begin{prop}\label{prop. properties of skew regular comodules}[\cite{C20} Proposition 3.0.6]
    Let $\Gamma$ be a regular left comodule of a left $B$-bialgebroid $\cL$. Then we have
    \begin{align}
        \label{equ. skew regular map 1}
        p\rz\mo p\ro\di p\rz\z=&1\di p,\\
        \label{equ. skew regular map 2}
        p\z\ro p\mo\ot^{B} p\z\rz=&1\ot^{B}p,\\
        \label{equ. skew regular map 4}
        (bpb')\ro\ot^{B} (bpb')\rz=&\Bar{b'}p\ro \Bar{b}\ot^{B} p\rz,\\
        \label{equ. skew regular map 5}
        p\ro b \ot^{B} p\rz=&p\ro\ot^{B} bp\rz 
    \end{align}
    Moreover, $\Gamma$ is a right comodule of $\cL$, with coaction $\delta_{L}(p):=p\rz\ot p\ro$,  for any $p\in \Gamma$ and $b, b'\in B$; we will call this the reverse of the given coaction.
\end{prop}

\begin{lem}\label{lem. anti-left Hopf gives skew regular}[\cite[Thm.3.0.4,Prop.3.0.6]{C20} ]
    If $\cL$ is a left $B$-Hopf algebroid such that $\cL$ is flat as a left $B$-module and left $\BB$-module, then any right $\cL$-comodule $\Gamma$ is  regular. 
    Moreover, we have
    \begin{align}\label{equ. regular map 7}
        p\rz\z\ot{}p\rz\o\ot p\rmo=p\z\ot{} p\o{}_{+}\ot p\o{}_{-}\in \int^{a,d}\int_{b,c}{}_{\Bar{b}}\Gamma_{\Bar{a}}\ot{}_{b,\Bar{d}}\cL_{a,\Bar{c}}\ot {}_{\Bar{c}}\cL_{\Bar{d}}, 
    \end{align}

    Similarly, if $\cL$ is left anti-Hopf with the same flatness conditions, then every left $\cL$-comodule is regular, and
    \begin{align}\label{equ. skew regular map 7}
        p\ro\ot^{B} p\rz\mo\di{} p\rz\z=p\mo{}_{[-]}\ot^{B} p\mo{}_{[+]}\di{} p\z\in \int^{a,d}\int_{b,c}{}_{b}\cL_{a}\ot{}_{a,\Bar{c}}\cL_{b,\Bar{d}}\ot {}_{c}\Gamma_{d},
    \end{align}
    for any $p\in\Gamma$.
\end{lem}

 We recall the definition of cotensor product:
Let $P$ be a right comodule of a left $B$-bialgebroid $\cL$ and $Q$ be a left comodule of $\cL$, then the cotensor product $P\Box^{\cL} Q$ is defined to be the 
\[P\Box^{\cL} Q:=\{p\ot q\in P\di{}Q\quad|\quad p\z\ot p\o\ot q=p\ot q\mo\ot q\z\}\subset P\times_BQ,\]
where the balanced tensor product $P\di{} Q$ above is induced by the $\BB$-bimodule of $P$ and the $B$-bimodule of $Q$, i.e. $\overline{b}p\di{}q=p\di{}bq$. One can check that $P\Box^\cL Q\subset P\times_BQ.$ Indeed, if $p\ot q\in P\Box^{\cL}Q$, then $p\Bar{b}\di{}q=p\di{}qb$, since 
\begin{align*}
\delta_{L}\di{}\id_{Q}(p\Bar{b}\di{}q)=&p\z\di{} p\o t_{L}(b)\di{} q=p\di{}q\mo t_{L}(b)\di{}q\z\\
=&p\di{}q\mo\di{}q\z b=\delta_{L}\di{}\id_{Q}(p\di{}qb).
\end{align*}

\begin{rem}
    If $U$ is a left $\cL$-comodule then $U\cong \cL\Box^{\cL} U$. Similarly,  if $V$ is a right $\cL$-comodule then $V\cong V\Box^{\cL}\cL$.
\end{rem}

The following identity of cotensor product and coinvariants in a tensor product is already in \cite{schneider}; the use of the antipode there is replaced by passing to reverse comodule structures here.
\begin{lem}\label{coinvariants vs cotensor product}
    Let $\cL$ be a left $B$-bialgebroid and $\Gamma,\Lambda\in{^\cL\M}$. If $\Gamma$ is regular, then
    \[^{co\cL}(\Gamma\ou B\Lambda)=\Gamma\Box^{\cL}\Lambda\]
    if we endow $\Gamma$ with the reverse right $\cL$-comodule structure as in \cref{prop. properties of skew regular comodules}.

   Similarly, for $P, Q\in \CM^{\cL}$, if $P$ is regular, we have
    \[(P\ot_{\BB}Q)^{co\cL}=Q\Box^{\cL}P,\]
    where $P$ is endowed with the reverse left $\cL$-comodule structure.
\end{lem}
\begin{proof}
    Let $p\ot q\in ^{co\cL}(\Gamma\ou B\Lambda)$. Then
    \begin{align*}
        (\phi\ot\id)(p\ot_{B} q\mo\di{} q\z)=p\z\ot p\mo q\mo\ot q\z=(\phi\ot\id)(p\rz\ot_{B} p\ro\di{} q).
    \end{align*}
    So we have
    \[p\rz\ot_{B} p\ro\di{} q=p\ot_{B} q\mo\di{} q\z,\]
    therefore, $p\ot q\in \Gamma\Box^{\cL}\Lambda$. Conversely, if $p\ot q\in \Gamma\Box^{\cL}\Lambda$, we have
    \begin{align*}
        \delta(p\ot q)=p\mo q\mo\ot p\z\ot q\z=p\rz\mo p\ro\ot p\rz\z\ot q=1\ot p\ot q.
    \end{align*}
\end{proof}

For later use we note an analog of \cite[Lemma 1.1]{U87}
\begin{lem}\label{lem:ulbrichanalog}
    Let $P$ be a right $\cL$-comodule and $V$ a regular left $\cL$-comodule. Then an isomorphism 
    \[P\Box^\cL(V\ou B\cL)\cong P\di V\]
    is given by applying $\varepsilon$ to the right factor, i.e. $p\ot v\ot X\mapsto p\ot v\varepsilon(X)$.
\end{lem}
\begin{proof}
    This is the isomorphism
    $$P^{\bullet}\Box^\cL({}^{\bullet}V\ou B{}^{\bullet}\cL)\cong P^{\bullet}\Box^\cL({}^{\bullet}\cL\di V)\cong P\di V$$
    in which the first stage is given by the map defining regularity, cotensored with $P$; the second maps $p\ot v\in P\di V$ to $p\sw 0\ot p\sw 1\ot v$, and applies $\varepsilon$ to the middle factor in the other direction. 
\end{proof}

\subsection{Hopf Galois extensions}

.

\begin{defi}
Given a left bialgebroid over $B$,    a left $\cL$-comodule algebra $P$ is a $B$-ring and a left $\cL$-comodule, such that the coaction is a $B$-ring map. Let $N={}^{co\cL}P$ be the left invariant subalgebra of $P$, $N\subseteq P$ is called a left $\cL$-Galois extension if the left canonical map $\lcan: P\ot_{N} P\to \cL\di{}P$ given by
    \[\lcan(p\ot_{N}q)=p\mo\di{}p\z q,\]
    is bijective. If $N\subseteq P$ is a
left $\cL$-Galois extension, the inverse of $\lcan$ is determined by the left translation map:
    \begin{align}
        \ltau:=\lcan^{-1}|_{\cL\di{}1}:\cL\to P\ot_{N}P, \qquad X\mapsto \tuno{X}\otimes_{N}\tdue{X}.
    \end{align}
If $P$ is faithfully flat as a right $N$-module we say it is a right faithfully flat extension, similarly for left faithful flatness. A faithfully flat extension will mean one that is both left and right faithfully flat.
\end{defi}
 We have $\lcan\inv(X\ot p)=\tuno{X}\otimes\tdue{X}p$.
We note that $\lcan$ is a left $B$-module map in the sense that \begin{equation}\label{Lcanlin}
       \lcan(p\ot bq)=p_{(-1)}\overline b\ot p_{(0)}q
    \end{equation}
    because the image of the coaction map lies in the Takeuchi product. In addition, the map is obviously right $P$-linear, and also left $B$-linear in a suitable sense, because the coaction is.

 \begin{rem}\label{different coinvariants}
    The Galois map is already defined in the same way if $M$ is just a subalgebra of the coinvariants. Assume that this version $\lcan_M$ of the Galois map is  a bijection. Then \emph{a fortiori} the Galois map $\lcan_{\hat M}$ with $M$ replaced by the coinvariant subalgebra $\hat M={}^{\co\CL}P$ is also a bijection. Indeed $\lcan_M\colon P\ou MP\to \CL\di P$ factors over the quotient $P\ou{\hat M}P$ of $P\ou MP$. If $\lcan_M$ is a bijection, then the two tensor products are therefore  the same and   $P$ is an $\CL$-Galois extension of its coinvariants.
    \end{rem}

In the theory of Hopf-Galois extensions over $k$-Hopf algebras, right comodule algebras with a Galois condition are more common. However, if one naively copies the definition of the canonoical map for a right comodule algebra, namely $p\ot q\mapsto pq\sw 0\ot p\sw 1$, one ends up with a ``map'' that is not well-defined. Already in \cite{schneider} a second canonical map is used; bijectivity of the two canonical maps is equivalent if the antipode of the Hopf algebra is bijective. We will use an analog of this second canonical map as a Galois condition for right comodule algebras.

\begin{defi}
Let $\cL$ be a left $B$-bialgebroid. A right $\cL$-comodule algebra $P$ is a $\BB$-ring and a right $\cL$-comodule, such that the coaction is a $\BB$-ring map. Let $M:=P{}^{co\cL}$ be the right invariant subalgebra of $P$. We say that  $M\subseteq P$ is an anti-right $\cL$-Galois extension if  the anti-right canonical map $\rcan: P\ot_{M} P\to P\di{}\cL$ given by

    \[\rcan(p\ot_{M}q)=p\z q\di{}p\o ,\]
    is bijective. If $M\subseteq P$ is a
anti-right $\cL$-Galois extension, the inverse of $\rcan$ is determined by the anti-right translation map:
    \begin{align}
        \rtau:=(\rcan)^{-1}|_{1\di{}\cL}:\cL\to P\ot_{M}P, \qquad X\mapsto \yi{X}\otimes_{M}\er{X},
    \end{align}
    by the formula
    \[\overset\rightharpoonup\can^{-1}(p\ot X)=X\yi{}\ot X\er{}p.\]

\end{defi}
\begin{rem}
   A right $\cL$-comodule algebra is the same as a left $\cL^\cop$-comodule algebra. Moreover, $M\subseteq P$ is an anti-right $\cL$-Galois extension if and only if $M={}^{co\cL^\cop}P\subseteq P$ is left $\cL^{\cop}$-Galois extension.
\end{rem}

The theory of right anti-Galois extensions is thus, in principle, entirely redundant. However, we will be looking at algebras that are naturally bi-comodule algebras over two Hopf algebroids, and it would be inconvenient to treat both structures as left comodule structures. In particular, a Hopf algebroid $\cL$ that is also anti-Hopf is an example of a ``biGalois object'' with its obvious left and right comodule structures over itself and it would be inconvenient to treat the right comodule structure as a left $\cL^\cop$-comodule algebra structure.

In the sequel we will write out some notions and facts on (Galois) left comodule algebras in an extra version for right comodule algebras to have both explicitly. In other instances we might only mention a result on one side, but will freely use it on the other, if this presents no extra notational traps.
As a generalization of \cite{schau2}, we have

\begin{lem}\label{lem. left Hopf Galois induce left Hopf}
    Let $\CL$ be a $B$-bialgebroid that  admits a left $\cL$-Galois extension $N\subseteq P$ such that $P$ is faithfully flat as left $B$-module.
    Then $\cL$ is a left Hopf algebroid. We have
    \[X_{+}\ot X_{-}\ot 1=\tuno{X}\mo\ot \tdue{X}\mo\ot \tuno{X}\z \tdue{X}\z\in \int_{a,b} {}_{\Bar{b}}\cL_{\Bar{a}}\ot{}_{\Bar{a}}\cL\ot {}_{b}P.\]
    
    If $P$ is (faithfully) flat as left $N$-module, then $\cL$ is (faithfully) flat as left $B$-module.

    If $P$ is (faithfully) flat as right $N$-module, then $\cL$ is (faithfully) flat as right $\BB$-module.

    If  $P$ is (faithfully) flat as right $B$-module  and (faithfully) flat as  left $N$-module, then $\cL$ is (faithfully) flat as  right $B$-module.
\end{lem}

\begin{proof}
        It is sufficient to  observe the following diagram  commutes:
     \[
\begin{tikzcd}
  &P\ot_{N} P\ot_{N}P \arrow[d, "P\ot \lcan"] \arrow[r, "\lcan\ot P"] & \cL\di{} P \ot_{N}P \arrow[dd, "\cL\ot \lcan"] &\\
  &P\ot_{N} (\cL \di{} P) \arrow[d, " \lcan_{1,3}"]\quad&\qquad&\\
   & \cL\ot_{\BB} \cL \di{}P   \arrow[r, "\lambda\ot P"] & \cL\di{} \cL \di{}P, &
\end{tikzcd}
\]
where the map $\lcan_{1,3}: P\ot_{N} (\cL \di{} P)\to \cL\ot_{\BB}\cL\di{}P$ is given by
\begin{align*}
    \lcan_{1,3}(p\ot X\ot q)=p\mo\ot X\ot p\z q.
\end{align*}
All maps in the diagram are well-defined as they can be viewed as tensor product of module maps with $P$ or $\cL$. For the map $\lcan_{1,3}$ in particular this is due to the linearity in \eqref{Lcanlin}, up to an identification. Now all the maps except for $\lambda\otimes P$ are isomorphisms by assumption, and so by faithful flatness of $P$ we can conclude that $\lambda$ is an isomorphism as well. 

Now also assume that $_NP$ is (faithfully) flat. Since
$(T\ou B\cL)\di P\cong T\ou B(\cL\di P)\cong T\ou BP\ou NP$ for $T\in\M_B$, we can conclude that $\cL$ is (faithfully) flat over $B$ on the left if $P$ is (faithfully) flat over $N$ on the left. 

Similarly, if $P$ is right $N$-(faithfully) flat, by using the fact that $\lcan(p\ot bq)=p\mo \overline{b}\ot p\z\,q$, we can derive $\cL$ is right $\BB$-(faithfully) flat.

 Finally, if $P$ is faithfully flat as right $B$-module and (faithfully) flat as a left $N$-module, by   using the fact that $\lcan(pb\ot q)=p\mo b\ot p\z\,q$, we can derive that $\cL$ is right $B$-faithfully flat.
\end{proof}

\begin{lem}\label{lem. left skew regular Hopf Galois induce anti-left Hopf}
    If a left $B$-bialgebroid $\cL$ admits a regular left $\cL$-Galois extension $N\subseteq P$ which is  faithfully flat as left $B$-module, then $\cL$ is an anti-left Hopf algebroid. More precisely,
    \[X_{[+]}\ot 1\ot X_{[-]}=\tuno{X}\rz\mo\ot \tuno{X}\rz\z\tdue{X}\ot \tuno{X}\ro \in \int_{a,b} {}_{\Bar{a}}\cL_{b}\ot {}_{a}P\ot {}_{b}P.\]
\end{lem}
\begin{proof}
    It is sufficient to observe the following diagram commutes:
     \[
\begin{tikzcd}
  &\cL\ot^{B} P\ot_{N}P \arrow[d, "\cL\ot \lcan"] \arrow[r, "\phi\ot P"] & \cL\di{} P \ot_{N}P \arrow[d, "\cL\ot \lcan"] &\\
   & \cL\ot^{B} \cL \di{}P   \arrow[r, "\mu\circ \textup{flip }\ot P"] & \cL\di{} \cL \di{}P, &
\end{tikzcd}
\]
it is not hard to see all the maps expect for $\mu\ot P$ are well defined isomorphism, so by the faithful flatness of $P$, $\mu$ is bijective. 
    
\end{proof}

The following proposition collects identities satisfied by the translation map of a Hopf Galois extension.
\begin{prop}Let $\CL$ be a Hopf algebroid over $B$ and 
     $N\subseteq P$  a left $\cL$-Galois extension. Then we have
\begin{align}\label{equ. translation map 1}
  \tuno{X}\mo \di \tuno{X}\z \ot_{N} \tdue{X} &= X\o\di{}\tuno{X\t}\ot_{N}  \tdue{X\t},\\
\label{equ. translation map 2}
~~ \tdue{X}\mo\ot{}\tuno{X}  \ot \tdue{X}\z &= X_{-}\ot{}\tuno{X_{+}}\ot\tdue{X_{+}}\in \int_{b}{}_{\Bar{b}}\cL\ot P\ot_{N} {}_{b}P,\\
\label{equ. translation map 3}
\tuno{X}\mo\di{}  \tuno{X}\z  \tdue{X} &= X\di{}1,\\
\label{equ. translation map 4}
    \tuno{p\mo}\ot_{N}\tdue{p\mo}p\z&=p\ot_{N}1,\\
\label{equ. translation map 4.5}
    nX\tuno{}\ot_{N}X\tdue{}=&
    X\tuno{}\ot_{N}X\tdue{}n,\\    
\label{equ. translation map 5}
\tuno{(aX b )}\ot_{N}\tdue{(aX b)}
&=a\tuno{X}b\ot_{N}\tdue{X},\\
\label{equ. translation map 6}
\tuno{(\Bar{a}X  \Bar{b})}\ot_{N}\tdue{(\Bar{a}X  \Bar{b})}
&=\tuno{X}\ot_{N}b\tdue{X}a,\\
\label{equ. translation map 6.5}
\tuno{X}\tdue{X}
&=\varepsilon(X),\\
\label{equ. translation map 7}
\tuno{(XY)}\ot_{N}\tdue{(XY)}&=\tuno{X}\tuno{Y}\ot_{N}\tdue{Y}\tdue{X},\\
\label{equ. translation map 7.5}
\tuno{X_{+}}\ot_{N} \tuno{X_{-}}\ot_{N} \tdue{X_{-}}\tdue{X_{+}}&=\tuno{X}\ot_{N}\tdue{X}\ot_{N} 1,
\end{align}
for any $X, Y\in \cL$, $p\in P$, $n\in N$ and $a, b\in B$. If $P$ is a regular comodule, then we have
\begin{align}
    \label{equ. skew translation map 1}
    p\rz \tuno{p\ro}\ot_{N} \tdue{p\ro}=&1\ot_{N}p,\\
    \label{equ. skew translation map 2}
    \tuno{X}\ot_{N}\tdue{X}\rz\di{}\tdue{X}\ro=&\tuno{X\o}\ot_{N}\tdue{X\o}\di{}X\t,\\
    \label{equ. skew translation map 3}
    \tuno{X}\rz\ot_{N}\tdue{X}\ot{} \tuno{X}\ro=&\tuno{X_{[+]}}\ot_{N}\tdue{X_{[+]}}\ot{} X_{[-]}\in\int_{b}P_{b}\ot_{N}P\ot {}_{b}\cL,\\
    \label{equ. skew translation map 4}
    \tuno{X_{[+]}} \tuno{X_{[-]}}\ot_{N} \tdue{X_{[-]}}\ot_{N} \tdue{X_{[+]}}=&1\ot_{N}\tuno{X}\ot_{N} \tdue{X}.
\end{align}
\end{prop}

\begin{proof}
Equations \eqref{equ. translation map 3} and \eqref{equ. translation map 4} are special cases of the definition of the translation map. Applying the canonical map to both sides of  \eqref{equ. translation map 4.5}, \eqref{equ. translation map 5} and \eqref{equ. translation map 6} gives the same result. We can show (\ref{equ. translation map 1}) by the fact that the canonical map $\lcan:{}^{\bullet}P\ot_{N} P\to {}^{\bullet}\cL\di P$ is left $\cL$-colinear.    To show \eqref{equ. translation map 2}, we observe that $P\ot_{N}P$ is a left $\cL$-comodule with the coaction given by
    \[{}_{L}\tilde{\delta}(p\ot_{N}q)=q\mo\di{}(p\ot_{N}q\z),\]
   with the underlying $B$-bimodule structure on $P\ot_{N}P$  given by $b.(p\ot_{N}{}q).b'=p\ot_{N}bqb',$ for any $b, b'\in B$ and $p, q\in P$.

   The tensor product of the reverse left comodule to $\CL$ with $P$ is $\CL\di P$ with the coaction 
   \[
   {}_{L}\delta(X\di{}p)=X_{-}p\mo\di{}(X_{+}\di{}p\z).
   \]

   Now $\lcan$ is a $B$-bimodule map from $P\ot_{N}P$ to $\cL\di{}P$ and moreover, it is a left $\cL$-comodule map. So the translation map is an $\cL$-colinear map as well and this proves \eqref{equ. translation map 2}. Equation \eqref{equ. translation map 7.5} follows from \eqref{equ. translation map 2} and \eqref{equ. translation map 4}. If $P$ is a regular left $\cL$-comodule, we can show \eqref{equ. skew translation map 1} by applying canonical map on both sides. Indeed,
   \begin{align*}
       \lcan(p\rz p\ro\tuno{}\ot p\ro\tdue{})=&p\rz\mo p\ro\tuno{}\mo\ot p\rz\z p\ro\tuno{}\z\, p\ro\tdue{}\\
      =&p\rz\mo p\ro\mo\ot p\rz\z\\
       =&p\rz\mo p\ro\ot p\rz\z\\
       =&p\ot 1
       =\lcan(1\ot p).
   \end{align*}

 We can show \eqref{equ. skew translation map 2} by comparing the results of the map $\id\ot \phi$  to both sides. We can show \eqref{equ. skew translation map 3} by comparing the results of the map $\lcan\ot{} \id$ on both sides. By applying $\id\ot \lcan$ to the left hand side of \eqref{equ. skew translation map 4}; we get 
   \begin{align*}
       \tuno{X_{[+]}} &\tuno{X_{[-]}}\ot \tdue{X_{[-]}}\mo\ot \tdue{X_{[-]}}\z\tdue{X_{[+]}}\\
       =&\tuno{X_{[+]}} \tuno{X_{[-]+}}\ot X_{[-]-}\ot \tdue{X_{[-]+}}\tdue{X_{[+]}}\\
       =&\tuno{X\t{}_{[+]}} \tuno{X\t{}_{[-]}}\ot X\o\ot \tdue{X\t{}_{[-]}}\tdue{X\t{}_{[+]}}\\
       =&1\ot X\ot 1,
   \end{align*}
   where the 1st step uses \eqref{equ. translation map 2} and the 2nd step uses \Cref{prop. anti left and left Galois maps}. Since we have the same result by applying to the right hand side, \eqref{equ. skew translation map 4} follows.
\end{proof}

The following is just the translation of the previous proposition for right anti-Galois extensions; we record it because we will use the notations.
\begin{prop}\label{prop. left Hopf Galois extension}
    Let $\CL$ be a Hopf algebroid over $B$ and $M\subseteq P$  an anti-right $\cL$-Galois extension. Then we have

\begin{align}\label{equ. anti translation map 1}
  \yi{X}\z \ot_{M} \er{X} \ot{} \yi{X}\o &= \yi{X\o}\ot_{M}\er{X\o}\ot{}X\t\in\int_{b}{}_{\Bar{b}}P\ot_{M}P\ot {}_{b}\cL,\\
\label{equ. anti translation map 2}
~~ \yi{X}\ot_{M}\er{X}\z\di{}\er{X}\o &= \yi{X_{[+]}}\ot_{M}\er{X_{[+]}}\di{}X_{[-]},\\
\label{equ. anti translation map 3}
\yi{X}\z \er{X}\di{}\yi{X}\o &= 1\di{}X,\\
\label{equ. anti translation map 4}
    \yi{p\o}\ot_{M}\er{p\o}p\z&=p\ot_{M}1,\\
 \label{equ. anti translation map 4.5}
    mX\yi{}\ot_{M}X\er{}=&
    X\yi{}\ot_{M}X\er{}m,\\   
\label{equ. anti translation map 5}
\yi{(aX b )}\ot_{M}\er{(aXb )}
&=\yi{X}\ot_{M}\Bar{b}\er{X}\Bar{a},\\
\label{equ. anti translation map 6}
\yi{(\Bar{a}X  \Bar{b})}\ot_{M}\er{(\Bar{a}X  \Bar{b})}
&=\Bar{a}\yi{X}\Bar{b}\ot_{M}\er{X},\\
\label{equ. anti translation map 6.5}
\yi{X}\er{X}
&=\overline{\varepsilon(X)},\\
\label{equ. anti translation map 7}
\yi{(XY)}\ot_{M}\er{(XY)}&=\yi{X}\yi{Y}\ot_{M}\er{Y}\er{X},\\
\label{equ. anti translation map 7.5}
\yi{X_{[+]}}\ot_{M} \yi{X_{[-]}}\ot_{M} \er{X_{[-]}}\er{X_{[+]}}&=\yi{X}\ot_{M}\er{X}\ot_{M} 1,
\end{align}
for any $X, Y\in \cL$, $p\in P$, $m\in M$ and $a, b\in B$. If $P$ is a regular right $\cL$-comodule, then we have
\begin{align}
    \label{equ. regular anti translation map 1}
    p\rz \yi{p\rmo}\ot_{M} \er{p\rmo}=&1\ot_{M}p,\\
    \label{equ. regular anti translation map 2}
    \yi{X}\ot_{M}\er{X}\rz\di \er{X}\rmo=&\yi{X\t}\ot_{M}\er{X\t}\di X\o,\\
    \label{equ. regular anti translation map 3}
    \yi{X}\rz\ot_{M}\er{X}\ot{} \yi{X}\rmo=&\yi{X_{+}}\ot_{M}\er{X_{+}}\ot{} X_{-}\in \int_{b}P_{\Bar{b}}\ot_{M}P\ot {}_{\Bar{b}}\cL,\\
    \label{equ. regular anti translation map 4}
    \yi{X_{+}} \yi{X_{-}}\ot \er{X_{-}}\ot \er{X_{+}}=&1\ot\yi{X}\ot \er{X}.
\end{align}
\end{prop}

By \Cref{bij between left and right regular com} we have a bijection between regular left and right coactions. We now show that this restricts to Galois comodule algebra structures:

\begin{lem}\label{lem. P opposite has Galois structure}
    Let $\cL$ be a left $B$-bialgebroid and $P$ a regular left $\cL$-comodule algebra. If $N\subseteq P$ is a left $\cL$-Galois extension, then $N^{op}\subseteq P^{op}$ is a regular anti-right $\cL$-Galois extension associated with the regular right coaction $p\mapsto p\rz\di{}p\ro$. The anti-right translation map is
    \[\rtau:X\mapsto X\tdue{}\ot_{N^{op}}X\tuno{},\]
    for any $p\in P$ and $X\in \cL$. In addition, if $P$ is right $B$ and right $N$ faithfully flat on the right, then $\CL$ is left $\BB$-faithfully flat.
    
    Similarly, if $P$ is a regular right comodule algebra and $M\subset P$ is a regular anti-right $\cL$-Galois extension, then $M^{op}\subseteq P^{op}$ is a left Hopf Galois extension of $\cL$. 
    The left translation map is
    \[\ltau: X\to X\er{}\ot_{M^{op}}X\yi,\]
    for any $p\in P$ and $X\in \cL$.
\end{lem}

\begin{proof}
    We only show the first half of the lemma. We already proved that $P^\op$ is a right comodule algebra, see \eqref{equ. skew regular map 3}.
    The right coinvariant subalgebra of $P^\op$ under the reverse comodule structure is $N^{op}$. Indeed, if $n\in N^{op}$, then 
    \begin{align*}
        n\rz\di{}n\ro=n\z\rz\di{} n\z\ro n\mo=n\di{} 1,
    \end{align*}
    and conversely, if $m\in (P^{op})^{co\cL}$, then
    \begin{align*}
        m\mo\di{}m\z=m\rz\mo m\ro\di{}m\rz\z=1\di{}m.
    \end{align*}
    By \eqref{equ. skew translation map 1} and \eqref{equ. skew translation map 2}, we see that $\rtau$ is the anti-right translation map. The reverse of right comodule structure is given by the original left $\cL$-coaction.

   The assertion that $\CL$ is left $\ol B$-faithfullly flat is now an application of the coopposite version of \Cref{lem. left Hopf Galois induce left Hopf}.
\end{proof}

Recall that for a left comodule algebra $P$ of a left $B$-bialgebroid $\cL$, a left-left (resp. left-right) relative Hopf module $M\in {}^{\cL}_{P}\M$ (resp. ${}^{\cL}\M_{P}$) is a left (resp. right) $P$-module in the category of left $\cL$-comodules. We  have the following generalizeation of Schneider's structure theorem  \cite{schneider}:

\begin{thm}\label{thm. fundamental theorem for left Hopf Galois extensions}
Let $\cL$ be a left $B$-bialgebroid and $N\subseteq P$ be a left $\cL$-Galois extension such that $P$ is faithfully flat as right $N$-module. Then
\begin{align*}    {}^{\cL}_{P}\M\to {}_{N}\M,\qquad& M\mapsto {}^{co\cL}M\\
    {}_{N}\M\to {}^{\cL}_{P}\M,\qquad& \Lambda\mapsto P\ot_{N}\Lambda
\end{align*}
   are quasi-inverse category equivalences; here, the $\cL$-comodule and $P$-module structure on $P\ot_{N}\Lambda$ is given by the structure on $P$.  
\end{thm}   

\begin{proof} Possibly this part of Schneider's theory \cite{schneider} can be viewed as part of the theory of Galois corings \cite{BW}, but we copy almost verbatim the approach in \cite{schau7}. Namely, by
\[\xymatrix{M\ar[r]^-\delta\ar[d]_\theta&L\diamond_RM\ar[r]^\cong &(L\diamond_RP)\ot_PM\\
P\ou N M \ar[rr]^\cong &&(P\ot_NP)\ot_PM\ar[u]_{\lcan\ot_PM}}\]
there is a bijection between $\cL$-comodule structures making $M$ a Hopf module, and descent data $\theta\colon M\to P\ou N M$, i.e. sections of the module structure making 
\[\xymatrix{M\ar[rr]^\theta\ar[d]_\theta&&P\ot_NM\ar[d]^{P\ot_N\theta}
\\P\ot_NM\ar[rr]^-{P\ot_N\iota\ot_NM}&&P\ot_NP\ot_NM}\]
commute. Faithfully flat descent says that the category of $P$-modules equipped with such $\theta$ is equivalent to the category of $N$-modules if $P$ is right faithfully flat over $N$.

 For explicit calculations we note that the isomorphism $\Lambda\simeq {}^{co\cL}(P\ot_{N}\Lambda)$ for any $\Lambda\in {}_{N}\M$ can be given by $\eta\mapsto 1\ot \eta$. The isomorphism $M\simeq P\ot_{N}{}^{co\cL}M$ for any $M\in {}^{\cL}_{P}\M$ can be given by $m\mapsto \tuno{m\mo}\ot_{N}\tdue{m\mo} m\z$ with inverse $p\ot_{N}\eta\mapsto p \eta$.
\end{proof}
Applying the various ways to pass between left and right comodule structures we obtain three more versions of the Hopf module structure theorem. None of them needs an extra proof: A right anti-Galois $\cL$-extension can be viewed as a left Galois $\cL^\cop$-extension. If a left (right) Galois extension is a regular comodule, it gives rise to a right (left) Galois extension. We list the three versions because we will use them explicitly.
\begin{cor}\label{schneider hopf module corollaries}
    Let $\cL$ be a bialgebroid over $B$.
    \begin{enumerate}
        \item If  $N\subset P$ is a right faithfully flat right anti-right $\cL$ Galois extension then 
        \begin{align*}
            _N\M&\to{_P\M^\cL}\\
            M&\mapsto P\ou NM\\
            \Gamma^{\operatorname{co}\cL}&\mapsfrom \Gamma
        \end{align*}
        is a category equivalence. We note the isomorphism
        $$\Gamma\to P\ot_N\Gamma^{\operatorname{co}\cL};\quad p\mapsto \yi{p\o}\ot_{N}\er{p\o}p\z$$
        \item If $N\subset P$ is a left faithfullly flat left $\cL$-Galois extension and a regular comodule,  then 
         \begin{align*}
            \M_N &\to{^\cL\M_P}\\
           M& \mapsto M\ot_{N}P\\
            ^{\operatorname{co}\cL}\Gamma&\mapsfrom \Gamma
        \end{align*}
        is a category equivalence. We note the isomorphism
        $$ \Gamma\to {^{co\cL}}\Gamma\ou NP;\quad p\mapsto p\rz \tuno{p\ro}\ot_{N}\tdue{p\ro}$$
        \item If $N\subset P$ is a left faithfullly flat right $\cL$-anti-Galois extension and a regular comodule,  then 
         \begin{align*}
            \M_N &\to{\M_P^\cL}\\
            M&\mapsto M\ot_NP\\
            \Gamma^{\operatorname{co}\cL}&\mapsfrom \Gamma
        \end{align*}
        is a category equivalence. We note the isomorphism
        $$ \Gamma\to \Gamma^{co\cL}\ou NP;\quad p\mapsto p\rz \yi{p\rmo}\ot_{N}\er{p\rmo}$$
    \end{enumerate}
\end{cor}

\begin{cor}\label{purity} Let $P$ be a right faithfully flat left $\cL$-Galois extension of $N$, and consider $\Gamma\in{^\cL_P\M}$ endowed with a right $S$-module structure such that the action of each $s\in S$ is a Hopf module map. Then the obvious map $^{co\cL}\Gamma\ou ST\to{^{co\cL}(\Gamma\ou ST)}$ is an isomorphism for every $T\in{_S\M}$. %
\end{cor}
\begin{proof}
    The equivalence from \Cref{thm. fundamental theorem for left Hopf Galois extensions} sends the map to the identity on $\Gamma\ou ST.$
\end{proof}

The following theorem generalizes results of Ulbrich \cite{U87}, \cite{U89} for the Hopf algebra case. To prepare the proof we give an analog of \cite[Lemma 1.3]{U87}
\begin{lem}\label{lem:aux-iso}
    Let $N\subset P$ be a left faithfully flat left $\cL$-Galois extension. Then for any right $\cL$-comodule $V$ we have
    \begin{equation}\label{eq:aux-iso}
        f\colon(V\Box^\cL P)\ou NP\cong V\di P;\quad v\ot p\ot q\mapsto v\ot pq.
    \end{equation}
    and $f(v\ot p\ot bq)=f(v\ot p\ot q)\Bar{b}$, where the module structure on the target is on the left tensor factor. 
\end{lem}
\begin{proof}
 We have
 \begin{equation*}
     (V\Box^\cL P)\ou NP\cong V\Box^{\cL}(P\ou NP)\cong V\Box^\cL(\cL\di P)\cong V\di P.
 \end{equation*}
 The first isomorphism stems from left faithful flatness of $P$, the second uses the Galois map, and the last is the canonical isomorphism applying $\varepsilon$ to the middle factor. It is straightforward to check that the composition has the stated form. The additional $B$-linearity holds because the cotensor product is contained in the Takeuchi product.
\end{proof}

\begin{rem}\label{exactness properties of cotensor product}Let $\CL$ be a $B$-Hopf algebroid that is faithfully flat as left $B$-module. As a consequence of the Lemma, if ${}_{N}P$ is a faithful flat left $\cL$-Galois extension, then $\Box^\CL P$ is right exact. 
    In addition, if $P$ is  flat over $B$ on the left, then $\Box^\CL P$ is  exact, and if $P$ is left faithfully flat over $B$  then $\Box^\CL P$ is faithfully exact.
    
    Now, consider $N\subseteq P$ is a left faithfully flat anti-right $\cL$-Galois extension, we note that applying \Cref{lem:aux-iso} to $P$ as a left $\cL^\cop$-comodule algebra 
gives the isomorphism 
\begin{equation*}
    f=f_V\colon (P\Box^\cL V)\ou NP \to P\di V;p\ot v\ot q\mapsto pq\ot v
\end{equation*}
We have
$$f(p\ot v\ot \ol bq)=p\ol bq\ot v=pq\ot vb=f(p\ot v\ot q)b$$ 
because the cotensor product is contained in the Takeuchi product. It is not hard to see $f^{-1}_{V}(p\ot v)=v\mo\yi{}\ot v\z\ot v\mo\er{} p$. Indeed, $f_{V}\circ f^{-1}_{V}(p\ot v)=v\mo\yi{}v\mo\er{} p\ot v\z=p\ot v$. Also,
\begin{align*}
   f^{-1}_{V}\circ f_{V}(p\ot v\ot q)=& v\mo\yi{}\ot v\z\ot v\mo\er{} pq
   =p\o\yi{}\ot v\ot p\o\er{} p\z\,q\\
   =&p\ot v\ot q.
\end{align*}
Similarly,  if $\CL$ is left faithfully flat over $\ol B$ and  ${}_{N}P$ is faithfully flat, then $ P\Box^\CL$ is right exact. 
    In addition, if $P$ is  flat over $\BB$ on the left, then $P\Box^\CL$ is  exact, and if $P$ is left faithfully flat over $\BB$  then $P\Box^\CL$ is faithfully exact.
\end{rem}

\begin{thm}\label{cotensor-monoidal}
    Let $\cL$ be a $B$-Hopf algebroid, which is left flat over $B$ and $\BB$. If $N\subseteq P$ is a left faithfully flat  anti-right $\cL$-Galois extension then the functor 
    \begin{align*}
       {}^{\cL}\M&\to\BiMod{N}\\
        V&\mapsto  P\Box^\cL V
    \end{align*}
    is monoidal with monoidal functor structure 
    \begin{align*}
       \xi: (P\Box^\cL V)\ou N(P\Box^\cL W)&\to P\Box^\cL(V\ot_{B}W)\\
       \xi: (p\diamond_{B} v)\ot (q\diamond_{B} w)&\mapsto pq\di (v\ot_{B}w).
    \end{align*} 
    Similarly, if $P$ is a left faithfully flat left $\cL$-Galois extension of $N$ then the functor 
    \begin{align*}
        \RComod\cL&\to\BiMod{N}\\
        V&\mapsto  V\Box^\cL P
    \end{align*}
    is monoidal with monoidal functor structure 
    \begin{align*}
       \varsigma: (V\Box^\cL P)\ou N(W\Box^\cL P)&\to (V\ot_{\BB}W)\Box^\cL P\\
       \varsigma: (v\diamond_{B} p)\ot (w\diamond_{B} q)&\mapsto (v\ot_{\BB} w)\diamond_{B} pq.
    \end{align*}
\end{thm}
\begin{proof}
The map $\xi$ is well defined since $P\Box^{\cL}V$ is contained in the Takeuchi product. For bijectivity, the diagram
\[\xymatrix{(P\Box^\cL V)\ou N(P\Box^\cL W)\ou NP
\ar[rr]^-{\xi\ou NP}
\ar[d]_{(P\Box^\cL V)\ou Nf_W}
&&(P\Box^\cL(V\ou BW))\ou NP\ar[d]^{f_{V\ou BW}}\\
(P\Box^\cL V)\ou NP\di W
\ar[rr]^-{f_V\ot W}&&P\di V\ou BW}\]
commutes, proving that $\xi$ is an isomorphism, where $f_{V}, f_{W}$ and $f_{V\ot_{B} W}$ are given in \Cref{exactness properties of cotensor product}.
We can also compute the inverse of $\xi$  which is given by
\[\xi^{-1}(p\diamond_{B}(v\ot_B w))=(v\mo\yi{}\diamond_{B}v\z)\ot_N (v\mo\er{} p\diamond_{B} w),\]
we can see $\xi^{-1}$ factors through all the balanced tensor products by applying \eqref{equ. anti translation map 5} and \eqref{equ. anti translation map 6}. Indeed,
\begin{align*}
    (\id_{P\Box^{\cL}V}&\ot f^{-1}_{W})\circ\,(f^{-1}_{V}\ot \id_{W})\circ\,f_{V\ot_{B}W}(p\ot v\ot w\ot 1)\\
    =& (\id_{P\Box^{\cL}V}\ot f^{-1}_{W})\circ\,(f^{-1}_{V}\ot \id_{W})(p\ot v\ot w)\\
     =& (\id_{P\Box^{\cL}V}\ot f^{-1}_{W})(v\mo\yi{}\ot v\z\ot v\mo\er{}p\ot w)\\
     =&v\mo\yi{}\ot v\z\ot w\mo\yi{}\ot w\z\ot_{N} w\mo\er{}v\mo\er{}p\in (P\Box^{\cL} V)\ot_{N}(P\Box^{\cL} V)\ot_{N}N\\
     =&v\mo\yi{}\ot v\z\ot w\mo\yi{}w\mo\er{}v\mo\er{}p\ot w\z\ot_{N} 1\\
     =&v\mo\yi{}\ot v\z\ot v\mo\er{}p\ot w\ot_{N} 1.
\end{align*}

   Similarly,  $\varsigma$ is well defined since $V\Box^{\cL}P$ is a subalgebra of $V\times_{B} P$. The inverse of $\varsigma$ is given by
    \[\varsigma^{-1}((v\ot_{\BB} w)\diamond_{B} p)=(v\z\di v\o\tuno{})\ot_{N}(w\di v\o\tdue{}p).\]
\end{proof}

In the theory of Galois and biGalois objects, monoidal "fiber" functors are always given by Galois extensions \cite{U87,U89}. At this time we cannot prove a full analog of this statement, but we have the following partial result: 
\begin{prop}\label{contensor-monoidal-converse}
    Let $P$ be a right $\cL$-comodule algebra with $N\subset P^{\co \CL}$ and such that we have a monoidal functor as in \Cref{cotensor-monoidal}. Then $P$ is an anti-right $\cL$-Galois extension.
\end{prop}
\begin{proof}
    Because the functor preserves the tensor units, we have $N\cong P\Box^\CL B=P^{\co \CL}$. The isomorphism between coinvariants and cotensor product we used is a special case of (the coopposite version of) \Cref{coinvariants vs cotensor product}. We have a commutative diagram
    \[
    \xymatrix{P\ou NP\ar[r]^{\rcan}\ar[d]^\cong
    &P\di\cL\\
    (P\Box^\cL\cL)\ou N(P\Box^\cL\cL)\ar[r]^-\xi
    &P\Box^\cL(\cL\ou B\cL)\ar[u]^f}
    \]
    where $f$ is the isomorphism from \Cref{lem:ulbrichanalog}.
\end{proof}

\section{Hopf BiGalois extensions}\label{sec:biGalois}
\subsection{Some technical preparations}

\begin{rem}
    Let $P,Q$ be  left $\cL$-comodule algebras and $M\in{^\cL_P\M_Q}$. Then 
    \[P\ou BQ\ot {^{co\cL}M}\ni p\ot q\ot m\mapsto pmq\in M\]
    is well defined and a left $\cL$-comodule map. In particular $^{co\cL}(P\ou B Q)$ is an algebra under 
    $(p\ot q)(p'\ot q')=pp'\ot q'q$ and there is a functor
    \[^{\cL}_P\M_Q\to {_{^{co\cL}(P\ou BQ)}\M}, M\mapsto {^{co\cL}M}.\]
    Also note that $^{co\cL}P$ and $(^{co\cL}Q)^{\op}$ are obvious subalgebras of $^{co\cL}(P\ou BQ).$
\end{rem}
Indeed the first map is well-defined by the previous remark, colinearity is obvious and implies that the map restricts to a map
\[^{co\cL}(P\ou BQ)\ot {^{co\cL}M}\to {^{co\cL}M}\]

\begin{prop}\label{adjoint trick}
    Let $\cL$ be a left Hopf algebroid over $B$, and $N\subset P$ a right faithfully flat left $\cL$-Galois extension. Then the functor 
    \[_N\M\ni M\mapsto P\ou NM\in{^\cL\M}\]
    is right adjoint. The left adjoint is 
    \[^\cL\M\ni \Gamma\mapsto {^{co\cL}}(P\ot_{B}\Gamma)\in{_N\M}.\]
    The unit of the adjunction is given by
    \[\alpha\colon\Gamma\ni \eta\mapsto \eta\mo\tuno{}\ot \eta\mo\tdue{}\ot \eta\z\in P\ou N{^{co\cL}(P\ou B\Gamma)}.\]
    The same formulas describe a pair of adjoint functors between $_N\M_N$ and $^\cL\M_N.$

\end{prop}
\begin{proof}
    This is just the composition between the adjoint equivalence of \Cref{thm. fundamental theorem for left Hopf Galois extensions} and the standard fact that the forgetful functor from the category $^\cL_P\M$ of left $P$ modules in the monoidal category $^\cL\M$ to the underlying category $^\cL\M$ is right adjoint to tensoring with $P$.
\end{proof}
We stated \Cref{adjoint trick} for left comodules and left modules because it involves as little as possible rearrangements of tensor product, and uses fewer hypotheses. To make the Ehresmann Hopf algebroid construction work below, however, we need the twisted version using part (2) of \Cref{schneider hopf module corollaries} instead. We will formulate this adjunction in terms of a universal property.

\begin{cor}\label{twisted adjoint trick}
    Let $\cL$ be a left Hopf algebroid over $B$, $N\subset P$ a left faithfully flat left $\cL$-Galois extension, and assume that $P$ is a regular comodule.
    Write $\ol N=N^\op$. Then for every $\Gamma\in{^{\cL}\M}$ we have a map 
    \[\alpha\colon\Gamma\ni  \eta\mapsto {\eta\ro}\tdue{}\ot \eta\rz{}\ot {\eta\ro}\tuno{}\in P\diamond_{\overline N}{^{\co\cL}(\Gamma\ou BP)}\]
    of left $\cL$-comodules
    with the universal property that for all $M\in {_{\ol N}}\M$ and every map $\beta\colon \Gamma\to P\diamond_{\ol N} M$ of left $\cL$-comodules there is a unique map $f\colon{^{co\cL}(\Gamma\ou BP)}\to M$ of  left $\ol N$-modules making
    \[\xymatrix{\Gamma\ar[rr]^-\alpha\ar[drr]_-\beta&&P\diamond_{\overline N}{^{\co\cL}(\Gamma\ou BP)}\ar[d]^{P\diamond_{\ol N}f}\\&&P\diamond_{\ol N}M}\]
    commute. 
     If $A$ is a $k$-algebra and $\Gamma\in{^\cL_A}\M$ (meaning $\Gamma$ is in addition an $A$-module such that multiplication with an element of $A$ is a comodule map) $\alpha$ is an $A$-module map, and if $\beta$ is an $A$-module map then so is $f$.
\end{cor}
Note that $\alpha$ is defined so that the composition
\[\Gamma\xrightarrow{\alpha}P\diamond_{\overline N}{^{\co\cL}(\Gamma\ou BP)}\cong \Gamma\ou BP\]
maps $\eta$ to $\eta\ot 1$. Also note that the map $f$ corresponding to $\beta$ is the coinvariant part of
\begin{align}\label{equ. universal prop}
    \Gamma\ot_{B} P \ni \eta\ot p\mapsto \beta(\eta)p.
\end{align}

\subsection{Ehresmann algebroids}

The following construction generalizes results in \cite{schau4}.
\begin{thm}\label{def. right ES bialgebroids}
    Let $\cL$ be a $B$-Hopf algebroid, faithfully flat as left $B$-module and $\ol B$-module. Let $N\subseteq P$ be a left faithfully flat left $\cL$-Hopf Galois extension.  We set $\overline N=N^{\op}$. Define $R(P, \cL):={}^{co\cL}(P\ot_{B}P)$.
    Then $R(P,\cL)$ is a $\overline N$-bialgebroid  with 
    \begin{align*}
        (p\ot q)(p'\ot q')&=pp'\ot q'q\\
         s\colon \overline N&\ni\overline{n}\mapsto 1\ot n\in R(P,\cL)\\
         t\colon N&\ni n\mapsto n\ot 1\in R(P,\cL)\\
         \Delta(p\ot q)&={p\ro}\tdue{}\ot q\ot p\rz\ot {p\ro}\tuno{}\\
         &=\tdue{q\mo}\ot_{B}q\z\diamond_{\overline{N}}p\ot_{B}\tuno{q\mo}\\
         \varepsilon(p\ot q )&=pq.
    \end{align*}
  Moreover,  $P$ is a right $R(P,\cL)$-comodule algebra by
    \[\delta\colon P\to P\diamond_{\ol N}R(P,\cL); p\mapsto {p\ro}\tdue{}\ot p\rz\ot {p\ro}\tuno{}.\]
    In this way $P$ is an anti-right $R(P,\cL)$-Galois extension of its coinvariants, with the anti-translation map given by 
    \begin{align}\label{expl. transl.}
\yi{(p\ot_{B}q)}\ot_{B}\er{(p\ot_{B}q)}=p\ot_{B}q.
    \end{align}
   
    If $\CH$ is an $\ol N$-bialgebroid and $P$ is a right $\CH$-comodule algebra such that it is a $\CL$-$\CH$-bicomodule algebra, then there is a unique bialgebroid homomorphism $f\colon R(P,\CL)\to\CH$ compatible with the comodule structures in the obvious way. It is an isomorphism if and only if $P$ is $\CH$-anti-Galois.
    
    If $P$ is faithfully flat as left or right $B$-module, then $P^{\co R(P,\cL)}=B$. 
    
    If $P$ is  right faithfully flat over $N$ then it is a regular right $R(P,\cL)$-comodule, with reverse comodule structure
     \begin{align}
        {}_{R}\delta(p)=p\rmo\diamond_{\overline{N}}p\rz=&\tdue{p\mo}\ot_{B}p\z\diamond_{\overline{N}}\tuno{p\mo},
    \end{align}
    and $R(P,\CL)$ is a Hopf algebroid.
\end{thm}
\begin{proof}
    We already remarked that the algebra structure is well defined. The proposed comodule structure on $P$ is an instance of the universal map in \Cref{twisted adjoint trick}.  It is an algebra map by \Cref{lem. P opposite has Galois structure} and \cref{equ. translation map 7}. 
    It follows from the universal property (with $A=N$) that $R(P,\cL)$ has a unique coring structure $R(P,\cL)\to R(P,\cL)\diamond_{\overline N}R(P,\cL)$ for which $P$ is a comodule, namely the unique $\Delta$ making
    \begin{equation*}
        \xymatrix{P\ar[rr]^-{\delta}\ar[d]_{\delta}&&P\diamond_{\ol N}R(P,\cL)\ar[d]^{P\diamond_{\ol N}\Delta}\\
        P\diamond_{\ol N}R(P,\cL)\ar[rr]^-{\delta\diamond_{\ol N}R(P,\cL)}&&P\diamond_{\ol N}R(P,\cL)\diamond_{\ol N}R(P,\cL)}
    \end{equation*} commute. It is a standard argument how to derive coassociativity of $\Delta$ from this definition. To derive the explicit form of $\Delta$ using \eqref{equ. universal prop}, first observe that
\begin{align*}
\beta(p)=&p\ro\tdue{}\ro\tdue{}\ot\,p\ro\tdue{}\rz\ot p\ro\tdue{}\ro\tuno{}\ot\,p\rz\ot\,p\ro\tuno{}\\
=&p\ro\t\tdue{}\ot p\ro\o\tdue{}\ot p\ro\t\tuno{}\ot p\rz\,\ot p\ro\o\tuno{}.
\end{align*}
Thus, for any $p\ot q\in R(P,\cL)$ we have
\begin{align*}
    \Delta(p\ot q)=&\beta(p)q\\
    =&p\ro\t\tdue{}\,q\ot p\ro\o\tdue{}\ot p\ro\t\tuno{}\ot p\rz\,\ot p\ro\o\tuno{}\\
    =&q\mo\tdue{}\,q\z\ot q\mt\tdue{}\ot q\mo\tuno{}\ot p\,\ot q\mt\tuno{}\\
    =&1\ot q\mo\tdue{}\ot q\z\ot p\,\ot q\mo\tuno{},
\end{align*}
where we use \Cref{coinvariants vs cotensor product} in the second step. This proves the second claimed formula for comultiplication.  The first formula is a consequence of \Cref{coinvariants vs cotensor product}. $\Delta$ is an algebra map by the same argument as for the coaction. 
   The composition 
    \[P\xrightarrow{\delta}P\diamond_{\ol N}R(P,\cL)\cong P\ou BP\]is given by $p\mapsto p\ot 1$, where the isomorphism is given by (2) of Corollary~\ref{schneider hopf module corollaries}. More precisely,
    \begin{align*}
        P\ot_{B}P&\cong P\diamond_{\ol N}R(P, \cL),\\
        p\ot q&\mapsto p\ro\tdue{}\,q\ot p\rz\ot p\ro\tuno{}\\
        p'\ot pq&\mapsfrom q\ot (p'\ot p).
    \end{align*}
      Therefore, $P^{\co R(P,\cL)}\subset P$ is the equalizer of 
   $\iota\ou BP,P\ou B\iota\colon P\to P\ou BP$ by the diagram
     \[\begin{tikzcd}
   P^{\co R(P,\cL)}
   \arrow[r,hookrightarrow]
   & 
   P 
   \arrow[rr,shift left,"\delta" near end]
   \arrow[rr,shift right,"P\ot\iota"' near end]
   \arrow[ddrr,shift left, "\!P\ot\iota" near end ]
   \arrow[ddrr,shift right,"\!\iota\ot P"' near end]&&P\diamond_{\ol N}R(P,\cL)
   \arrow[dd,"\cong"]
   \\\\&&&P\ou BP
   \end{tikzcd}
   \]
 Thus, if $P$ is right or left faithfully flat over $B$ then $P^{\co R(P,\cL)}=B$ by faithfully flat descent. In general, $P^{\co R(P,\CL)}=\hat B:=\{p\in P|p\ot 1=1\ot p\in P\ou BP\}$. Note for the following that $P\ou{\hat B}P=P\ou BP$ by \Cref{different coinvariants}.
 
     The anti-right Galois map for the right $\cL$-extension $B\subset P$ is the isomorphism 
    \[P\ou BP\cong P\diamond_{\ol N}R(P,\cL)\]
    coming from (2) of Corollary~\ref{schneider hopf module corollaries} and thus $P$ is anti-right Galois; its inverse maps $1\ot p\ot q$ to $p\ot q$ proving that the translation map has the claimed form. By definition of $\delta$, the diagram
\[
\begin{tikzcd}
    P\ot_{N} R(P,\cL)
     \arrow[rr,"\psi"]
     \arrow[drr,"f"]
    &&P\diamond_{\ol N}R(P,\cL)
     \arrow[d,"g"]
     \\
    &&P\ou BP
\end{tikzcd}\]
 commutes, where $\psi$ is from the definition of regular comodule, and given by $\psi(p'\ot (p\ot q))=p'\ro\tdue{}\ot p'\rz\,p\ot q\, p'\ro\tuno{}$;  $f$ is the counit of the adjoint equivalence from \Cref{thm. fundamental theorem for left Hopf Galois extensions}, which is given by $f(p'\ot p\ot q)=p'p\ot q$ with inverse $f^{-1}(p\ot q)=p\,q\mo\tuno{}\ot q\mo\tdue{}\ot q\z$; and $g$, given by $g(p'\ot p\ot q)=p\ot qp'$, is, up to the obvious identification, the counit of the variant (2) from \Cref{schneider hopf module corollaries}, thus also is an isomorphism under the additional flatness hypothesis. Thus $P$ is regular as claimed. To find the reverse comodule structure explicitly, we calculate
\begin{align*}
    \psi\inv(p\ot 1)&=f\inv g(p\ot 1\ot 1)\\
    &=f\inv(1\ot p)\\
    &= p\mo\tuno{}\ot p\mo\tdue{}\ot p\z.
\end{align*}
 Assume $P$ is a $\cL$-$\CH$-bicomodule algebra for some $\ol N$-bialgebroid  $\CH$. Denote the comodule structure by $\rho\colon p\mapsto p\z\ot p\o\in P\diamond_{\ol N}\CH$. The corresponding map $f\colon R(P,\cL)\to \CH$ as in \Cref{twisted adjoint trick} is given by
    $f(p\ot q)=\rho(p)q$ as an element of ${^{\co\cL}(P\diamond_{\ol N}\CH)}\cong \CH$, so 
    \begin{align*}
        f(p\ot q)=&p\z q\ot p\o\in N\diamond_{\ol N} \CH \\
        =&\ol{p\z q}p\o\in \CH\subset P\diamond_{\ol N}\CH.
    \end{align*}
 
    In other words, if we restrict the anti-Galois map 
    \[{}^{\bullet}_{\bullet}P\ou B {}^{\bullet}P\to {}^{\bullet}_{\bullet}P\diamond_{\ol N}\CH,\quad p\ot q\mapsto p\z q\ot p\o,\]
    for the right $\CH$-comodule algebra $P$, which is a map of Hopf modules in $^{\cL}_P\mathcal M$, to the left $\cL$-coinvariants, we obtain
    \[f\colon R(P,\cL)={^{\co\cL}(P\ou BP)}\to {^{\co\cL}(P\diamond_{\ol N}\CH)}\cong \CH.\]
    In particular, if $P$ is $\CH$-anti-Galois, then $f$ is an isomorphism.  Also, $f$ is multiplicative because the image of $\rho$ lies in the Takeuchi product:
    $$f(pp'\ot q'q)=\ol{p\z p'\z q'q}p\o p'\o=\ol{p\z q}p\o\ol{p'\z q'}p'\o=f(p\ot q)f(p'\ot q').$$
    From the definition of both $f$ and the comultiplication of $R(P,\cL)$ through the universal property in \Cref{twisted adjoint trick} it is clear that $f$ is also a coring map, and thus it is a morphism of $\ol N$-bialgebroids. 
    
    Finally, that $R(P,\CL)$ is Hopf follows from \Cref{lem. left Hopf Galois induce left Hopf} and  \Cref{lem. left skew regular Hopf Galois induce anti-left Hopf} because it is regular and we have the needed faithful flatness conditions over $N$.
\end{proof}
 
 \begin{rem}\label{ehresmann regular comodule structures}
     The two formulas for comultiplication on $R(P,\CL)$ can be read as saying that the right comodule structure of $R(P,\CL)={}^{\co\CL}(P\ou BP)$ over itself is induced by the right comodule structure of the left tensor factor, and the left comodule structure by the reverse left comodule structure of the right tensor factor.
 \end{rem}

\begin{prop}\label{prop. right ES}
     Let $\cL$ be a $B$-Hopf algebroid, left faithfully flat over $B$ and $\ol B$, and let $N\subseteq P$ be a left faithfully flat  left $\cL$-Hopf Galois extension. Then 
     \begin{align}\label{equ. right ES 1}
         R(P, \cL)=&\{p\ot_{B}q\in P\ot_{B}P\quad|\quad p\tuno{q\mo}\ot_{N}\tdue{q\mo}\ot_{B}q\z=1\ot_{N} p\ot_{B} q\}, \\
         \label{equ. right ES 2} 
         =&\{p\ot_{B}q\in P\ot_{B} P\quad|\quad p\rz\ot_{B} p\ro\di{} q=p\ot_{B} q\mo\di{} q\z \}, \\
      \label{equ. right ES 3}
      =&\{p\ot_{B}q\in P\ot_{B}P\quad|\quad p\rz\ot_{B}\tuno{p\ro}\ot_{N}\tdue{p\ro}q=p\di{}q\ot_{N}1 \}. 
     \end{align}   
\end{prop}
\begin{proof}

    For \eqref{equ. right ES 1}, by definition, $p\ot q\in P\ou BP$ belongs to $R(P,\cL)$ if and only if
    \begin{align*}1\ot p\ot q=p\mo\,q\mo\ot p\z\ot q\z\in \cL\diamond_B P\ou B P.\end{align*}
    Now applying the translation map to the first two factors maps the left hand side to $1\ot p\ot q\in P\ou NP\ou BP$, and the right hand side to 
    \begin{align*}\tuno{(p\mo\,q\mo\,)}&\ot\tdue{(p\mo\,q\mo\,)}p\z \ot q\z\\
    &=\tuno{p\mo\,}\tuno{q\mo\,}\ot\tdue{q\mo\,}\tdue{p\mo\,}p\z\ot q\z\\
    &=
    p\tuno{q\mo\,}\ot \tdue{q\mo}\ot q\z,\end{align*}
    using successively \eqref{equ. translation map 7} and \eqref{equ. translation map 4}.  The proof of \eqref{equ. right ES 3} is similar. Also, \eqref{equ. right ES 2} is the result of \Cref{coinvariants vs cotensor product}.   
\end{proof}

By the general correspondence between left comodule structures and right comodule structures over the coopposite, we get a corresponding construction of a left Ehresmann Hopf algebroid for a right anti-Galois extension. 
\begin{cor}\label{lem and def. left Ehresmann Hopf algebroids}
     Let $\CH$ be a $B$-Hopf algebroid, faithfully flat as left $B$ and $\ol B$-module. Let $N\subseteq P$ be a right  $\CH$-anti-Hopf Galois extension such that $P$ is faithfully flat over $N$ on the left hand side. Define
    $L(P, \CH):=(P\ot_{\overline{B}}P)^{co\CH}$. Then $L(P,\CH)$ is a $N$-bialgebroid by  \[s(n)=n\ot_{\overline{B}}1,\quad t(n)=1\ot_{\overline{B}}n,\quad (p\ot_{\overline{B}}q)(p'\ot_{\overline{B}}q')=pp'\ot_{\overline{B}}q'q,\]
    \[\Delta(p\ot_{\BB} q)=p\ot_{\BB}\yi{q\o}\ot_{N}\er{q\o}\ot_{\BB}\,q\z=p\rz\ot_{\BB}\yi{p\rmo}\ot_{N}\er{p\rmo}\ot_{\BB}\,q,\qquad \varepsilon(p\ot_{\BB} q)=pq,\]
    for any $p\ot_{\BB}q, p'\ot_{\BB}q'\in L(P, \CH)$.
    $P$ is a left comodule algebra over $L(P,\CH)$ 
    \begin{align}
        p\mo\ot p\z=&p\rz\ot\yi{p\rmo}\ot\er{p\rmo}\in L(P,\CH)\diamond_{N}P
        ,\\
        \intertext{If $P$ is right faithfully flat over $N$ the comodule structure is regular with reverse}
        p\rz\ot p\ro=&\yi{p\o}\ot\er{p\o}\ot p\z\in P\diamond_{N}L(P,\CH),
    \end{align}
    $P$ is a left $L(P,\CH)$-Galois extension of its coinvariants with translation map
    \begin{align}
        \tuno{(p\ot_{\BB}q)}\ot_{\BB}\tdue{(p\ot_{\BB}q)}=p\ot_{\BB}q,
    \end{align}
    and $L(P,\CH)$ is a Hopf algebroid.
    
    If $P$ is left or right faithfully flat over $\ol B$ then the coinvariants are $\ol B$. 

    \end{cor}

\begin{prop}\label{prop. left ES}
Let $\CH$ be a  $B$-Hopf algebroid, faithfully flat as left $B$ and $\ol B$-module. Let $N\subseteq P$  be a  left  faithfully flat anti-right $\CH$-Hopf Galois extension. Then for any $p\ot_{\BB}q\in L(P, \CH)$, $\Bar{b}p\ot_{\BB}q= p\ot_{\BB}q\Bar{b}$ and
     \begin{align}\label{equ. left ES 1}
         L(P, \CH)=&\{p\ot_{\BB}q\in P\ot_{\BB}P\quad|\quad p\yi{q\o}\ot_{N}\er{q\o}\ot_{\BB}q\z=1\ot_{N} p\ot_{\BB} q\} \\
       \label{equ. left ES 2} 
       =&\{p\ot_{\BB}q\in P\ot_{\BB}P\quad|\quad p\rz\ot p\rmo\ot q=p\ot q\o\ot q\z \} \\
      \label{equ. left ES 3}
      =&\{p\ot_{\BB}q\in P\ot_{\BB}P\quad|\quad p\rz\ot_{\BB}\yi{p\rmo}\ot_{N}\er{p\rmo}q=p\ot_{\BB}q\ot_{N}1 \} 
     \end{align}    
\end{prop}

\begin{rem}
    Recall that for any Hopf algebra $H$ and a right $H$-comodule algebra  $P$, $N=P^{coH}\subseteq P$ is a faithfully flat $H$-Galois extension if the map
    \[\can:P\ot_{N}P\to P\ot H,\qquad p\ot q\mapsto pq\z\ot q\o, \]
    is bijective. If $H$ has a bijective antipode, then $N=P^{coH}\subseteq P$ is an $H$-Galois extension if and only if $N=P^{coH}\subseteq P$ is an anti-right $H$-Galois extension. More precisely, if we denote the translation map for the classical Hopf Galois extension by \[\can{}^{-1}(1\ot h)=\teins{h}\ot_{N}\tzwei{h},\qquad \forall h\in H.\] 
    Then we have the anti-right translation map
\[\yi{h}\ot_{N}\er{h}=\teins{S^{-1}(h)}\ot_{N}\tzwei{S^{-1}(h)}.\]
Recall that in \cite{schau4}, $L(P, H)$ has the coproduct
\[\Delta(p\ot q)=p\z\ot p\o\teins{}\ot_{N}p\o\tzwei{}\ot q,\qquad \forall p\ot q\in L(P, H).\]
Therefore, $L(P, H)$ in \Cref{lem and def. left Ehresmann Hopf algebroids} recovers the Ehresmann Hopf algebroid in \cite{schau4}.
\end{rem}

\subsection{BiGalois Extensions}
\begin{defi}\label{def. Hopf biGalois extensions}
    Let $B,N$ be two algebras. 
    Let $\cL$ be a $B$-Hopf algebroid  and let $\CH$ an $N$-Hopf algebroid 
    
    A $\cL$-$\CH$-Hopf biGalois extension is an algebra $P$,\ such that $\ol N\subseteq P$ is a 
        left $\cL$-Galois extension and $B\subseteq P$ is an  anti-right $\CH$-Galois extension, in such a way that the two comodule structures make $P$ a $\cL$-$\CH$-bicomodule. 

    We denote by $\Bi(\cL, \CH)$ the collection of $\cL$-$\CH$-Hopf biGalois extensions that are faithfully flat as left and right modules both over $B$ and $\ol N$. 

    By $\underline{\Bi}(\CL,\CH)$ we will denote the set of isomorphism classes of elements of $\Bi(\CL,\CH)$
\end{defi}

\begin{thm}\label{lem. Hopf Galois extension is HOpf biGalois extension}
     Let $\cL$ be a $B$-Hopf algebroid and faithfully flat as left $B$ and $\ol B$-module.
     Let $N\subseteq P$ be a faithfully flat left $\cL$-Hopf Galois extension such that $P$ is also faithfully flat over $B$ on one side. 
     
     Then $P$ is an $\cL$-$R(P,\cL)$-biGalois extension (in particular $R(P,\CL)$ is a Hopf algebroid), and this structure is unique in the sense that for every $\cL$-$\CH$-biGalois structure extending the given $\cL$-structure there is a unique isomorphism $R(P,\cL)\cong \CH$ identifying the structures in the obvious sense. Put $\CH=R(P,\CL)$ for the sequel.

     If $P$ is left faithfully flat over $B$, then $R(P,\CL)$ is left and right faithfully flat over $N$.
     
     If $P$ is right faithfully flat over $B$, then $R(P,\CL)$ is left and right faithfully flat over $\ol N$.
\end{thm}
\begin{proof}
   That $P$ becomes a biGalois extension was proved in \Cref{def. right ES bialgebroids}. Notably the fact that the two comodule structures commute is implicit in our construction of the $R(P,\CL)$-comodule structure through the universal property \Cref{twisted adjoint trick}.

   It remains to treat the faithful flatness assertions on $R(P,\CL)$.
By (2) of \Cref{schneider hopf module corollaries} we have $P\ot_{B}P\cong {}^{\co\cL}(P\ot_{B}P)\ot_{N}P= R(P,\cL)\ot_{N} P$. The left (resp.\ right) $N$-module structure on $R(P,\CL)$ corresponds under this isomorphism to the left (resp.\ right) $N$ module structure on the left factor $P$. As a result, $R(P,\cL)$ is left $N$-faithfully flat if $P$ is left faithfully flat over $N$ and $B$. Similarly, by \Cref{thm. fundamental theorem for left Hopf Galois extensions}, $P\ot_{N}R(P,\cL)\cong P\ot_{B}P$. Here, the left (resp.\ right) $\ol N$-module structure of $R(P,\CL)$ corresponds to the right (resp.\ left) $N$-module structure on the right tensor factor $P$. As a result, $R(P,\CL)$ is faithfully flat as left $\ol N$-module if $P$ is faithfully flat as right module over $N$ and over $B$.   
\end{proof}

 \begin{lem} Let $\CH$ be a $B$-Hopf algebroid and faithfully flat as left module over $B$ and $\ol B$. Let $N\subseteq P$ be a right $\CH$-anti-Galois extension, faithfully flat as left and right module over $N$ and $\ol B$. Then $L(P,\CH)$ is faithfully flat over $N$ and $\ol N$ on both sides and we have
    \[X_{+}\ot_{\overline{N}}X_{-}=p\ot_{\BB}\er{q\rmo}\ot_{\overline{N}}q\rz\ot_{\BB}\yi{q\rmo}\]
        \[X_{[+]}\ot_{N}X_{[-]}=\yi{p\o}\ot_{\BB}q\ot_{N}\er{p\o}\ot_{\BB}p\z,\]
        for $X=p\ot q\in L(P,\CH)$.
\end{lem}
\begin{proof}
   The assertions on faithful flatness of $L(P,\CH)$ are contained in the coopposite of \Cref{lem. Hopf Galois extension is HOpf biGalois extension}.
   
   For the statement on the form of the translation maps we use faithful flatness to show the formulas are well defined. We first check $\yi{p\o}\ot_{\BB}q\ot_{N}\er{p\o}\ot_{\BB}p\z\in L(P,\CH)\ot_{N}L(P,\CH)$. We have
   \begin{align*}
      \delta(\yi{p\o}\ot q)\ot \er{p\o}\ot p\z=& \yi{p\o}\z\ot q\z\ot \yi{p\o}\o q\o \ot \er{p\o}\ot p\z\\
      =&\yi{p\o}\ot q\z\ot p\t q\o \ot \er{p\o}\ot p\z\\
      =&\yi{p\o}\ot q\ot 1 \ot \er{p\o}\ot p\z,
   \end{align*}
   and
   \begin{align*}
       \yi{p\o}\ot q\ot \delta(\er{p\o}\ot p\z)=& \yi{p\t}\ot q\ot \er{p\t}\z\ot p\z\ot \er{p\t}\o p\o\\
       =& \yi{p\t{}_{[+]}}\ot q\ot\er{p\t{}_{[+]}}\ot p\z\ot p\t{}_{[-]} p\o\\
       =& \yi{p\o}\ot q\ot \er{p\o}\ot p\z\ot 1.
   \end{align*}
Now we check that the formulas indeed give the translation maps by calculating on the one hand
\begin{align*}
    \mu(X_{[+]}\ot X_{[-]})=&\mu( \yi{p\o}\ot q\ot \er{p\o}\ot p\z)\\
    =&(\yi{p\o}\ot q\o\yi{})(\er{p\o}\ot p\z)\ot (q\o\er{}\ot q\z)\\
    =&(\yi{p\o}\er{p\o}\ot p\z\,q\o\yi{})\ot (q\o\er{}\ot q\z)\\
    =&(1\ot p\,q\o\yi{})\ot (q\o\er{}\ot q\z)\in 1\ot_{\BB} N\ot_{N} L(P,\cL)\\
    =&(1\ot 1)\ot (p\,q\o\yi{}q\o\er{}\ot q\z)\\
    =&(1\ot 1)\ot(p\ot q).
\end{align*}
and on the other hand
\begin{align*}
     \two{X}{}_{[+]}\ot\two{X}{}_{[-]}\one{X}=&(p\rmo{}\er{}\o\yi{}\ot q)\ot (p\rmo{}\er{}\o\er{}\ot p\rmo{}\er{}\z)(p\rz\ot p\rmo{}\yi{})\\
     =&(p\rmo{}\er{}\o\yi{}\ot q)\ot (p\rmo{}\er{}\o\er{}p\rz\ot p\rmo{}\yi{} p\rmo{}\er{}\z)\\
     =&(p\rmo{}\er{}\o\yi{}\ot q)\ot (p\rmo{}\er{}\o\er{}p\rz\ot p\rmo{}\yi{} p\rmo{}\er{}\z)\\
     =&(p\rmo{}{}_{[-]}\yi{}\ot q)\ot (p\rmo{}{}_{[-]}\er{}p\rz\ot p\rmo{}{}_{[+]}\yi{} p\rmo{}{}_{[+]}\er{})\\
     =&(p\o\yi{}\ot q)\ot (p\o\er{}p\z\rz\ot p\z\rmo\yi{} p\z\rmo\er{})\\
     =&(p\o\yi{}\ot q)\ot (p\o\er{}p\z\ot 1)\\
     =&(p\ot q)\ot (1\ot 1).
\end{align*}
where we use $ p\z\rz\ot p\z\rmo\ot p\o=p\rz\ot p\rmo{}_{[+]}\ot p\rmo{}_{[-]}$ in the sixth step.
\end{proof}

\begin{rem}
    In other words, given a $B$-Hopf algebroid $\CL$, faithfully flat over $B$ and $\ol B$ on the left, and a left $\CL$-Galois extension $N\subset P$ faithfully flat over $N$ and $B$ on both sides, there is a unique $\CH=R(P,\CL)$, an $\ol N$-Hopf algebroid faithfully flat over $\ol N$ and $N$ on both sides, completing $P$ to be an element $P\in\Bi(\CL,\CH)$. (Note that $\CL$ is also faithfully flat over $B$ and $\ol B$ on the right by \Cref{lem. left Hopf Galois induce left Hopf}.) The situation is symmetric, so that similarly, given a $B$-Hopf algebroid $\CH$ left faithfully flat over both copies of its base, and a right  $\CH$-anti-Galois extension $N\subset P$ faithfully flat on the left and right over $N$ and $\ol B$, there is a unique $\CL=L(P,\CH)$,  faithfully flat over both copies of its base on both sides, making $P\in\Bi(\CL,\CH)$. 
\end{rem}

\begin{cor}\label{lem. round trip Ehresmann bialgebroid}
    Let $\CL$ be a $B$-Hopf algebroid that is left faithfully flat over $B$ and $\ol B$, and $N\subseteq P$ a left $\CL$-Galois extension which is faithfully flat as a left and right module over $B$ and $N$. Then $L(P, R(P, \cL))\simeq \cL$. 

    Similarly, starting with a right $\CH$-anti-Galois extension $M\subset P$ satisfying analogous faithful flatness conditions we have $R(P, L(P, \CH))\simeq \CH$. 
    
    More precisely, the isomorphism $L(P, R(P, \cL))\simeq \cL$ is given by  the left translation map 
    \[\ltau: \cL\to L(P, R(P, \cL)), \quad X\mapsto \tuno{X}\ot_{N}\tdue{X}.\]
    Similarly, the isomorphism $R(P, L(P, \CH))\simeq \CH$ is  given by the anti-right translation map
    \[\rtau:\CH \to R(P, L(P, \CH)), \qquad X\mapsto \yi{X}\otimes_{M}\er{X}.\]
\end{cor}
\begin{prop}\label{prop. left to right relations}
    Let $P$ be an $\CL$-$\CH$-biGalois extension with regular comodule structures. Then 
    \begin{align}\label{equ. biGalois translation maps 1}
        \alpha\tuno{}\rz\ot \alpha\tuno{}\rmo\ot \alpha\tdue{}=&\alpha\tuno{}\ot \alpha\tdue{}\o\ot \alpha\tdue{}\z;\\
        \label{equ. biGalois translation maps 2}
        X\yi{}\rz\ot X\yi{}\ro\ot X\er{}=&X\yi{}\ot X\er{}\mo\ot X\er{}\z,
    \end{align}
    for any $\alpha\in \cL$ and $X\in \CH$.
\end{prop}
\begin{proof}
   Both sides of \eqref{equ. biGalois translation maps 1} belong to $\int^{a, b}\int_{c, d}{}_{\Bar{b}}P_{\Bar{c}}\ot {}_{\Bar{c},\Bar{d}}\cL_{a, \Bar{b}}\ot {}_{\Bar{d}}P_{\Bar{a}}$. Indeed, the map
   \begin{align*}
f:\int^{a}\int_{b} {}_{\Bar{a}}P_{\Bar{b}}\ot {}_{b}P_{a}\to \int^{a, b}\int_{c, d}{}_{\Bar{b}}P_{\Bar{c}}\ot {}_{\Bar{c},\Bar{d}}\cL_{a, \Bar{b}}\ot {}_{\Bar{d}}P_{\Bar{a}} 
   \end{align*}
given by
\begin{align*}
    f(p\ot q)=p\rz\ot p\rmo\ot q
\end{align*}
is well defined. Similarly,
the map
   \begin{align*}
g:\int^{a}\int_{b} {}_{\Bar{a}}P_{\Bar{b}}\ot {}_{b}P_{a}\to \int^{a, b}\int_{c, d}{}_{\Bar{b}}P_{\Bar{c}}\ot {}_{\Bar{c},\Bar{d}}\cL_{a, \Bar{b}}\ot {}_{\Bar{d}}P_{\Bar{a}} 
   \end{align*}
given by
\begin{align*}
    g(p\ot q)=p\ot q\o\ot q\z
\end{align*}
is well defined.
Apply $ \lcan_{1,3}\circ (\psi\ot \id_{P})$ on both sides of \eqref{equ. biGalois translation maps 1}. On the one hand,
\begin{align*}
    \lcan_{1,3}\circ (\psi\ot \id_{P})(\alpha\tuno{}\rz\ot \alpha\tuno{}\rmo\ot \alpha\tdue{})=&\lcan_{1,3}(\alpha\tuno{}\ot 1\ot  \alpha\tdue{})\\
    =&\alpha\ot 1\ot 1.
\end{align*}
On the other hand,
\begin{align*}
     \lcan_{1,3}&\circ (\psi\ot \id_{P})(\alpha\tuno{}\ot \alpha\tdue{}\o\ot \alpha\tdue{}\z)\\
     =&  \lcan_{1,3}(\alpha\tuno{}\z\ot \alpha\tuno{}\o \alpha\tdue{}\o\ot \alpha\tdue{}\z)\\
     =&\alpha\tuno{}\mo\ot \alpha\tuno{}\o \alpha\tdue{}\o\ot \alpha\tuno{}\z\alpha\tdue{}\z\\
     =&\alpha\tuno{}\mo\ot (\alpha\tuno{}\z \alpha\tdue{})\o\ot( \alpha\tuno{}\z\alpha\tdue{})\z\\
     =&\alpha\ot 1\ot 1.
\end{align*}
Since $ \lcan_{1,3}\circ (\psi\ot \id_{P})$ is an isomorphism, this proves \eqref{equ. biGalois translation maps 1}. The proof of \eqref{equ. biGalois translation maps 2} is similar.
\end{proof}

We have shown that one of the Hopf algebroids in a biGalois extension (under faithful flatness hypotheses) is always determined by the structures on the other side, but this is only true up to isomorphism. The following discussion makes this more precise.

\begin{defi}[\cite{HM23}]
    Let $\cL$ be a Hopf algebroid over $B$. A \emph{vertical bisection} is a unital $B$-bimodule map $\sigma:\cL\to B$ in the sense that $\sigma(b\Bar{b'}X)=b\sigma(X)b'$ such that $\sigma(XY)=\sigma(X\sigma(Y))=\sigma(X\overline{\sigma(Y)})$ for any $X, Y\in \cL$ and $b, b'\in B$. 
\end{defi}
Recall that instead of considering a bialgebroid to have a counit $\varepsilon\colon \CL\to B$, one can also consider an equivalent notion of counit as a $B^e$-ring map $\CL\to\operatorname{End}(B)$. Similarly, vertical bisections $\sigma\colon\CL\to B$ are in bijection with $B^e$-ring maps $\hat\sigma\colon \CL\to \operatorname{End}(B)$ by $\sigma(X)=\hat\sigma(X)(1)$, respectively $\hat\sigma(X)(b)=\sigma(Xb)$.

We know the collection of vertical bisections of a Hopf algebroid is a group with convolution product as the product and counit as the unit. We denote the group of bisections of $\cL$ by $Z^{1}(\cL, B)$. We can  generalize the main theorem of \cite{HL20}.

\begin{lem}\label{autos are bisections}
    Let $\CH$ be a Hopf algebroid over $B$, and $N\subseteq P$ be faithfully flat anti-right $\CH$-Galois extension,  then the automorphism group $\Aut_N^\CH(P)$ (the group of isomorphisms of right
$\CH$-comodule algebras $P$ to $P$ which preserve the coinvariant subalgebra $N$) is isomorphic to $Z^{1}(L(P, \CH), N)$.
\end{lem}

\begin{proof}
   The coopposite version of \Cref{twisted adjoint trick} gives a bijection between right colinear maps $F\colon P\to P$ and left $\ol N$-module maps $\sigma\colon L(P,\CH)\to N$, given by $F_\sigma(p)=\sigma(p\mo)p\z$, where we use the left comodule structure of $P$ over the Ehresmann Hopf algebroid. Using the universal property again, one can show that $F_\sigma$ is left $N$-linear iff $\sigma$ is, and that $F$ is unital iff $\sigma$ is, and $F$ is right $N$-linear iff $\sigma(X\ol n)=\sigma(Xn)$ for all $X\in L(P,\CL) $ and $n\in N$

   If $\sigma$ is a vertical bisection, then 
   \begin{multline*}
       F_\sigma(pq)=\sigma(p\mo q\mo)p\z q\z
    =\sigma(p\mo\ol{\sigma(q\mo)})p\z q\z
       \\=\sigma(p\mo)p\z\sigma(q\mo)q\z
       =F_\sigma(p)F_\sigma(q).
   \end{multline*}

    Conversely, assume $F_\sigma\in\Aut(P)$. We have $\sigma(p\ot q)=F_\sigma(p)q$. Therefore 
    \begin{align*}
        \sigma((p\ot q)(p'\ot q'))=&\sigma(pp'\ot q'q)=F_\sigma(pp')q'q=F_\sigma(p)F_\sigma(p')q'q\\
        =&\sigma(p\ot \sigma(p'\ot q')q)\\
        =&\sigma((p\ot q)\overline{\sigma(p'\ot q')}).
    \end{align*}

\end{proof}

Denote by $\Aut(\cL)$ the group of automorphism of the Hopf algebroid $\cL$. By Lemma 3.3 of \cite{HM23} we have a group homomorphism $\Ad: Z^{1}(\cL, B)\to \Aut(\cL)$ given by
\begin{align}
    \Ad_{\sigma}(X):=\sigma(X\o)\overline{\sigma(X\th)}X\t,
\end{align}
for  $X\in\cL$.  We define
\[\textup{CoInn}(\cL):=\{\Ad_{\sigma}| \sigma\in Z^{1}(\cL, B) \}\]
to be the set of coinner automorphism of the Hopf algebroid $\cL$. By Theorem of 3.4 of \cite{HM23}, we know 
$\textup{CoInn}(\cL)$ is a normal subgroup of $\Aut(\cL)$. More precisely,
for any $\Phi\in \Aut(\cL)$ and $\sigma\in Z^{1}(\cL, B)$, we have
\begin{align*}
    \Phi\circ \Ad_{\sigma}\circ \Phi^{-1}=\Ad_{\Phi\la \sigma},
\end{align*}
where $\Phi\la \sigma$ is a vertical bisection given by $\Phi\la\sigma:=\sigma\circ \Phi^{-1}$.
We denote by $\textup{CoOut}(\cL)=\Aut(\cL)/\textup{CoInn}(\cL)$ the group of co-outer automorphisms of $\cL$. As a generalization of Lemma 3.11 of \cite{schau5}, we have
\begin{lem}\label{free action of coinner auto group}
    Let $\cL$ be a Hopf algebroid over $B$ and $\CH$ be a Hopf algebroid over $N$, both faithfully flat on the left over both bases.  Then $\textup{CoOut}(\cL)$
 acts freely on $\underline{\Bi}(\cL, \CH)$. The orbit of the class of a biGalois extension $P\in\underline{\Bi}(\cL, \CH)$ consists of classes in $\underline{\Bi}(\cL, \CH)$ of those biGalois extensions $Q$ such that $P\cong Q$ as anti-right $\CH$-extensions.
 \end{lem}

\begin{proof}
    Let $\Phi\in \Aut(\cL)$ and $P\in \Bi(\cL,\CH)$. We define ${}^{\Phi}P:=\Phi\la P$ be the $\cL-\CH$-biGalois extension with the modified left $\cL$-comodule structure given by 
    \[{}^{\Phi}\delta(p)=\Phi(p\mo)\di p\z,\]
    for all $p\in {}^{\Phi}P$. The remaining structures of ${}^{\Phi}P$ are those of $P$. If $N\subseteq P$ is a left $\cL$-Galois extension, and $\BB\subseteq \CH$ is an anti-right $\CH$-Galois extension then obviously $N\subseteq {}^{\Phi}P$ is a left $\cL$-Galois extension with the deformed coaction, and a $\CL$-$\CH$-biGalois extension. This defines an action of $\Aut(\CL)$ on $\underline\Bi(\CL,\CH)$. The orbit is the same as stated because of the uniqueness statement in \Cref{def. right ES bialgebroids}. We need to prove that the stabilizer of $P$ is $\operatorname{CoInn}(\CL)$. Let $\sigma\in Z^{1}(\cL, B)$
    and $F_\sigma\colon P\to P$ as in \Cref{autos are bisections}. 
    \begin{align*}
        \delta(F_\sigma(p))=&\delta(\sigma(p\mo)p\z)=\sigma(p\mt)p\mo\ot p\z=\sigma(p\mfo)\overline{\sigma(p\mt)}p\mth\ot \sigma(p\mo)p\z\\
        =&(\id\ot F_\sigma)\circ{}^{\Ad_\sigma}\delta(p).
    \end{align*}
    Since every automorphism of $P$ as right $\CH$-comodule algebras is of this form, it follows from the universal property of the left coaction that $^\Phi P\cong P$ as bicomodule algebras iff there is $\sigma$ such that $\Phi=\Ad_\sigma$.
\end{proof}

\subsection{The groupoid of BiGalois extensions}

\begin{rem}
  If $P$ is a right comodule algebra of the $B$-bialgebroid  $\CH$ and $Q$ is a left comodule algebra of $\CH$, then $P\Box^{\CH}Q$ is a subalgebra of the Takeuchi product algebra $P\times_{B}Q$. 
\end{rem}

\begin{lem}\label{cotensor product galois extension}
    Let $A\subseteq P$ be a left faithfully flat anti-right $\CH$-Galois extension, where $\CH$ is a Hopf algebroid over $B$ and flat as a left module over $B$ and $\ol B$. 
    Let $Q$ be an $\CH$-$\mathcal G$-bicomodule algebra, where $\mathcal G$ is a Hopf algebroid over $C$. Assume that $\rcan\colon Q\ou BQ\to Q\diamond_C\mathcal G$ is bijective (so that $Q$ is an anti-right $\mathcal G$-Galois extension of  its coinvariants).

    The right $\mathcal G$-comodule structure of $Q$ induces a right $\mathcal G$-comodule structure on $P\Box^\CH Q$ if either $\mathcal G$ is flat as a left $C$-module, or $P$ is flat as left $\BB$-module.
    
    Under these hypotheses, $P\Box^\CH Q$ is a anti-right $\mathcal G$-Galois extension of its coinvariants.

    If $Q^{\co\mathcal G}=B$ and $P$ is left $\ol B$-flat, then $(P\Box^\CH Q)^{\co\mathcal G}=A$.
\end{lem}
\begin{proof}
    If either $\mathcal G$ is a flat left $C$-module, or $P$ is flat over $  \ol B$ on the left and faithfully flat over $A$ on the  left then the canonical map
    \[(P\Box^\CH Q)\diamond_C\mathcal G\to P\Box^\CH(Q\diamond_C\mathcal G)\]
    is an isomorphism (because the equalizer defining the cotensor product is preserved by the tensor product as in \Cref{exactness properties of cotensor product}, or the coequalizer defining the tensor product over $C$ is preserved by the cotensor product.) Therefore, the right $\mathcal G$-coaction on $P\Box^\CH Q$ is well-defined.
    
    The composition
   \[(P\Box^{\CH}Q)\ou A(P\Box^{\CH}Q)\xrightarrow\xi P\Box^{\CH}(Q\ou BQ)\xrightarrow{P\Box \rcan}P\Box^{\CH}(Q\diamond_C\mathcal G)\cong (P\Box^{\CH}Q)\diamond_C\mathcal G\]
   is the anti-right Galois map for $P\Box^{\CH}Q$, in the wider sense of \Cref{different coinvariants}; here $\xi$ is from \Cref{cotensor-monoidal}. Therefore $P\Box^\CH Q$ is an anti-Galois extension of its coinvariants.
   
     Finally, if $P$ is left  $\ol B$-flat then $(P\Box^{\CH}Q)^{\co\CH}=P\Box^{\CH}Q^{\co\CH}=P\Box^{\CH}B\cong P^{\co\CH}=A$ by the right version of  \Cref{coinvariants vs cotensor product}, using \Cref{exactness properties of cotensor product} for the first step.
\end{proof}

\begin{rem}\label{rem. left hand side version cotensor product galois extension}
Let $P$ be $\cL$-$\CH$ bicomodule algebra  and $\lcan \colon P\ot_{\BB}P\to \cL\diamond_{A}P$ bijective where $\cL$ is an $A$-Hopf algebroid and $\CH$ is a left  $B$ and $\BB$-faithfully flat $B$-Hopf algebroid. Let  $\overline{C}\subseteq Q$ be a left faithfully flat left $\CH$-Galois extension. If either $Q$ is left $B$ faithfully flat or $\cL$ is left $A$-faithfully flat, then $P\Box^{\CH}Q$ is a left $\cL$-Galois over its coinvariants. If ${}^{\co\cL}P=\BB$ and $Q$ is left $B$-flat, then ${}^{\co\CL}(P\Box^{\CH}Q)=\overline{C}$. 
\end{rem}

\begin{prop}\label{prop. inverse of bigalois extension}Let $\CH$ be a $B$-Hopf algebroid, faithfully flat as a left module over $B$ and $\ol B$.
    Let $N\subseteq P$ be a faithfully flat anti-right $\CH$-Galois extension, and $\CL=L(P,\CH)$ its left Ehresmann Hopf algebroid. Define $P\inv=(\CH\ot_{\ol B}P)^{\co\CH}$, we have
    \begin{enumerate}
        \item $P\inv$ is an algebra with multiplication induced by that of $\CH\ot P^\op$.\label{algstr of inv}
        \item $P\inv\cong (P\Box^\CH\CH^\op)^\op$, where $\CH^\op$ is the opposite algebra with the reverse comodule structures.
        \item \label{inverse is left galois}$P\inv$ is a left $\CH$-Galois extension of its coinvariants with the comodule structure induced by $\Delta\ou{\ol B}P$; if $P$ is left $\ol B$-flat then the coinvariants are $N$.
        \item Assume $\CH$ is right faithfully flat over $B$. Then $P\inv$ is a anti-right $\CL$-Galois extension of its coinvariants with the comodule structure induced by the reverse comodule structure of $P$.
        \item If $P$ is faithfully flat as left $\ol B$-module, then $P\Box^ \CH P\inv\cong \cL$ as $\CL$-bicomodule algebras.
        \item If $P$ is faithfully flat as left $\ol B$-module then $P\inv\Box^\CL P\cong\CH$ as $\CH$-bicomodule algebras. 
    \end{enumerate}
\end{prop}
\begin{proof}
  The isomorphism $P\inv\cong P\Box^\CH\CH^\op$ is a special case of \Cref{coinvariants vs cotensor product}. It is straightforward to check that the algebra structrue of $(P\Box^\CH\CH^\op)^\op$ thus corresponds to that claimed in \cref{algstr of inv}.

  By \Cref{cotensor product galois extension} (applied  to $Q=\CH^\op$), $P\Box^\CH\CH^\op$ is an anti-right $\CH$-Galois extension of its coinvariants with the right comodule structure induced from the reverse right comodule structure on $\CH^\op$. Therefore the opposite $P\inv$ is a left $\CH$-Galois extension of its coinvariants. If $P$ is left $\ol B$-flat then the coinvariants are 
  $({}^{\co\CH}(\CH\ou{\ol B}P))^{\co\CH}\cong ({}^{\co\CH}\CH\ou{\ol B}P)^{\co\CH}\cong(\BB\ou{\ol B}P)^{\co\CH}\cong P^{\co\CH}$.

  By \Cref{rem. left hand side version cotensor product galois extension}, applied to $Q=\CH^\op$,  we have that $P\Box^\CH\CH^\op$ is a left $\CL$-Galois extension of its coinvariants, and therefore its opposite is an anti-right $\CL$-Galois extension of its coinvariants under the reverse structure.

  If $P$ is left $\ol B$-flat then we have the isomorphism 
\begin{align*}
    (P\ot_{\BB}P)^{co\CH}\to&P\Box^{\CH}(\CH\ot_{\BB}P)^{co\CH}\\
    p\ot q\mapsto& p\z\ot p\o\ot q\\
    \overline{\varepsilon(X)} p\ot q\mapsfrom& p\ot X\ot q.
\end{align*}
It is left and right colinear over $\CL$ by \Cref{ehresmann regular comodule structures}. 

If $P$ is left $\ol B$-faithfully flat then by the coopposite version of \Cref{lem. Hopf Galois extension is HOpf biGalois extension}, $L(P,\CH)$ is faithfully flat over $N$ on the left. Also $\Box^\CL P$ is exact. Now
\begin{align*}
    \CH\cong& (\CH\ot_{\BB}\CH)^{\co\CH},\\
    X\mapsto& X_{+}\ot X_{-},
\end{align*}
and 
\begin{align*}   (\CH\ot_{\BB}\CH)^{co\CH}\cong& (\CH\ot_{\BB}P)^{\co\CH}\Box^{\CL} P\\
    X\ot Y\mapsto& X\ot Y\yi{}\ot Y\er{}.
\end{align*}
  The first isomorphism is clear. For the second isomorphism, we use  that   $(\CH\ot_{\BB}\CH)^{co\CH}\cong (\CH\ot_{\BB}P)^{\co\CH}\Box^{\CL} P$ as left $\CH$-comoudles because $\Box^\CL P$ is exact. The isomorphism is also right $\CH$-colinear by \Cref{equ. anti translation map 2}. More precisely, the right comodule structure on $(\CH\ot_{\BB}\CH)^{co\CH}$ is given by $\delta(X\ot Y)=X\ot Y_{[+]}\ot Y_{[-]}$. Therefore, we have
  \begin{align*}
      \delta\circ F(X\ot Y)=&X\ot Y\yi{}\ot Y\er{}\z\ot Y\er{}\z\\
      =&X\ot Y{}_{[+]}\yi{}\ot Y{}_{[+]}\er{}\ot Y{}_{[-]}\\
      =&(F\ot \id)\circ\delta (X\ot Y).
  \end{align*}
\end{proof}

\begin{thm}
The category whose objects are  Hopf algebroids and right faithfully flat over source and target algebra, and whose morphisms from $\CH$ to $\CL$ are $\underline\Bi(\CL,\CH)$, is a groupoid, with composition of morphisms given by cotensor product.
\end{thm}
\begin{proof}
    We have shown that right faithful flatness of a Hopf algebroid over its bases implies left faithful flatness in \Cref{rightff implies leftff}. Thus for an object $\CH$ of our groupoid we have in fact $\CH\in\Bi(\CH,\CH)$, which will be the unit morphism.

    Composition was studied in \Cref{cotensor product galois extension} and \Cref{rem. left hand side version cotensor product galois extension}. Under our extensive faithful flatness for $P\in\Bi(\CL,\CH)$ and $Q\in\Bi(\CH,\mathcal G)$ we can conclude that $P\Box^\CH\in\Bi(\CL,\mathcal G)$.

    Composition is associative, that is
    \[(P\Box^\CH Q)\Box^{\mathcal G}T\cong P\Box^\CH (Q\Box^{\mathcal G}T)\]
    for a third biGalois extension $T\in\Bi(\mathcal G,\mathcal J)$; this follows since $P\Box^\CH$ as well as $\Box^{\mathcal G}T$ are exact functors.

    Finally the inverse of $P\in\underline\Bi(\CL,\CH)$ is $P\inv$ as studied in \Cref{prop. inverse of bigalois extension}. The required faithful flatness conditions are satisfied because $P\inv$ is obtained itself as the opposite of a cotensor product of biGalois extensions with all required faithful flatness properties.
\end{proof}

\begin{rem}
    For a Hopf algebroid $\CL$ faithfully flat as a right module over both bases, the results of \Cref{free action of coinner auto group} can be phrased in terms of the groupoid of biGalois extensions; namely, $\operatorname{CoInn}(\CL)$ is isomorphic to a subgroup of $\underline{\Bi}(\CL,\CL)$, namely the subgroup of biGalois extensions isomorphic to $\CL$ as a right comodule algebra, with the left coaction twisted by an automorphism. Its free left action on $\underline{\Bi}(\CL,\CH)$ is an instance of the groupoid operation.
\end{rem}

\subsection{Co-Morita equivalence of Hopf algebroids}In this subsection, we are going to show a monoidal equivalence between two Hopf algebroids implies a Hopf BiGalois extension.

\begin{prop}\label{prop. comodule equivalence}
    There is a morphism from the groupoid of Hopf biGalois extensions to the groupoid of monoidal equivalences between comodule categories of Hopf algebroids. More precisely, $P\in\Bi(\cL,\CH)$ gives rise to a monoidal category equivalence
    \[\LComod{\CH}\ni V\mapsto P\Box^{\CH}V\in\LComod{\cL}.\]
\end{prop}
\begin{proof}
The inverse of the category equivalence is given by the inverse of $P$ in the groupoid above; the fact that the equivalences are monoidal was essentially checked in \Cref{cotensor-monoidal}. 
\end{proof}
\begin{rem}
    If $P\in\Bi(\CL,\CH)$ with $\CL$ an $N$-Hopf algebroid and $\CH$ a $B$-Hopf algebroid, then $P\inv\cong {}^{\co\CL}(P\ou N\CL)$ by switching sides in \Cref{prop. inverse of bigalois extension}. Thus the inverse of the functor in \Cref{prop. comodule equivalence} maps $W\in\LComod\CL$ to
    \[P\inv\Box^\CL W={}^{\co\CL}(P\ou N\CL)\Box^\CL W\cong {}^{\co\CL}(P\ou NW)\]
    with the $\CH$-comodule structure given by the reverse of that of $P$.
    This description might be useful because it is simpler than the one by cotensor product with $P\inv$ which is itself a coinvariant subspace.
\end{rem}

\begin{conj} \label{theconjecture}
        Let $\cL$ be a Hopf algebroid over $B$ and $\CH$ be a Hopf algebroid over $N$. If $(F,\xi):{}^{\CH}\CM\cong {}^{\CL}\CM$ is a monoidal equivalence then there is a $P\in \Bi(\CL,\CH)$ such that $P\Box^{\CH}-\cong F$.
\end{conj}

In fact in the sequel we will prove the conjecture under the added hypothesis that $F(\CH)$ is faithfully flat as a left and right $B$-module. 

We divide the proof into several steps.

\begin{prop}\label{prop. induced right H comodule}
     Let $\cL$ be a $B$-Hopf algebroid. If $(F,\xi):{}^{\CH}\CM\cong {}^{\CL}\CM$ is a monoidal functor preserving colimits, then $F(M)\in{}^{\CL}\CM^{\CH}$ for every $M\in {}^{\CH}\CM^{\CH}$.
\end{prop}

\begin{proof}
    By \cite[ Section 39.3]{BW}, we know there is a natural isomorphism $\Psi:F(-\diamond_{N}-)\to F(-)\diamond_{N}-:{}^{\CH}\CM\times {}_{N}\CM\to {}^{\CL}\CM$ and similarly $\Phi:F(-\diamond_{\overline{N}}-)\to -\diamond_{\overline{N}}F(-):{}_{\overline{N}}\CM\times{}^{\CH}\CM \to {}^{\CL}\CM$. We first observe that the right $\BB$-module structure on $F(M)$ is given by the composition of the following maps:
     \[\ra:F(M)\ot \overline{N}\cong F(M\ot \overline{N})\xrightarrow{F(\ra)}F(M),\]
     since $\ra:M\ot \overline{N}\to M$ is left $\CH$-colinear.
     Similarly, we can construct the left $\overline{N}$-action.      
     Now, we define the right $\CH$-coaction of $F(M)$ by the composition of the following maps:
     \[\delta^{\CH}:F(M)\xrightarrow{F(\delta)} F(M\diamond_ {N}\CH)\xrightarrow{\Psi_{M,\CH}}F(M)\diamond_{N}\CH.\]
    We can see the coaction is an $\overline{N}$-bimodule map by observing the following commuting diagram:
    \[\xymatrix{F(M)\ar[r]^-{F(\delta)}\ar[d]_{F(r_{\overline{n}})}&F(M\diamond_{N}\CH)\ar[r]^{\Psi_{M,\CH}}\ar[d]_{F(\id\ot r_{\overline{n}})} &F(M)\diamond_{N}\CH\ar[d]_{\id\ot r_{\overline{n}}}\\
F(M)\ar[r]^-{F(\delta)} &F(M\diamond_{N}\CH)\ar[r]^{\Psi_{M,\CH}}&F(M)\diamond_{N}\CH},\]
where $r_{\overline{n}}:M\to M$ is given by $m\mapsto m\overline{n}$ for any $m\in M$ (we can similarly define $r_{n}$ for later use), the left diagram commutes since $\Delta$ is right $\overline{N}$-linear and the right diagram commutes since $\Psi$ is natural on the second leg. Similarly, the coaction is also left $\overline{N}$-linear. The fact that the image of the coaction belongs to the Takeuchi product can be derived from the following diagram:
    \[\begin{tikzcd}
F(M)\arrow[r,"F(\delta)"]&F(M\times_{N}\CH)\arrow[r]\arrow[d,dashrightarrow]&F(M\diamond_{N}\CH)\arrow[r,shift left,"F(\id\diamond r_{n})"]
   \arrow[r,shift right,"F(r_{\overline{n}}\diamond \id)" ']\arrow[d,"\Psi_{M,\CH}"] &F(M\diamond_{N}\CH)\arrow[d,"\Psi_{M,\CH}"]\\
&F(M)\times_{N}\CH\arrow[r] &F(M)\diamond_{N}\CH\arrow[r,shift left,"F(r_{\overline{n}})\diamond \id"]
   \arrow[r,shift right,"F(\id)\diamond r_{n}" ']&F(M)\diamond_{N}\CH,
   \end{tikzcd}
   \]
where the right inner and outer diagrams commute by the naturality of $\Psi$. 
By the same method as in \cite[Section 39.7]{BW}, we can see the coaction on $F(M)$ is coassociative and counital. We can see $\delta^{\CH}$ is left $\cL$-colinear as $F(\delta)$ and $\Psi_{M,\CH}$ are left $\cL$-colinear.
\end{proof}

\begin{rem}
Notice that for every $M\in {}^{\CH}\CM^{\CH}$, there is a map $\Psi: F(M\Box^{\CH}\CH)\to F(M)\Box^{\CH}\CH $ induced by the following diagram:
 \[\begin{tikzcd}
&F(M\Box^{\CH}\CH)\arrow[r]\arrow[d,dashrightarrow]&F(M\diamond_{N}\CH)\arrow[r,shift right,"F(\id\diamond \Delta)"']
   \arrow[r,shift left,"F(\delta\diamond \id)" ]\arrow[d,"\Psi_{M,\CH}"]
&F(M\diamond_{N}\CH\diamond_{N}\CH)\arrow[d,shift right,"\Psi_{M,\CH\diamond_{N}\CH}"']\arrow[d,shift left,"(\Psi_{M,\CH}\diamond_{N}\id)\circ\Psi_{M\diamond_{N}\CH,\CH}"]\\
&F(M)\Box^{\CH}\CH\arrow[r] &F(M)\diamond_{N}\CH\arrow[r,shift right,"\delta^{\CH}\diamond \id"']
   \arrow[r,shift left,"F(\id)\diamond \Delta"]&F(M)\diamond_{N}\CH\diamond_{N}\CH,
   \end{tikzcd}
   \]
by the naturality if $\Psi$, it is not hard to see both the inner and outer diagram commute. As a result, we also have $\delta^{\CH}:F(M)\cong F(M)\Box^{\CH}\CH$.     
\end{rem}

In the following, we denote by $({}^{\CH}_{\overline{N}}\CM_{\overline{N}}, \ou{N^e})$ the monoidal category with objects being left $\CH$-comodules and $\overline{N}$-bimodules such that the comodule structure commutes with the $\overline{N}$-bimodule structure in the sense that $\delta(\overline{n} \,m\,\overline{n'})=m\mo\ot \overline{n} \,m\z\,\overline{n'}$ for any $m\in M\in {}^{\CH}_{\overline{N}}\CM_{\overline{N}}$. Similarly, we can also define the monoidal category $({}^{\CL}_{\overline{N}}\CM_{\overline{N}}, \ou{B\ot\overline{N}})$.

\begin{lem}
     Let $\CH$ be a Hopf algebroid over $N$ and $B$ an algebra.

     An exact monoidal functor 
     $(F,\xi)\colon \LComod\CH\to\BiMod B$ preserving colimits induces monoidal functors
     $(F,\hat\xi)\colon ({}^\CH_{\ol N}\CM_{\ol N},\ou{N^e}))\to (\BiMod{B\ot\ol N},\ou{B\ot\ol N})$ and $(F,\hat\xi)\colon ({}^\CH\CM^\CH.\ou{N^e})\to({}_B\CM^\CH_B,\ou{B\ot\ol N}).$
     
     If  in addition $\cL$ is a Hopf algebroid over $B$ and $(F,\xi):{}^{\CH}\CM\cong {}^{\CL}\CM$ is a monoidal equivalence then we have the monoidal equivalence $(F,\hat{\xi}):({}^{\CH}_{\overline{N}}\CM_{\overline{N}}, \ou{N^e})\cong ({}^{\CL}_{\overline{N}}\CM_{\overline{N}}, \ou{B\ot\overline{N}})$. Moreover, we also have the monoidal equivalence $(F,\hat{\xi}):({}^{\CH}\CM^{\CH}, \ou{N^e})\cong ({}^{\CL}\CM^{\CH}, \ou{B\ot\overline{N}})$.

 \end{lem}

\begin{proof}

    If $M\in {}^{\CH}_{\overline{N}}\CM_{\overline{N}}$ then $F(M)$ is also a $\overline{N}$-bimodule by the same construction as $F(H)$ in the above proposition. Also, the $\overline{N}$-bimodule structure commutes with the left $\cL$-comodule structure since the $\overline{N}$-action on $\CH$ is left $\CH$-colinear (therefore, it is also a left $\cL$-comodule map). As a result, $F(M)$ is an object in ${}^{\CL}_{\overline{N}}\CM_{\overline{N}}$. Given $M,N\in {}^{\CL}_{\overline{N}}\CM_{\overline{N}}$, we can define $\hat{\xi}_{M,M'}:F(M)\ou{B\ot\overline{N}}F(N)\to F(M\ot_{B^e}N)$ by the following diagram:
    \[\begin{tikzcd}
F(M)\ot_{B}B\ot\overline{N}\ot_{B}F(M')\arrow[r,shift left,"\ra\ot \id"]
   \arrow[r,shift right,"\id\ot\la" ']\arrow[d,"\eta"] &F(M)\ot_{B}F(M')\arrow[r]\arrow[d,"\xi_{M,M'}"]&F(M)\ou{B\ot\overline{N}}F(M')\arrow[d,dashrightarrow,"\hat{\xi}_{M,M'}"]\\
F(M\ot_{N}N^e\ot_{N}M')\arrow[r,shift left,"F(\id\ot \la)"]
   \arrow[r,shift right,"F(\ra\ot\id)" ']&F(M\ot_{B}M')\arrow[r]&F(M\ot_{B^e}M'),
   \end{tikzcd}
   \]
   where $\eta$ is defined by the composition of the following maps (by identifying $F(N^e)=F(N\ot \overline{N})\cong F(N)\ot \overline{N}=B\ot\overline{N}$)
 \[\eta:F(M)\ot_{B}F(N^e)\ot_{B}F(M')\xrightarrow{\xi\ot\id} F(M\ot_{N}N^e)\ot_{B}F(M')\xrightarrow{\xi}F(M\ot_{N}N^e\ot_{N}M').\]
 We only check the left outer diagram commute by observing the following commuting diagram:
 \[\begin{tikzcd}
F(M)\ot_{B}F(N^e)\ot_{B}F(M')
   \arrow[r,"\xi\ot\id"]\arrow[d,"\ra\ot\id"] &F(M\ot_{N}N^e)\ot_{B}F(M')\arrow[r,"\xi"]\arrow[d,"F(\ra)\ot \id"]&F(M\ot_{N}N^e\ot_{N}M')\arrow[d,"F(\ra\ot\id)"]\\
F(M)\ot_{B}F(M')\arrow[r,"="]&F(M)\ot_{B}F(M')\arrow[r,"\xi"]&F(M)\ot_{B}F(M'),
   \end{tikzcd}
   \]
 where the left diagram commutes by the definition of the right $\overline{N}$-action on $F(M)$ and `unital property' of the monoidal functor, the right diagram commutes by the naturality of $\xi$. As a result, $(F,\hat{\xi})$ is a monoidal functor since  $(F,\xi)$ is.
 
 Now, we show the monoidal equivalence $(F,\hat{\xi}):({}^{\CH}\CM^{\CH}, \ou{N^e})\cong ({}^{\CL}\CM^{\CH}, \ou{B\ot\overline{N}})$.
 By \Cref{prop. induced right H comodule}, we know $F(M)\in {}^{\CL}\CM^{\CH}$ for any $M\in {}^{\CH}\CM^{\CH}$. It is sufficient to show $\hat{\xi}$ is right $\CH$-colinear. Using \Cref{rem:codiagonal_as_composition} we can write the comodule structures on the source and target of $\hat\xi$ as compositions. Therefore, it is sufficient to show the following diagram commutes:
 \[\begin{tikzcd}
&\scriptscriptstyle{F(M)\Box\CH\ou{\scriptscriptstyle{B\ot\overline{N}}}F(M')\diamond_N\CH}
   \arrow[r,"\scriptscriptstyle{\tau_{2,3}}"]\arrow[d,"\Psi^{-1}\ot\Psi^{-1}"] &\scriptscriptstyle{F(M)\ou{\scriptscriptstyle{B\ot\overline{N}}}F(M')\diamond_{N}\CH\ou{N^e}\CH}\arrow[r,"\scriptscriptstyle{\id\ot m}"]\arrow[d,"\hat{\xi}\ot \id"]&\scriptscriptstyle{F(M)\ou{\scriptscriptstyle{B\ot\overline{N}}}F(M')\diamond_{N}\CH}\arrow[d,"\hat{\xi}\ot \id"]\\
\scriptscriptstyle{F(M)\ou{B\ot\overline{N}}F(M')}\arrow[r,"\scriptscriptstyle{F(\delta)\ot F(\delta)}"]\arrow[ru,"\delta^{\CH}\ot\delta^{\CH}"]\arrow[d,"\hat{\xi}"]&\scriptscriptstyle{F(M\Box\CH)\ou{\scriptscriptstyle{B\ot\overline{N}}}F(M'\diamond_N\CH)}\arrow[d,"\hat{\xi}"]&\scriptscriptstyle{F(M\ou{N^e}M')\diamond_{N}\CH\ou{N^e}\CH}\arrow[r,"\scriptscriptstyle{F(\id)\ot m}"]&\scriptscriptstyle{F(M\ou{N^e}M')\diamond_{N}\CH} \\
\scriptscriptstyle{F(M\ou{N^e}M')}\arrow[r,"\scriptscriptstyle{F(\delta\ot \delta)}"]&\scriptscriptstyle{F(M\Box\CH\ou{N^e}M'\diamond_N\CH)}\arrow[r,"\scriptscriptstyle{F(\tau_{2,3})}"]&\scriptscriptstyle{F(M\ou{N^e}M'\diamond_{N}\CH\ou{N^e}\CH)}\arrow[r,"\scriptscriptstyle{F(\id\ot m)}"]\arrow[u,"\Psi"]&\scriptscriptstyle{F(M\ou{N^e}M'\diamond_{N}\CH)}\arrow[u,"\Psi"],
   \end{tikzcd}
   \]
 where the two diagrams on the left clearly commute, the two diagrams on the right commute since $\hat{\xi}$ and $\Psi$ are natural transformations.  In the following, we will show that the middle diagram commutes as well by reducing it in several steps to a diagram already occurring in \cite{U87}.

   First, we find it as the outside hexagon of the following diagram:
 \[\begin{tikzcd}
\scriptscriptstyle{F(M)\Box\CH\ou{B\ot\overline{N}}F(M')\diamond_{N}\CH}\arrow[rrr,"\tau_{2,3}"]\arrow[dd,"\Psi^{-1}\ot\Psi^{-1}"]&&&\scriptscriptstyle{F(M)\ou{B\ot\overline{N}}F(M')\diamond_{N}\CH\ou{N^e}\CH}\arrow[dd,"\hat{\xi}\ot\id"]\\
&\scriptscriptstyle{F(M)\Box\CH\ou{B\ot\overline{N}}F(M')\ot\CH}\arrow[lu,"\pi"]\arrow[r,"\tau_{2,3}"]\arrow[d,"\Psi^{-1}\ot\hat{\Psi}^{-1}"]&\scriptscriptstyle{F(M)\ot_{\scriptscriptstyle{B}}F(M')\diamond_{N}\CH\ou{N^e}\CH}\arrow[d,"\xi\ot\id"]\arrow[ru,"\pi"]&\\
\scriptscriptstyle{F(M\Box\CH)\ou{B\ot\overline{N}}F(M'\diamond_{N}\CH)}\arrow[dd,"\hat{\xi}"]&\scriptscriptstyle{F(M\Box\CH)\ou{B\ot\overline{N}}F(M'\ot\CH)}\arrow[d,"\hat{\xi}"]\arrow[l,"\pi"]&\scriptscriptstyle{F(M\ot_{N}M')\diamond_{N}\CH\ou{N^e}\CH}\arrow[d,"\Psi^{-1}"]\arrow[r,"\pi"]&\scriptscriptstyle{F(M\ou{N^e}M')\diamond_{N}\CH\ou{N^e}\CH}\arrow[dd,"\Psi^{-1}"]\\
&\scriptscriptstyle{F(M\Box\CH\ou{N^e}M'\ot\CH)}\arrow[ld,"\pi"]\arrow[r,"\tau_{2,3}"]&\scriptscriptstyle{F(M\ot_{N}M'\diamond_{N}\CH\ou{N^e}\CH)}\arrow[rd,"\pi"]&\\
\scriptscriptstyle{F(M\Box\CH\ou{N^e}M'\diamond_{N}\CH)}\arrow[rrr,"\tau_{2,3}"]&&&\scriptscriptstyle{F(M\ou{N^e}M'\diamond_{N}\CH\ou{N^e}\CH)},
   \end{tikzcd}
   \]   
  Here, all the arrows connecting the inner to the outer hexagon are obvious surjections. We need surjectivity of the arrow terminating in the top left corner to deduce commutativity of the outside hexagon from that of the inside hexagon.  The latter in turn is the outside hexagon of the following:
\[\begin{tikzcd}
\scriptscriptstyle{F(M)\Box\CH\ou{B\ot\overline{N}}F(M')\ot\CH}\arrow[rrr,"\tau_{2,3}"]\arrow[rd]\arrow[dd,"\Psi^{-1}\ot\Psi^{-1}"]&&&\scriptscriptstyle{F(M)\ot_{\scriptscriptstyle{B}}F(M')\diamond_{N}\CH\ou{N^e}\CH}\arrow[dd,"\xi\ot\id"]\\
&\scriptscriptstyle{F(M)\diamond_{N}\CH\ou{B\ot\overline{N}}F(M')\ot\CH}\arrow[r,"\tau_{2,3}"]\arrow[d,"\Psi^{-1}\ot\hat{\Psi}^{-1}"]&\scriptscriptstyle{F(M)\ot_{\scriptscriptstyle{B}}F(M')\diamond_{N}\CH\ot_{\overline{N}}\CH}\arrow[d,"\xi\ot\id"]\arrow[ru,"\pi"]&\\
\scriptscriptstyle{F(M\Box\CH)\ou{B\ot\overline{N}}F(M'\ot\CH)}\arrow[r]\arrow[dd,"\hat{\xi}"]&\scriptscriptstyle{F(M\diamond_{N}\CH)\ou{B\ot\overline{N}}F(M'\ot\CH)}\arrow[d,"\hat{\xi}"]&\scriptscriptstyle{F(M\ot_{N}M')\diamond_{N}\CH\ot_{\overline{N}}\CH}\arrow[d,"\Psi^{-1}"]\arrow[r,"\pi"]&\scriptscriptstyle{F(M\ot_{N}M')\diamond_{N}\CH\ou{N^e}\CH}\arrow[dd,"\Psi^{-1}"]\\
&\scriptscriptstyle{F(M\diamond_{N}\CH\ou{N^e}M'\ot\CH)}\arrow[r,"\tau_{2,3}"]&\scriptscriptstyle{F(M\ot_{N}M'\diamond_{N}\CH\ot_{\overline{N}}\CH)}\arrow[rd,"\pi"]&\\
\scriptscriptstyle{F(M\Box\CH\ou{N^e}M'\ot\CH)}\arrow[rrr,"\tau_{2,3}"]\arrow[ru]&&&\scriptscriptstyle{F(M\ot_{N}M'\diamond_{N}\CH\ou{N^e}\CH)},
   \end{tikzcd}
   \]   

   So we need to verify commutativity of that inner hexagon, which is in turn the outer
hexagon of the following:

 \[\begin{tikzcd}
\scriptscriptstyle{F(M)\diamond_{N}\CH\ou{B\ot\overline{N}}F(M')\ot\CH}\arrow[rrr,"\tau_{2,3}"]\arrow[dd,"\Psi^{-1}\ot\hat{\Psi}^{-1}"]&&&\scriptscriptstyle{F(M)\ot_{B}F(M')\diamond_{N}\CH\ot_{\overline{N}}\CH}\arrow[dd,"\xi\ot\id"]\\
&\scriptscriptstyle{F(M)\ot\CH\ot_{B}F(M')\ot\CH}\arrow[lu,"\pi"]\arrow[r,"\tau_{2,3}"]\arrow[d,"\Psi^{-1}\ot\hat{\Psi}^{-1}"]&\scriptscriptstyle{F(M)\ot_{B}F(M')\ot\CH\ot\CH}\arrow[d,"\xi\ot\id"]\arrow[ru,"\pi"]&\\
\scriptscriptstyle{F(M\diamond_{N}\CH)\ou{B\ot\overline{N}}F(M'\ot\CH)}\arrow[dd,"\xi"]&\scriptscriptstyle{F(M\ot\CH)\ot_{\scriptscriptstyle{B}}F(M'\ot\CH)}\arrow[d,"\hat{\xi}"]\arrow[l,"\pi"]&\scriptscriptstyle{F(M\ot_{N}M')\ot\CH\ot\CH}\arrow[d,"\Psi^{-1}"]\arrow[r,"\pi"]&\scriptscriptstyle{F(M\ot_{N}M')\diamond_{N}\CH\ot_{\overline{N}}\CH}\arrow[dd,"\Psi^{-1}"]\\
&\scriptscriptstyle{F(M\ot\CH\ot_{N}M'\ot\CH)}\arrow[ld,"\pi"]\arrow[r,"\tau_{2,3}"]&\scriptscriptstyle{F(M\ot_{N}M'\ot\CH\ot\CH)}\arrow[rd,"\pi"]&\\
\scriptscriptstyle{F(M\diamond_{N}\CH\ou{N^e}M'\ot\CH)}\arrow[rrr,"\tau_{2,3}"]&&&\scriptscriptstyle{F(M\ou{N^e}M'\diamond_{N}\CH\ot_{\overline{N}}\CH)},
   \end{tikzcd}
   \]   

   Again, we use surjectivity of the arrow terminating in the top left corner to reduce this to commutativity of the inner hexagon. This in turn is again the outer hexagon of the following:
 \[\begin{tikzcd}
\scriptscriptstyle{F(M)\ot\CH\ot_{B}F(M')\ot\CH}\arrow[rrr,"\tau_{2,3}"]\arrow[d,"\Psi^{-1}\ot\hat{\Psi}^{-1}"]\arrow[rd,"\Psi^{-1}\ot\id"]&&&\scriptscriptstyle{F(M)\ot_{B}F(M')\ot\CH\ot\CH}\arrow[d,"\xi\ot\id"]\\
\scriptscriptstyle{F(M\ot\CH)\ot_{B}F(M'\ot\CH)}\arrow[dd,"\xi"]\arrow[r,"\id\ot\Psi"]&\scriptscriptstyle{F(M\ot\CH)\ot_{B}F(M')\ot\CH}\arrow[d,"\xi\ot\id"]&&\scriptscriptstyle{F(M\ot_{N}M')\ot\CH\ot\CH}\arrow[dd,"\Psi^{-1}"]\arrow[ld,"\Psi^{-1}"]\\
&\scriptscriptstyle{F(M\ot\CH\ot_{N}M')\ot\CH}\arrow[r,"F(\tau_{2,3})\ot\id"]&\scriptscriptstyle{F(M\ot_{N}M'\ot\CH)\ot\CH}\arrow[rd,"\Psi^{-1}"]&\\
\scriptscriptstyle{F(M\ot\CH\ot_{N}M'\ot\CH)}\arrow[ru,"\Psi"]\arrow[rrr,"F(\tau_{2,3})"]&&&\scriptscriptstyle{F(M\ot_{N}M'\ot\CH\ot\CH)},
   \end{tikzcd}
   \] 
where the lower left diagram is by the diagram (4) of \cite{U87},    the lower diagram is by the naturality of $\Psi$, the lower right diagram is by the diagram (3) of \cite{U87}. To see the top right diagram, it is sufficient to show
 \[\begin{tikzcd}
F(M)\ot\CH\ot_{B}F(M')\arrow[r,"\tau_{2,3}"]\arrow[d,"\Psi^{-1}\ot\id"]&F(M)\ot_{B}F(M')\ot\CH\arrow[d,"\xi\ot\id"]\\
F(M\ot\CH)\ot_{B}F(M')\arrow[d,"\xi"]&F(M\ot_{N}M')\ot \CH\arrow[d,"\Psi^{-1}"]\\
F(M\ot\CH\ot_{N}M')\arrow[r,"F(\tau_{2,3})"]&F(M\ot_{N}M'\ot \CH),
   \end{tikzcd}
   \] 
 where is clearly true by diagram (4') of \cite{U87}. Finally, we can see $(F,\hat{\xi})$ is a monoidal equivalence. Indeed, by the same method, we can show $F^{-1}$ (as the inverse funtor of the original functor $(F,\xi)$) is a monoidal functor from ${}^\CL_{\ol N}\CM_{\ol N}$ to ${}^\CH_{\ol N}\CM_{\ol N}$ and ${}^\CL\CM^\CH$ to ${}^\CH\CM^\CH$ respectively. 
\end{proof}

\begin{lem}
  Let $F\colon \LComod \CH\to\BiMod B $ be an exact monoidal functor preserving colimits. Then $P:=F(\CH)$ is a right $\CH$-comodule algebra, and there is a natural isomorphism $P\Box^\CH V\to F(V)$ making
  \[\begin{tikzcd}
      (P\Box^\CH V)\ou B (P\Box^\CH W)\arrow[r,"\xi_{P}"]\arrow[d]&P\Box^\CH(V\ou NW)\arrow[d]\\
      F(V)\ou B F(W)\arrow[r,"\xi"]&F(V\ou NW)
  \end{tikzcd}\]
  commute, where the top arrow is induced by the multiplication on $P$ as in \Cref{contensor-monoidal-converse}. In particular, $P$ is a right $\CH$-anti-Galois extension. 

  If $F$ factors over the forgetful functor $\LComod\CL\to\BiMod B$ for a bialgebroid $\CL$ over $B$ then $P$ is a $\CL$-$\CH$-bicomodule algebra.
\end{lem}
\begin{proof}
    $\CH$ is an algebra in $^\CH\CM^\CH$, and therefore mapped to an algebra in $_B\CM^\CH_B$ by the previous lemma.  If the functor $F$ factors over $\LComod\CL$ then by the previous lemma again the algebra $\CH$ is mapped to an algebra in $^\CL\CM^\CH$.
    First, the map $P\Box V\to F(V)$ is given by the composition of the following map of isomorphisms:
    \[F(\CH)\Box V\xrightarrow{\Psi^{-1}}F(\CH\Box V)\xrightarrow{\cong}F(V).\]
    Second, we define multiplication on $P$ by
    \[F(\CH)\ou{B\ot \ol N}F(\CH)\xrightarrow{\hat\xi} F(\CH\ou{N^e}\CH)\xrightarrow{F(m)} F(\CH)\]
    which is right $\CH$-colinear because multiplication of $\CH$ is (on the tensor product over $N^e$). The diagram of this lemma turns out to be the outer diagram of the following
  \[\begin{tikzcd}
\scriptscriptstyle{F(\CH)\Box \,V\ot_{\scriptscriptstyle{B}}F(\CH)\Box\,W}\arrow[d,"\Psi^{-1}\ot\Psi^{-1}"]\arrow[r,"\tau_{2,3}"]&\scriptscriptstyle{F(\CH)\ou{B\ot \ol N}F(\CH)\Box\,V\ot_{N}W}\arrow[r,"\hat{\xi}\ot\id"]&\scriptscriptstyle{F(\CH\ou{N^e}\CH)\Box\,V\ot_{N}W}\arrow[r,"F(m)\ot\id"]\arrow[d,"\Psi^{-1}"]&\scriptscriptstyle{F(\CH)\Box\,V\ot_{N}W}\arrow[d,"\Psi^{-1}"]\\
\scriptscriptstyle{F(\CH\Box \,V)\ot_{\scriptscriptstyle{B}}F(\CH\Box\,W)}\arrow[r,"\xi"]\arrow[d,"\cong\ot\cong"]&\scriptscriptstyle{F(\CH\Box \,V\ou{N}\CH\Box\,W)}\arrow[d,"F(\cong\ot\cong)"]\arrow[r,"F(\tau_{2,3})"]&\scriptscriptstyle{F(\CH\ou{N^e}\CH\Box\,V\ot_{N}W)}\arrow[r,"F(m\ot\id)"]&\scriptscriptstyle{F(\CH\Box\,V\ot_{N}W)}\arrow[lld]\\
\scriptscriptstyle{F(V)\ot_{B}F(W)}\arrow[r,"\xi"]&\scriptscriptstyle{F(V\ot_{N}W)},
   \end{tikzcd}
   \]       
where the two lower diagrams and right diagram clearly commute, so we only need to verify the top left diagram, which is in turn the outer hexagon of the following:

     \[\begin{tikzcd}
\scriptscriptstyle{F(\CH)\Box \,V\ot_{\scriptscriptstyle{B}}F(\CH)\Box\,W}\arrow[rrr,"\tau_{2,3}"]\arrow[dd,"\Psi^{-1}\ot\Psi^{-1}"]\arrow[rd]&&&\scriptscriptstyle{F(\CH)\ou{B\ot\overline{N}}F(\CH)\Box\,V\ot_{N}\,W}\arrow[dd,"\hat{\xi}\ot\id"]\arrow[ld]\\
&\scriptscriptstyle{F(\CH)\Box\,V\ot_{\scriptscriptstyle{B}}F(\CH)\diamond_{N}\,W}\arrow[r,"\tau_{2,3}"]\arrow[d,"\Psi^{-1}\ot\Psi^{-1}"]&\scriptscriptstyle{F(\CH)\ou{B\ot\overline{N}}F(\CH)\diamond_{N}V\ot_{N}\,W}\arrow[d,"\hat{\xi}\ot\id"]&\\
\scriptscriptstyle{F(\CH\Box \,V)\ot_{\scriptscriptstyle{B}}F(\CH\Box\,W)}\arrow[dd,"\xi"]\arrow[r]&\scriptscriptstyle{F(\CH\Box \,V)\ot_{\scriptscriptstyle{B}}F(\CH\diamond_{N}\,W)}\arrow[d,"\xi"]&\scriptscriptstyle{F(\CH\ou{N^e}\CH)\diamond_{N}V\ot_{N} W}\arrow[d,"\Psi^{-1}"]&\scriptscriptstyle{F(\CH\ou{N^e}\CH)\Box\,V\ot_{N} W}\arrow[l]\arrow[dd,"\Psi^{-1}"]\\
&\scriptscriptstyle{F(\CH\Box \,V\ot_{N}\CH\diamond_{N}\,W)}\arrow[r,"\tau_{2,3}"]&\scriptscriptstyle{F(\CH\ou{N^e}\CH\diamond_{N}V\ot_{N} W)}&\\
\scriptscriptstyle{F(\CH\Box \,V\ot_{N}\CH\Box\,W)}\arrow[ru]\arrow[rrr,"\tau_{2,3}"]&&&\scriptscriptstyle{F(\CH\ou{N^e}\CH\Box\,V\ot_{N} W)}\arrow[lu],
   \end{tikzcd}
   \]   
where we
need injectivity of the arrow terminating in the lower right corner to deduce commutativity
of the outside hexagon from that of the inside hexagon. So we need to verify commutativity of that inner hexagon, which is in turn the outer
hexagon of the following:

     \[\begin{tikzcd}
\scriptscriptstyle{F(\CH)\Box \,V\ot_{\scriptscriptstyle{B}}F(\CH)\diamond_{N}\,W}\arrow[rrr,"\tau_{2,3}"]\arrow[dd,"\Psi^{-1}\ot\Psi^{-1}"]&&&\scriptscriptstyle{F(\CH)\ou{B\ot\overline{N}}F(\CH)\diamond_{N}\,V\ot_{N}\,W}\arrow[dd,"\hat{\xi}\ot\id"]\\
&\scriptscriptstyle{F(\CH)\Box\,V\ot_{\scriptscriptstyle{B}}F(\CH)\ot\,W}\arrow[r,"\tau_{2,3}"]\arrow[d,"\Psi^{-1}\ot\hat{\Psi}^{-1}"]\arrow[lu,"\pi"]&\scriptscriptstyle{F(\CH)\ot_{\scriptscriptstyle{B}}F(\CH)\diamond_{N}V\ot\,W}\arrow[d,"\xi\ot\id"]\arrow[ru,"\pi"]&\\
\scriptscriptstyle{F(\CH\Box \,V)\ot_{\scriptscriptstyle{B}}F(\CH\diamond_{N}\,W)}\arrow[dd,"\xi"]&\scriptscriptstyle{F(\CH\Box \,V)\ot_{\scriptscriptstyle{B}}F(\CH\ot\,W)}\arrow[d,"\xi"]\arrow[l,"\pi"]&\scriptscriptstyle{F(\CH\ot_{N}\CH)\diamond_{N}V\ot W}\arrow[d,"\Psi^{-1}"]\arrow[r,"\pi"]&\scriptscriptstyle{F(\CH\ou{N^e}\CH)\diamond_{N}V\ot_{N} W}\arrow[dd,"\Psi^{-1}"]\\
&\scriptscriptstyle{F(\CH\Box \,V\ot_{N}\CH\ot\,W)}\arrow[ld,"\pi"]\arrow[r,"\tau_{2,3}"]&\scriptscriptstyle{F(\CH\ot_{N}\CH\diamond_{N}V\ot W)}\arrow[rd,"\pi"]&\\
\scriptscriptstyle{F(\CH\Box \,V\ot_{N}\CH\diamond_{N}\,W)}\arrow[rrr,"\tau_{2,3}"]&&&\scriptscriptstyle{F(\CH\ou{N^e}\CH\diamond_{N}V\ot_{N} W)}.
   \end{tikzcd}
   \]   
Here, all the arrows connecting the inner to the outer hexagon are obvious surjections. We need surjectivity of the arrow terminating in the top left corner to deduce commutativity of the outside hexagon from that of the inside hexagon.   The latter in turn is the outside hexagon of the following:
\[\begin{tikzcd}
\scriptscriptstyle{F(\CH)\Box \,V\ot_{\scriptscriptstyle{B}}F(\CH)\ot\,W}\arrow[rr,"\tau_{2,3}"]\arrow[dd,"\Psi^{-1}\ot\Psi^{-1}"]\arrow[rd]&&\scriptscriptstyle{F(\CH)\ot_{\scriptscriptstyle{B}}F(\CH)\diamond_{N}\,V\ot\,W}\arrow[dd,"\xi\ot\id"]\\
&\scriptscriptstyle{F(\CH)\diamond_{N}\,V\ot_{\scriptscriptstyle{B}}F(\CH)\ot\,W}\arrow[d,"\Psi^{-1}\ot\Psi^{-1}"]\arrow[ru,"\tau_{2,3}"]&\\
\scriptscriptstyle{F(\CH\Box\,V)\ot_{\scriptscriptstyle{B}}F(\CH\ot\,W)}\arrow[r]\arrow[dd,"\xi"]&\scriptscriptstyle{F(\CH\diamond_{N}\,V)\ot_{\scriptscriptstyle{B}}F(\CH\ot\,W)}\arrow[d,"\xi"]&\scriptscriptstyle{F(\CH\ot_{N}\CH)\diamond_{N}V\ot\,W}\arrow[dd,"\Psi^{-1}"]\\
&\scriptscriptstyle{F(\CH\diamond_{N}\,V\ot_{N}\CH\ot\,W)}\arrow[rd,"F(\tau_{2,3})"]&\\
\scriptscriptstyle{F(\CH\Box\,V\ot_{N}\CH\ot\,W)}\arrow[rr,"F(\tau_{2,3})"]\arrow[ru]&&\scriptscriptstyle{F(\CH\ot_{N}\CH\diamond_{N}V\ot\,W)}.
   \end{tikzcd}
   \]  
  So we need to verify commutativity of that right hexagon, which is in turn the outer
hexagon of the following:
  \[\begin{tikzcd}
\scriptscriptstyle{F(\CH)\diamond_{N} \,V\ot_{\scriptscriptstyle{B}}F(\CH)\ot\,W}\arrow[rrr,"\tau_{2,3}"]\arrow[dd,"\Psi^{-1}\ot\Psi^{-1}"]&&&\scriptscriptstyle{F(\CH)\ot_{\scriptscriptstyle{B}}F(\CH)\diamond_{N}\,V\ot\,W}\arrow[dd,"\xi\ot\id"]\\
&\scriptscriptstyle{F(\CH)\ot\,V\ot_{\scriptscriptstyle{B}}F(\CH)\ot\,W}\arrow[r,"\tau_{2,3}"]\arrow[d,"\Psi^{-1}\ot\Psi^{-1}"]\arrow[lu,"\pi"]&\scriptscriptstyle{F(\CH)\ot_{\scriptscriptstyle{B}}F(\CH)\ot\,V\ot\,W}\arrow[d,"\xi\ot\id"]\arrow[ru,"\pi"]&\\
\scriptscriptstyle{F(\CH\diamond_{N} \,V)\ot_{\scriptscriptstyle{B}}F(\CH\ot\,W)}\arrow[dd,"\xi"]&\scriptscriptstyle{F(\CH\ot \,V)\ot_{\scriptscriptstyle{B}}F(\CH\ot\,W)}\arrow[d,"\xi"]\arrow[l,"\pi"]&\scriptscriptstyle{F(\CH\ot_{N}\CH)\ot\,V\ot W}\arrow[d,"\Psi^{-1}"]\arrow[r,"\pi"]&\scriptscriptstyle{F(\CH\ot_{N}\CH)\diamond_{N}V\ot W}\arrow[dd,"\Psi^{-1}"]\\
&\scriptscriptstyle{F(\CH\ot \,V\ot_{N}\CH\ot\,W)}\arrow[ld,"\pi"]\arrow[r,"\tau_{2,3}"]&\scriptscriptstyle{F(\CH\ot_{N}\CH\ot\,V\ot W)}\arrow[rd,"\pi"]&\\
\scriptscriptstyle{F(\CH\diamond_{N} \,V\ot_{N}\CH\ot\,W)}\arrow[rrr,"\tau_{2,3}"]&&&\scriptscriptstyle{F(\CH\ot_{N}\CH\diamond_{N}V\ot\, W)}.
   \end{tikzcd}
   \]   
Again, we use surjectivity of the arrow terminating in the top left corner to reduce this to commutativity of the inner hexagon.

This in turn is again the outer hexagon of the following:   
\[\begin{tikzcd}
\scriptscriptstyle{F(\CH)\ot\,V\ot_{B}F(\CH)\ot\,W}\arrow[rrr,"\tau_{2,3}"]\arrow[d,"\Psi^{-1}\ot\Psi^{-1}"]\arrow[rd,"\Psi^{-1}\ot\id"]&&&\scriptscriptstyle{F(\CH)\ot_{B}F(\CH)\ot\,V\ot\,W}\arrow[d,"\xi\ot\id"]\\
\scriptscriptstyle{F(\CH\ot\,V)\ot_{B}F(\CH\ot\,W)}\arrow[dd,"\xi"]\arrow[r,"\id\ot\Psi"]&\scriptscriptstyle{F(\CH\ot\,V)\ot_{B}F(\CH)\ot\,W}\arrow[d,"\xi\ot\id"]&&\scriptscriptstyle{F(\CH\ot_{N}\CH)\ot\,V\ot\,W}\arrow[dd,"\Psi^{-1}"]\arrow[ld,"\Psi^{-1}"]\\
&\scriptscriptstyle{F(\CH\ot\,V\ot_{N}\CH)\ot\,W}\arrow[r,"F(\tau_{2,3})\ot\id"]&\scriptscriptstyle{F(\CH\ot_{N}\CH\ot\,V)\ot\,W}\arrow[rd,"\Psi^{-1}"]&\\
\scriptscriptstyle{F(\CH\ot\,V\ot_{N}\CH\ot\,W)}\arrow[ru,"\Psi"]\arrow[rrr,"F(\tau_{2,3})"]&&&\scriptscriptstyle{F(\CH\ot_{N}\CH\ot\,W\ot\,W)},
   \end{tikzcd}
   \] 
where the lower left diagram is by the diagram (4) of \cite{U87},    the lower diagram is by the naturality of $\Psi$, the lower right diagram is by the diagram (3) of \cite{U87}. To see the top right diagram, it is sufficient to show
\[\begin{tikzcd}
F(\CH)\ot\,V\ot_{B}F(\CH)\arrow[r,"\tau_{2,3}"]\arrow[d,"\Psi^{-1}\ot\id"]&F(\CH)\ot_{B}F(\CH)\ot\,V\arrow[d,"\xi\ot\id"]\\
F(\CH\ot\,V)\ot_{B}F(\CH)\arrow[d,"\xi"]&F(\CH\ot_{N}\CH)\ot \,V\arrow[d,"\Psi^{-1}"]\\
F(\CH\ot\,V\ot_{N}\CH)\arrow[r,"F(\tau_{2,3})"]&F(\CH\ot_{N}\CH\ot \,V),
   \end{tikzcd}
   \] 
  where is clearly true by diagram (4') of \cite{U87}.  As a result, the diagram of the Lemma commute, which indicates $\xi_{P}$ is bijective. Therefore, $P$ is an anti-right $\CH$-Galois extension by \Cref{contensor-monoidal-converse}.
\end{proof}

\begin{rem}
    We fall just short of being able to prove \Cref{theconjecture}. Namely, we are unable to prove at this point that the Galois extension and bicomodule algebra $P $ constructed above from a monoidal functor, resp. monoidal equivalence, is fathfully flat as a module over its coinvariant subalgebra $B$.

    If we assume that $F\colon\LComod\CH\to\LComod\CL$ is a monoidal equivalence, and that $P=F(\CH)$ above is in fact faithfully flat as left and right $B$-module, then $P$ is a bi-Galois extension with the Ehresmann Hopf algebroid $L(P,\CH)$. coacting on the left. Thus $P$ induces an equivalence $\LComod \CH\to\LComod{L(P,\CH)}$ isomorphic to $F$ as a functor to $B$-bimodules, implying that $\CL\cong L(P,\CH)$.
    \end{rem}

\section{Structure Theorems for Hopf modules}\label{sec:structhms}
\begin{thm}\label{thm. structure theorem for right ES bialgebroid}
     Let $\cL$ be a left $B$-bialgebroid and $N\subseteq P$ be a  regular left faithfully flat  anti-right $\cL$-Hopf Galois extension. Then we have the monoidal category equivalences
     \[({}_{P}\mathcal{M}^{\cL}_{P},\ot_{P})\simeq({}_{L(P,\cL)}\mathcal{M},\diamond_{N}).\]
\end{thm}
\begin{proof}
      The equivalence can be given via
    \begin{align}
        {}_{P}\mathcal{M}^{\cL}_{P}\to {}_{L(P,\cL)}\mathcal{M},&\qquad M\mapsto M^{co\cL}\\
        {}_{L(P,\cL)}\mathcal{M}\to{}_{P}\mathcal{M}^{\cL}_{P},&\qquad \Lambda\mapsto \Lambda\ot_{N}P.\label{leftLmod to Hopf bimod}
    \end{align}
    More precisely,  for any $p\ot q\in L:=L(P,\cL)$ and $\eta\in M^{co\cL}$,
    the left $L$-module structure is be given by
    \[(p\ot_{\BB}q)\bla \eta=p\eta q.\]
   For any $\Lambda\in {}_{L(P,\cL)}\mathcal{M}$, $\Lambda\ot_{N}P$ has the $P$-bimodule structure
   \begin{align*}
    p(\eta\ot_{N}q)p'=p\rz \eta p\rmo\yi{}\ot_{N}p\rmo\er{} qp'=p\mo\bla\eta\ot p\z qp',
\end{align*}
and right $\cL$-comodule structure on the right factor. The isomorphism $\Lambda\simeq (\Lambda\ot_{N}P)^{co\cL}$ for any $\Lambda\in {}_{L(P,\cL)}\mathcal{M}$ can be given by $\eta\mapsto \eta\ot 1$. The isomorphism $M\simeq M^{co\cL}\ot_{N}P$ for any $M\in {}_{P}\M^{\cL}_{P}$ can be given by $m\mapsto m\rz \yi{m\rmo}\ot_{N}\er{m\rmo}$ with inverse $\eta\ot_{N}p\mapsto \eta p$.  

The monoidal functor structure of \eqref{leftLmod to Hopf bimod} is given by the obvious isomorphism 
\begin{align}
    \label{mon fun str hopf bim}
    (\Lambda\ou NP)\ou P(\Lambda'\ou NP)\to (\Lambda\ou N\Lambda')\ou NP.
\end{align}
It is immediate to check that it is a morphism in the appropriate category, and it is obviously coherent.
\end{proof}

\begin{thm}\label{thm. structure theorem of ES bialgebroid}
     Let $\cL$ be a left $B$-bialgebroid and $N\subseteq P$ be a regular left faithfully flat left $\cL$-Hopf Galois extension, then we have the monoidal category equivalences
     \[({}^{\cL}_{P}\mathcal{M}_{P},\ot_{P})\simeq({}_{R(P,\cL)^{cop}}\mathcal{M},\diamond_{N})\simeq({}_{R(P,\cL)}\mathcal{M},\diamond_{\overline{N}})^{\sym}.\]
\end{thm}
\begin{proof}
    The equivalence can be given via
    \begin{align}
        {}^{\cL}_{P}\mathcal{M}_{P}\to {}_{R(P,\cL)^{cop}}\mathcal{M},&\qquad M\mapsto {}^{co\cL}M\\
        {}_{R(P,\cL)^{cop}}\mathcal{M}\to {}^{\cL}_{P}\mathcal{M}_{P},&\qquad \Lambda\mapsto \Lambda \ot_{N} P.
    \end{align}
    Given $M\in  {}^{\cL}_{P}\mathcal{M}_{P}$, ${}^{co\cL}M$ is a left $R(P,\cL)$-module with the left action given by
    \[(p\ot_{B}q)\bla m=pmq,\]
    for any $p\ot_{B}q\in R(P,\cL)$ and $m\in {}^{co\cL}M$. We can see $bm=mb$ for any $b\in B$ and $m\in  {}^{co\cL}M$ as both sides have the same images of the left coaction. Moreover, for any $p\ot_{B}q\in R(P,\cL)$ and $m\in {}^{co\cL}M$, $pmq\in {}^{co\cL}M$. So the left action is well defined, which is clearly associative and unital. Also, ${}^{co\cL}M$ is a $N$-bimodule with the structure given by $n\la m\ra n'=nmn'$ which is also a $\overline{N}$-bimodule with $\overline{n'}\la m\ra \overline{n}=nmn'$. Conversely, given $\Lambda\in {}_{R(P,\cL)^{cop}}\mathcal{M}$, $ \Lambda \ot_{N} P \in {}^{\cL}_{P}\mathcal{M}_{P}$, with the left $\cL$-coaction given by the right factor and the $P$-bimodule structure given by
    \[p\,\la \,(\eta\ot_N q)\, \ra \,p'=p\rz \eta p\ro\tuno{} \ot_N p\ro\tdue{} q p'=p\o\bla\eta\ot p\z qp' , \]
    for any $\eta\ot_N q\in \Lambda \ot_{N} P$.
    
    \end{proof}

\begin{defi}\cite[Def.4.2]{schau1}\label{def. left YD}
    Let $\cL$ be a left bialgebroid over $B$, a left-left Yetter-Drinfeld module of $\cL$ is a left $\cL$-comodule  and a left $\cL$-module $\Lambda$, such that 
    \begin{itemize}
        \item $s_{L}(b)\bla \rho=b\rho$ and $t_{L}(b)\bla\rho=\rho b$, $\forall b\in B, \rho\in \Lambda$.
        \item $(X\o \bla\rho)\mo X\t\di{}(X\o \bla\rho)\z=X\o\rho\mo\di{}X\t\bla \rho\z$, $\forall X\in \cL, \rho\in \Lambda$.
    \end{itemize}
    We denote the category of left-left Yetter-Drinfeld modules of $\cL$ by ${}_{\cL}^{\cL}\mathcal{YD}$.
\end{defi}

We define right-left Yetter-Drinfeld modules so that they are left-left ones over the coopposite bialgebroid.

\begin{defi}\label{def. right-left YD}
    Let $\cL$ be a left bialgebroid over $B$, a right-left Yetter-Drinfeld module of $\cL$ is a right $\cL$-comodule  and a left $\cL$-module $\Lambda$, such that 
    \begin{itemize}
        \item $s_{L}(b)\bla \rho=\rho\overline{b}$ and $t_{L}(b)\bla \rho=\overline{b}\rho$, $\forall b\in B, \rho\in \Lambda$.
        \item $(X\t\bla \rho)\z\di{}(X\t\bla \rho)\o X\o=X\o\bla \rho\z\di{}X\t \rho\o$.
    \end{itemize}
    The category of right-left Yetter-Drinfeld modules of $\cL$ is denoted by ${}_{\cL}\mathcal{YD}^{\cL}$.
\end{defi}

Recall that if $\cL$ is an anti-left Hopf algebroid, then every left comodule of $\cL$ is skew regular, so we can generalize the results in \cite{XH23} (where a full Hopf algebroid is assumed) and \cite{schau1}[Prop.4.4] (where the center is replaced by the weak center):
\begin{lem}
    If $\cL$ is an anti-left Hopf algebroid and left flat over $B$ and $\BB$,  then $({}_{\cL}^{\cL}\mathcal{YD},  \di{})$ is the Drinfeld center of $\LMod\cL$, and in particular a is a braided monoidal category. The half-braiding and its inverse are given by
    \begin{align}
        \sigma(v\di{}w)=&v\mo\bla w\di{}v\z\\
        \sigma^{-1}(w\di{}v)=&v\rz\di{}v\ro\bla w,
    \end{align}
   for any a Yetter-Drinfeld module $V\in{}_{\cL}^{\cL}\mathcal{YD}$, module $W\in\LMod \cL$, and $v\in V,w\in W$.
\end{lem}
The only novelty is that the inverse of the half-braiding exists and is given by the reverse comodule structure. Similarly, we have

\begin{lem}
     If $\cL$ is a left Hopf algebroid left flat over $B$ and $\BB$, then $({}_{\cL}\mathcal{YD}^{\cL},  \di{})$ is the Drinfeld center of $\LMod \cL$; the half-braiding and its inverse are given by
     \begin{align}
        \tilde{\sigma}(v\di{}w)=&w\z\di{}w\o\bla v\\
        \tilde{\sigma}^{-1}(w\di{}v)=&w\rmo\bla v\di{}w\rz,
    \end{align}
    for $w\in W\in{}_{\cL}\mathcal{YD}^{\cL}$ and  $v\in V\in \LMod\cL$.
\end{lem}
It may be worth noting that in this case the underlying functor from Yetter-Drinfeld modules to comodules reverses tensor products. We also note that from the two previous results it is obvious that in case that $\cL$ is Hopf and anti-Hopf and suitably flat, there is a bijection between left-left and right-left Yetter-Drinfeld module structures given by taking reverse comodule structures. We will omit the proof, as the proof is similar to \cite{XH23} and \cite{schau4}.

\begin{thm}
     Let $L$ be a $N$-Hopf algebroid and $\cL$ be a $B$-Hopf algebroid, if $P\in\Bi(L, \cL)$, then we have the monoidal category equivalences
     \[({}_{L}^{L}\mathcal{YD}, \ot_{N})\simeq ({}^{L}_{P}\mathcal{M}^{\cL}_{P}, \ot_{P})\simeq ({}_{\cL}\mathcal{YD}^{\cL}, \di{})^\sym.\]
\end{thm}

\begin{proof}
    The first equivalence can be given via
    \begin{align}
        {}^{L}_{P}\mathcal{M}^{\cL}_{P}\to {}_{L}^{L}\mathcal{YD},&\qquad M\mapsto M^{co\cL}\label{equ. LL-YD to Hopf bimodule}\\
        {}_{L}^{L}\mathcal{YD}\to{}^{L}_{P}\mathcal{M}^{\cL}_{P},&\qquad \Lambda\mapsto \Lambda\ot_{N}P.\label{equ. Hopf bimodule to LL_YD}
    \end{align}
    We already know these assignments define an equivalence $\LMod L\to{_P\M^\cL_P}$ from \Cref{thm. structure theorem for right ES bialgebroid}. By the identification of $L$ with $L(P,\cL)$ the action of $L$ on $M^{\co\cL}$ is given by 
    
     \[\alpha\bla \eta=\tuno{\alpha}\eta\tdue{\alpha},\]
    for any $\alpha\in L$ and $\eta\in M^{co\cL}$.
If $M\in{^L_P\M^\cL_P}$ then the additional left $L$-comodule structure restricts to a left $L$-comodule structure on $M^{\co\cL}$. Now, we can check the Yetter-Drinfeld conditions on $M^{co\cL}$. First, we can see $s_{L}(n)\bla \eta=n\eta$ and $t_{L}(n)\bla \eta=\eta n$ for any $n\in N$ and $\eta\in M^{co\cL}$. Second, for any $\alpha\in L$.
    \begin{align*}
        (\alpha\o\bla\eta)\mo \alpha\t\ot_{N}(\alpha\o\bla\eta)\z
        =&\alpha\o\tuno{}\mo\eta\mo\alpha\o\tdue{}\mo \alpha\t\ot \alpha\o\tuno{}\z\eta\z\alpha\o\tdue{}\z\\    
        =&\alpha\o{}_{+}\tuno{}\mo\eta\mo\alpha\o{}_{-} \alpha\t\ot \alpha\o{}_{+}\tuno{}\z\eta\z\alpha\o{}_{+}\tdue{}\\  
        =&\alpha\tuno{}\mo\eta\mo\ot \alpha\tuno{}\z\eta\z\alpha\tdue{}\\  
        =&\alpha\o\eta\mo\ot \alpha\t\tuno{}\eta\z \alpha\t\tdue{}\\
        =&\alpha\o \eta\mo\ot \alpha\t\bla \eta\z,
        \end{align*}
where the 2nd step uses \eqref{equ. translation map 2}, the 4th step uses \eqref{equ. translation map 1}. 

For the inverse equivalence \eqref{equ. Hopf bimodule to LL_YD}, we already know from \Cref{thm. structure theorem for right ES bialgebroid} that the $P$-bimodule structure on $\Lambda\ou NP$ is given by \begin{align*}
    p(\eta\ot_{N}q)p'=p\mo\bla \eta\ot_{N}p\z qp'.
\end{align*} 
We define the left $L$-comodule structure as the  codiagnal coaction on $\Lambda\ot_{N}P$.
  We know that $\Lambda\ot_{N}P\in {}_{P}\mathcal{M}_{P}^{\cL}$, and  it is obviously a bicomodule. So we only need to check $\Lambda\ot_{N}P\in {}_{P}^{L}\mathcal{M}_{P}$,
    \begin{align*}
        {}_{L}\delta(p(\eta\ot q)p')=&(p\mt\bla \eta)\mo p\mo q\mo p'\mo\ot (p\mt\bla \eta)\z\ot p\z q\z p'\z\\
        =&p\mt \eta\mo  q\mo p'\mo\ot p\mo\bla \eta\z\ot p\z q\z p'\z.
    \end{align*}
    It is easy to check that the adjunction morphsims in the category equivalence from \Cref{thm. structure theorem for right ES bialgebroid} respect the additional $L$-structures. As for the monoidal functor structure, we only have to add that the isomorphism \eqref{mon fun str hopf bim} is also an $L$-comodule map in our present situation.

      The second equivalence can be given via the coopposite version of the first
    \begin{align}
        {}^{L}_{P}\mathcal{M}^{\cL}_{P}\to {}_{\cL}\mathcal{YD}^{\cL},&\qquad M\mapsto {}^{coL}M\label{equ. RL-YD to Hopf bimodule}\\
        {}_{\cL}\mathcal{YD}^{\cL}\to{}^{L}_{P}\mathcal{M}^{\cL}_{P},&\qquad \Lambda\mapsto \Lambda\ot_{\BB}P.\label{equ. Hopf bimodule to RL-YD}
    \end{align}
    More precisely, for \eqref{equ. RL-YD to Hopf bimodule}, the left $L$-module action is given by
    \[X\bla \rho=X\yi{}\rho \er{X},\]
     for any $\rho\in {}^{coL}M$ and $X\in\cL$. 
    For \eqref{equ. Hopf bimodule to RL-YD}, the balanced tensor product of $\Lambda\ot_{\BB}P$ is given by $s_{L}(b)\bla \eta\ot_{\BB}p=\eta\ot_{\BB}\Bar{b}p$, the left $L$-comodule structure is given by
    \[{}_{L}\delta(\eta\ot_{\BB} p)=p\mo\ot_{N}\eta\ot_{\BB}p\z.\]
   the right $\cL$-comodule structure is given by
\begin{align*}
    \delta_{\cL}(\eta\ot_{\BB} p)=\eta\z\ot_{\BB}p\z\di{}\eta\o p\o,
\end{align*}
The $P$-bimodule structure is given by
\[p(\eta\ot_{\BB}q)p'=p\o\bla\eta\ot_{\BB}p\z q p',\]
\end{proof}
As a result, we have
\begin{lem}
    Let $L$ be a $N$-Hopf algebroid and $\cL$ be a $B$-Hopf algebroid, if $P\in\Bi(L, \cL)$, then the Hopf bimodule category $({}^{L}_{P}\mathcal{M}^{\cL}_{P}, \ot_{P})$ is a braided monoidal category, namely the center of both monoidal categories $({}^{}_{P}\mathcal{M}^{\cL}_{P}, \ot_{P})$ and $({}^{L}_{P}\mathcal{M}^{}_{P}, \ot_{P})$. The half-braidings are given by
    \begin{align}
        \sigma(m\ot_{P}n)=&\tuno{m\mo} n\rz \yi{n\rmo}\ot_{P} \tdue{m\mo} m\z \er{n\rmo}\\
        \sigma^{-1}(n\ot_{P}m)=& \yi{n\o}m\rz \tuno{m\ro}\ot_{P}\er{n\o}n\z \tdue{m\ro},
    \end{align}
    defined for  $m\in M\in {}^{L}_{P}\mathcal{M}^{}_{P}$ and $n\in M'\in {}^{}_{P}\mathcal{M}^{\cL}_{P}$.
    \end{lem}
We omit the proof as it is similar to \cite{schau4}.    

 \section{Applications and examples}\label{Applications and examples}
In this section, we will study the 2-cocycle twist theory of Hopf Galois extensions with the base algebra deformed by a 2-cocycle.
Recall that\cite{HM25}, 
\begin{defi} \label{Lcotwist}
Let $\cL$ be a left $B$-bialgebroid. An \textup{left 2-cocycle} on
$\cL$ is an element $\Gamma\in {}\Hom_{\overline{B}-}(\cL\otimes_{\overline{B}}\cL, B)$, such that
\begin{itemize}
    \item [(1)]$\Gamma(X, \Gamma(Y\o, Z\o)Y\t Z\t)=\Gamma(\Gamma(X\o, Y\o)X\t Y\t, Z),$
    \item[(2)] $\Gamma(1_{\cL}, X)=\varepsilon(X)=\Gamma(X, 1_{\cL}),$
\end{itemize}
for all $X, Y, Z\in \cL$. 
\end{defi}
Given a 2-cocycle $\Gamma$ on a left $B$-bialgebroid, we can deform the base algebra $B$ with a new product 
\[a\cdot_{\Gamma}b=\Gamma(a, b),\]
    for any $a,b\in B$. We denote the new algebra by $B^{\Gamma}$. Moreover, for any left $\cL$-comodule $M$, $M$ has a $B^{\Gamma}$-bimodule structure given by
\begin{align}
a\cdot_{\Gamma}m=\Gamma(a,m\mo)m\z,\qquad m\cdot_{\Gamma}a=\Gamma(m\mo,a)m\z,
\end{align}
for any $a\in B$ and $m\in M$. For any $N,M\in {}^{\cL}\CM$, we can define $\Gamma^{\#}:M\ot_{B^{\Gamma}}N\to M\ot_{B}N$ given by
    \begin{align*}
        \Gamma^{\#}(m\ot_{B^{\Gamma}}n)=\Gamma(m\mo, n\mo)m\z\ot_{B}n\z,
    \end{align*}
    for any $m\in M, n\in N$. We call $\Gamma$ \textit{invertible}, if $\Gamma^{\#}$ is invertible for any left $\cL$-comodule $M,N$.
    \begin{thm}\cite{HM25}
        If $\Gamma$ is an invertible left 2-cocycle on a left $B$-Hopf algebroid then $\cL^\Gamma$ is a Hopf algebroid over $\BG$ with the $\BG^e$-ring structure
        \[ X\cdot_{\Gamma} Y:=\Gamma(X\o, Y\o)X\t{}_{+} Y\t{}_{+}\overline{\Gamma(Y\t{}_{-}, X\t{}_{-})},\]
        and $\BG$-coring structure 
         \[\Delta^{\Gamma}(X)=\Gamma^{\#-1}(X\o\di X\t),\qquad \varepsilon^{\Gamma}(X)=\Gamma(X_{+}, X_{-}),\]
     where $\Gamma^{\#}:\cL^{\Gamma}\diamond_{B^{\Gamma}} \cL^{\Gamma}\to \cL\di \cL$  given by \begin{align*}
    \Gamma^{\#}(X\diamond_{B^\Gamma} Y)=X_{+}\overline{\Gamma(X_{-}, Y\o)}\di Y\t,
\end{align*}
 is invertible as we consider the left term of $\cL\di \cL$ has the regular comodule structure induced by its coproduct. We denote the image of the twisted coproduct by
 \[\Delta^{\Gamma}(X)=X\tone\diamond_{B^{\Gamma}} X\ttwo,\]
 for any $X\in \cL$.
\end{thm}
Let $M$ be a left $\cL$-comodule, we can see $M$ is also a left $\cL^{\Gamma}$ comodule with the coproduct given by
\[{}_{\cL^{\Gamma}}\delta:=\Gamma^{\#-1}\circ {}_{\cL}\delta,\]
 where $\Gamma^{\#}:\cL^{\Gamma}\diamond_{B^{\Gamma}} M\to \cL\di M$  given by \begin{align*}
    \Gamma^{\#}(X\diamond_{B^\Gamma} m)=X_{+}\overline{\Gamma(X_{-}, m\mo)}\di m\z,
\end{align*}
 is invertible as $\Gamma$ is invertible. We denote the image of the twisted coaction by
 \[{}_{\cL^{\Gamma}}\delta(m)=m\tmo\diamond_{B^{\Gamma}} m\tz,\]
 \begin{prop}\label{prop. twsited comodule algebra}
       If $P$ is a left $\cL$-comodule algebra then $P$ is a left $\cL^{\Gamma}$-comodule algebra with a deformed product 
     \[ p\cdot_{\Gamma}q=\Gamma(p\mo, q\mo)p\z q\z,\]
      for any $p, q\in P$. We denote the new algebra by ${}_{\Gamma}P$.
 \end{prop}
\begin{proof}
      We can check that the coaction is an algebra map. Indeed,
   \begin{align*}
       \GH(&p\tmo\CG q\tmo\diamond_{\BG} p\tz\CG p\tz)\\
      =&\GH(\Gamma(p\tmo\o, q\tmo\o)p\rmo\t{}_{+}q\tmo\t{}_{+}\overline{\Gamma(q\tmo\t{}_{-}, p\tmo\t{}_{-})}\diamond_{\BG}\Gamma(p\tz\mo, q\tz\mo)p\tz\z q\tz\z)\\
      =&\Gamma(p\tmo\o, q\tmo\o)p\tmo\t{}_{++}q\tmo\t{}_{++}\\
      &\overline{\Gamma(\Gamma(q\tmo\t{}_{-}, p\tmo\t{}_{-})q\tmo\t{}_{+-}p\tmo\t{}_{+-}, \Gamma(p\tz\mt, q\tz\mt)p\tz\mo\, q\tz\mo)}\\
      &\di p\tz\z q\tz\z\\
      =&\Gamma(p\tmo\o, q\tmo\o)p\tmo\t{}_{+}q\tmo\t{}_{+}\\
      &\overline{\Gamma(\Gamma(q\tmo\t{}_{-}\o, p\tmo\t{}_{-}\o)q\tmo\t{}_{-}\t p\tmo\t{}_{-}\t, \Gamma(p\tz\mt, q\tz\mt)p\tz\mo\, q\tz\mo)}\\
      &\di p\tz\z q\tz\z\\
      =&\Gamma(p\tmo\o, q\tmo\o)p\tmo\t{}_{+}q\tmo\t{}_{+}\\
      &\overline{\Gamma(q\tmo\t{}_{-},\Gamma( p\tmo\t{}_{-}\o,\Gamma(p\tz\mth, q\tz\mth)p\rz\mt\, q\tz\mt)) p\tmo\t{}_{-}\t p\tz\mo\, q\tz\mo)}\\
      &\di p\tz\z q\tz\z\\
      =&\Gamma(p\tmo\o, q\tmo\o)p\tmo\t{}_{+}q\tmo\t{}_{+}\\
      &\overline{\Gamma(q\tmo\t{}_{-},\Gamma(\Gamma( p\tmo\t{}_{-}\o, p\tz\mth)p\tmo\t{}_{-}\t p\tz\mt, q\tz\mt) p\tmo\t{}_{-}\th p\tz\mo\, q\tz\mo)}\\
      &\di p\tz\z q\tz\z\\
      =&\Gamma(p\tmo\o, q\tmo\o)p\tmo\t{}_{+}q\tmo\t{}_{+}\\
      &\overline{\Gamma(q\tmo\t{}_{-},\Gamma(\Gamma( p\tmo\t{}_{-}, p\tz\mth)p\rmo\t{}_{+-}\o p\tz\mt, q\tz\mt) p\tmo\t{}_{+-}\t p\tz\mo\, q\tz\mo)}\\
      &\di p\tz\z q\tz\z\\
       =&\Gamma(p\mth, q\tmo\o)p\mt{}_{+}q\tmo\t{}_{+}\\
      &\overline{\Gamma(q\tmo\t{}_{-},\Gamma(p\mt{}_{-}\o p\mo\o, q\tz\mt) p\mt{}_{-}\t p\mo\t\, q\tz\mo)}\di p\tz\z q\tz\z\\
       =&\Gamma(p\mt, q\tmo\o)p\mo q\tmo\t{}_{+}\overline{\Gamma(q\tmo\t{}_{-}, q\tz\mo)}\di p\z q\tz\z\\
       =&\Gamma(p\mt, q\mt)p\mo q\mo\di p\z q\z\\
       =&{}^{\cL}\delta(p\CG q)=\GH({}^{\cL^{\Gamma}}\delta(p\CG q))
   \end{align*}
   where the 7th and 9th steps use the fact that
   \[p\tmo{}_{+}\overline{\Gamma(p\tmo{}_{-}, p\tz\mo)}\di p\tz\z=p\mo\di p\z.\]
\end{proof}
 \begin{lem}
    If $N\subseteq P$ is a left $\cL$- Galois extension and $\Gamma$ is an invertible left 2-cocycle on $\cL$ then $N\subseteq {}_{\Gamma}P$ is a left $\cL^{\Gamma}$- Galois extension. 
 \end{lem}
\begin{proof}
   First, we can see that the coinvariant subalgebra $N$ is not changed.   It is sufficient to show the following diagram commute:
     \[
\begin{tikzcd}
  &{}_{\Gamma}P\ot_{N} {}_{\Gamma}P \arrow[d, "\id\ot \id"] \arrow[r, "\lcan^{\Gamma}"] & \cL^{\Gamma}\diamond_{\BG} {}_{\Gamma}P  \arrow[d, "\Gamma^{\#}"] &\\
   & P\ot_{N} P   \arrow[r, "\lcan"] & \cL \di{}P. &
\end{tikzcd}
\]
On the one hand, 
\begin{align*}
    \Gamma^{\#}\circ \lcan^{\Gamma}(p\ot q)=&\GH(p\tmo \diamond_{\BG} p\tz\CG q)\\
    =&\GH(p\tmo\ot \Gamma(p\tz\mo,q\mo)p\tz\z q\z)\\
    =&p\tmo{}_{+}\overline{\Gamma(p\tmo{}_{-},\Gamma(p\tz\mt,q\mt)p\tz\mo q\mo)}\di p\tz\z q\z\\
=&p\tmo{}_{+}\overline{\Gamma(\Gamma(p\tmo{}_{-}\o, p\tz\mt) p\tmo{}_{-}\t p\tz\mo, q\mo)}\di p\tz\z q\z\\
 =&p\tmo{}_{++}\overline{\Gamma(\Gamma(p\tmo{}_{-}, p\tz\mt) p\tmo{}_{+-} p\tz\mo, q\mo)}\di p\tz\z q\z\\  
 =&p\mt{}_{+}\overline{\Gamma(p\mt{}_{-}p\mo, q\mo)}\di p\z q\z\\ 
 =&p\mo\di p\z q\\
 =&\lcan(p\ot q),
\end{align*}
where the 6th step use the fact that
\[p\tmo{}_{+}\overline{\Gamma(p\tmo{}_{-},p\tz\mo) }\di p\tz\z=p\mo\di p\z,\]
for any $p\in P$.
\end{proof}

\begin{lem}\label{lem. right Ehresamm don't change by the twist of left Hopf algebroid}
    Let $N\subseteq P$ be a left $\cL$-Galois extension and $\Gamma$ is a 2-cocycle on $\cL$. Then $R({}_{\Gamma}P, \cL^{\Gamma})\cong R(P, \cL)$.
\end{lem}
\begin{proof}
   By \Cref{lem. Hopf Galois extension is HOpf biGalois extension}, it is sufficient to show ${}_{\Gamma}P\in \Bi(\cL^{\Gamma}, R(P, \cL))$. First, we show that the ${}_{\Gamma}P$ is a right $R(P, \cL)$-comodule algebra with the original coaction $\delta_{R}$. Indeed,
   \begin{align*}
       \delta_{R}(p\CG q)=\delta_{R}(\Gamma(p\mo, q\mo)p\z q\z)=\Gamma(p\mo, q\mo)p\z q\z \ot p\o q\o=\delta_{R}(p)\CG \delta_{R}(q).
   \end{align*}
   Second, we show ${}_{\Gamma}B\subseteq {}_{\Gamma}P$ is an anti-right $R(P, \cL)$-Galois extension. It is sufficient to show the following diagram commute:
    \[
\begin{tikzcd}
  &{}_{\Gamma}P\ot_{\BG} {}_{\Gamma}P \arrow[d, "\Gamma^{\#}"] \arrow[r, "\rcan^{\Gamma}"] & {}_{\Gamma}P\diamond_{\overline{N}}R(P,\cL)  \arrow[d, "\id"] &\\
   & P\ot_{B} P   \arrow[r, "\rcan"] & P\diamond_{\overline{N}}R(P,\cL). &
\end{tikzcd}
\]
Third, we can show the left and right coaction cocommute. In other words, we need to show
\[({}_{\cL^{\Gamma}}\delta\diamond_{\overline{N}} \id)\circ \delta_{R}=(\id\diamond_{\BG} \delta_{R})\circ {}_{\cL^{\Gamma}}\delta.\]
So it is equivalent to show
\[(\GH\ot \id)\circ({}_{\cL^{\Gamma}}\delta\diamond_{\overline{N}} \id)\circ \delta_{R}=(\GH\ot \id)\circ(\id\diamond_{\BG} \delta_{R})\circ {}_{\cL^{\Gamma}}\delta.\]
On the one hand,
\begin{align*}
    (\GH\ot \id)\circ({}_{\cL^{\Gamma}}\delta\ot \id)\circ \delta_{R}(p)=p\mo\ot p\z\ot p\o.
\end{align*}
On the other hand,
\begin{align*}
    (\GH\ot \id)&\circ(\id\diamond_{\BG} \delta_{R})\circ {}_{\cL^{\Gamma}}\delta(p)\\
    =&p\tmo{}_{+}\overline{\Gamma(p\tmo{}_{-},p\tz\mo)}\ot p\tz\z\ot p\tz\o\\
    =&p\mo\ot p\z\ot p\o.
\end{align*}
\end{proof}
\subsection{Twisted Ehresmann Hopf algebroids}
Let $H$ be a Hopf algebra and $P$ be a right $H$-comodule algebra such that $B\subseteq P$ is a right $H$-Galois extension. Assume $K$ is another Hopf algebra, and $P$ is a $K$ comodule algebra such that $P$ is a $K$-$H$-bicomodule. If $\gamma:K\ot K\to k$ is an invertible 2-cocycle then ${}_{\gamma}P$ is a left $K^{\gamma}$ comodule algebra in the sense of \Cref{prop. twsited comodule algebra}. By \cite{ppca}, define ${}_{\gamma}B:={}_{\gamma}(P^{coH})=({}_{\gamma}P)^{coH}$ being the deformed algebra of $B$ with product
\[a\cdot_{\gamma}b=\gamma(a\mo, b\mo)a\z b\z,\qquad \forall a, b\in B,\]
 we know ${}_{\gamma}B\subseteq {}_{\gamma}P$ is a $H$-Galois extension with the twisted translation map given by
\begin{align}
  \teins{h}{}'\ot_{{}_{\gamma}N}\tzwei{h}{}'=\gamma^{-1}(\teins{h}\mo,\tzwei{h}\mo)\teins{h}\z\ot_{{}_{\gamma}N}\tzwei{h}\z,
\end{align}
where $\teins{h}\ot_{{}_{\gamma}N}\tzwei{h}$ is the image of the original translation map. 

\begin{lem}\label{lem. gamma induce Gamma}
    For $(H, K, P, B)$ be as above and $\gamma$ be an invertible 2-cocycle on $K$, we can construct an invertible left 2-cocycle $\Gamma$ on $L(P, H)$ by
    \[\Gamma(p\ot q, p'\ot q')=\gamma(p\mo, p'\mo)p\z p'\z q' q,\]
    for any $p\ot q, p'\ot q'\in L(P, H)$.
\end{lem}

\begin{proof}
    It is not hard to see $\Gamma$ factor through $\ot_{\overline{B}}$ and left $\overline{B}$-linear. We can also see the image of $\Gamma$ belongs to $B$. Moreover, $\Gamma$ is unital as $\gamma$ is. Now, let's check $\Gamma$ satisfies the 2-cocycle  condition. Let $X=p\ot q, Y=p'\ot q', Z=p''\ot q''\in L(P, H)$. We have on the one hand,
    \begin{align*}
        \Gamma&(\Gamma(X\o, Y\o)X\t Y\t, Z)\\
        =&\Gamma((\gamma(p\mo, p'\mo)p\z p'\z p'\o\teins{}p\o\teins{}p\o\tzwei{}\ot q)(p'\o\tzwei{}\ot q'), p''\ot q'')\\
        =&\Gamma(\gamma(p\mo, p'\mo)p\z p'\z \ot q' q, p''\ot q'')\\
        =&\gamma(\gamma(p\mt, p'\mt)p\mo p'\mo,p''\mo) p\z p'\z p''\z q'' q' q.
    \end{align*}
    On the other hand,
    \begin{align*}
         \Gamma&(X, \Gamma(Y\o, Z\o)Y\t Z\t)\\
         =&\Gamma(p\ot q, \gamma(p'\mo, p''\mo)p'\z p''\z \ot q'' q')\\
         =&\gamma(p\mo, \gamma(p'\mt, p''\mt)p'\mo p''\mo)p\z p'\z p''\z q'' q' q.
    \end{align*}
    They are equal as $\gamma$ is a 2-cocycle. Next, we will show $\GH:M\ot_{\BG}N\to M\ot_{B}N$ is invertible for any left $L(P, H)$-comodules $M$ and $N$. We can see the map ${}_{K}\delta:L(P, H)\to K\ot L(P, H)$, given by
    \[{}_{K}\delta(p\ot q)=p\mo\ot (p\z\ot q),\qquad \forall p\ot q\in L(P, H),\]
    is a left $K$-coaction on  $L(P, H)$. By applying the fact that $M\cong L(P, H)\Box M$ as left $L(P, H)$-comodule, it is not hard to see $\GH$ can be written explicitly by
\[\GH{}(m\ot n)=\gamma(m\mo\mo, n\mo\mo)\varepsilon(m\mo\z n\mo\z)m\z\ot n\z,\qquad \forall m\in M, n\in M.\]

    Therefore, we can define $\GH{}^{-1}:M\ot_{B}N\to M\ot_{\BG}N$ by
    \[\GH{}^{-1}(m\ot n)=\gamma^{-1}(m\mo\mo, n\mo\mo)\varepsilon(m\mo\z n\mo\z)m\z\ot n\z,\qquad \forall m\in M, n\in M.\]
    As $\gamma$ is convolution invertible, $\GH$ is invertible. Indeed, let $p\ot q\ot m\in L(P, H)\Box M$ and $p'\ot q'\ot n\in L(P, H)\Box N$. We can see
    \[\GH(p\ot q\ot m, p'\ot q'\ot n)=\gamma(p\mo,p'\mo)p\z p'\z q'q\ot m\ot n.\]
    As a result,
    \begin{align*}
        \GH\circ& \GH^{-1}(p\ot q\ot m, p'\ot q'\ot n)\\
        =&\gamma^{-1}(p\mo,p'\mo)p\z\,p'\z\,p'\o\teins{}\,p'\o\teins{}\gamma(p\o\tzwei{}\mo,p'\o\tzwei{}\mo)p\o\tzwei{}\z\,p'\o\tzwei{}\z\,q'\,q\ot m\ot n\\
        =&\gamma^{-1}(p\mo,p'\mo)\gamma(p\z\mo{},p'\z\mo)p\z\z\,p'\z\z\,q'\,q\ot m\ot n\\
        =&pp'q'q\ot m\ot n\\
        =&1\ot m\ot n.
    \end{align*}
\end{proof}

\begin{thm}\label{thm. Ehresmann of twist is equal to the twist of Ehresmann}
    For $(H, K, P, B)$ as above and $\gamma$ be an invertible 2-cocycle on $K$, we have $L({}_{\gamma}P, H)=L(P, H)^{\Gamma}$, where $\Gamma$ is an invertible left 2-cocycle given by \Cref{lem. gamma induce Gamma}.
\end{thm}

\begin{proof}
We first observe that $B^{\Gamma}={}_{\gamma}B$. Indeed,
\begin{align*}
    a\CG b=\Gamma(a, b)=\gamma(a\mo, b\mo)a\z b\z=a\cdot_{\gamma}b, 
\end{align*}
for any $a, b\in B$.
    We can see the $\BG^{e}$-ring structure on $L(P, H)^{\Gamma}$ is given by
    \begin{align*}
        b\CG (p\ot q)=\gamma(b, p\mo) b\z p\z\ot q;\qquad (p\ot q)\CG b=\gamma(p\mo, b) p\z b\z\ot q,
    \end{align*}
    and
    \begin{align*}
        \overline{b}\CG (p\ot q)=&(p\ot q\o\tzwei{})\overline{\Gamma((q\z\ot q\o\teins{}),b)}\\
        =&p\ot \gamma(q\mo, b\mo)q\z b\z q\o\teins{}q\o\tzwei{}\\
        =&p\ot \gamma(q\mo, b\mo)q\z b\z, 
    \end{align*}
    where we use 
    \[(p\ot q)_{+}\ot_{\BB}(p\ot q)_{-}=p\ot q\o\tzwei{}\ot_{\BB}q\z\ot q\o\teins{}\]
    in the  2nd step which is shown in \cite{HM22}. Similarly,
    \[ (p\ot q)\CG \overline{b}=p\ot \gamma(b\mo, q\mo)b\z q\z.\]
    We can also see,
    \begin{align*}
    (p&\ot q)\CG(p'\ot q')\\
    =&\Gamma((p\z \ot p\o\teins{}), (p'\z\ot p'\o\teins{}))(p\o\tzwei{}p'\o\tzwei{}\ot 
 q'\o\tzwei{}q\o\tzwei{})\overline{\Gamma((q'\z\ot q'\o\teins{}),(q\z\ot q\o\teins{}))}\\
 =&\gamma(p\mo, p'\mo)p\z p'\z p'\o\teins{} p\o\teins{}p\o\tzwei{}p'\o\tzwei{}\ot \gamma(q'\mo, q\mo)q'\z q\z q\o\teins{} q'\o\teins{}q'\o\tzwei{}q\o\tzwei{}\\
 =&\gamma(p\mo, p'\mo)p\z p'\z\ot \gamma(q'\mo, q\mo)q'\z q\z.
    \end{align*}
 So  $L({}_{\gamma}P, H)=L(P, H)^{\Gamma}$ as a ${}_{\gamma}B^{e}$-ring. Now, let's check that they have the same coring structure. We can see
 \begin{align*}
     \varepsilon^{\Gamma}(p\ot q)=&\Gamma((p\ot q\o\tzwei{}), (q\z\ot q\o\teins{}))=\gamma(p\mo, q\mo)p\z q\z q\o\teins{}q\o\tzwei{}\\
     =&\gamma(p\mo, q\mo)p\z q\z.
 \end{align*}
 We can also see 
 \[\Delta^{\Gamma}(p\ot q)=p\z\ot p\o\teins{}\z\gamma^{-1}(p\o\teins{}\mo, p\o\tzwei{}\mo)\diamond_{\BG} p\o\tzwei{}\z\ot q.\]
 Indeed, by applying $\Gamma^{\#}$ on the right hand side, we get
 \begin{align*}
     (p\z\ot& p\o\teins{}\o\tzwei{})\gamma^{-1}(p\o\teins{}\mo, p\o\tzwei{}\mo)\overline{\Gamma((p\o\teins{}\z\ot p\o\teins{}\o\teins{}), (p\o\tzwei{}\z\ot p\o\tzwei{}\o\teins{}))}\\
     &\di p\o\tzwei{}\o\tzwei{}\ot q\\
     =&(p\z\ot p\o\teins{}\z p\o\tzwei{}\z p\o\tzwei{}\o\teins{} p\o\teins{}\o\teins{} p\o\teins{}\o\tzwei{})\\
     &\gamma^{-1}(p\o\teins{}\mt, p\o\tzwei{}\mt)\gamma(p\o\teins{}\mo, p\o\tzwei{}\mo)\di p\o\tzwei{}\o\tzwei{}\ot q\\
     =&p\z\ot p\o\teins{}\di p\o\tzwei{}\ot q\\
     =&\Delta(p\ot q).
 \end{align*}
\end{proof}

Let $H$ be a Hopf algebra and $B\subseteq P$ be a right $H$-Galois extension. If $\sigma:H\ot H\to k$ is a 2-cocycle, by \cite{MS} $B\subseteq P_{\sigma}$ is a $H^{\sigma}$-Galois extension with the original $H$-coaction on $P$ and a twisted product
\[p\cdot_{\sigma}q=p\z q\z \sigma^{-1}(p\o, q\o),\]
for any $p, q\in P$. By \cite{ppca}, if $(H, K, P, B)$ is as above, $\gamma$ is a 2-cocycle on $K$ and $\sigma$ is a 2-cocycle on $H$, then ${}_{\gamma}(P_{\sigma})=({}_{\gamma}P)_{\sigma}=:{}_{\gamma}P_{\sigma}$. Moreover, ${}_{\gamma}B\subseteq {}_{\gamma}P_{\sigma}$ is a $H^{\sigma}$-Galois extension. As a result of \Cref{lem. right Ehresamm don't change by the twist of left Hopf algebroid} and \Cref{thm. Ehresmann of twist is equal to the twist of Ehresmann}, we have
\begin{cor}\label{cor. Ehresmann Hopf algebroid by two sides twist}
    If $(H, K, P, B)$ is as above, $\gamma$ is a 2-cocycle on $K$ and $\sigma$ is a 2-cocycle on $H$, then $L({}_{\gamma}P_{\sigma}, H^{\sigma})\cong L({}_{\gamma}P, H)=L(P, H)^{\Gamma}$.
\end{cor}

\subsection{Examples}

\subsubsection{The $SU(2)$ principal fibration} 
 Let $P=A(S^7)$, $B=A(S^4)$, $H=A(SU(2))$ and $K=A(\mathbb{T}^2)$. More precisely, $P=A(S_{\theta}^7)$ is generated by elements
$z_a, z_a^*$, $a=1,\dots,4$ with with the spherical relation $\sum_a z_a^* z_a=1$. The Hopf algebra is $A(SU(2))$ is an unital complex $*$-algebra generated by $\omega_{1}, \overline{\omega}_{1}, \omega_{2}, \overline{\omega}_{2}$ subject to the relation $\omega_{1}\overline{\omega}_{1}+\omega_{2}\overline{\omega}_{2}=1$. The coproduct, counit and antipode is given by:
\begin{align*}
\Delta : \begin{pmatrix} \omega_{1} &-\overline{\omega}_{2}\\\omega_{2} & \overline{\omega}_{1} \end{pmatrix}\mapsto  \begin{pmatrix} \omega_{1} &-\overline{\omega}_{2}\\\omega_{2} & \overline{\omega}_{1} \end{pmatrix} \otimes  \begin{pmatrix} \omega_{1} &-\overline{\omega}_{2}\\\omega_{2} & \overline{\omega}_{1} \end{pmatrix},
\end{align*}
with counit $\varepsilon(\omega_{1})=\varepsilon(\overline{\omega}_{1})=1$, $\varepsilon(\omega_{2})=\varepsilon(\overline{\omega}_{2})=0$ and antipode $S(\omega_{1})=\overline{\omega}_{1}$, $S(\omega_{2})=-\omega_{2}$. If we denote  $A(S^{7})$ and $A(SU(2))$ by matrix-valued function by
\begin{align*}
    \Psi  =
\begin{pmatrix}
z_1 & - z^*_2 \\
z_2 & z^*_1 \\
z_3 & -z^*_4 \\
z_4& z^*_3
\end{pmatrix},\qquad
\omega=\begin{pmatrix}
    \omega_{1} &-\overline{\omega}_{2}\\\omega_{2} & \overline{\omega}_{1}
\end{pmatrix},
\end{align*}
then the right coaction can be written as $\delta(\Psi)=\Psi\ot\omega$. This means in components $\delta(\Psi_{ik})=\Psi_{ij}\ot\omega_{jk}$. The algebra $A(S^{4})$ generated by   $\zeta_1=  z_1 z^{*}_3 + z^*_2 z_4$, $\zeta_2 =  z_2 z^*_3 - z^*_1 z_4$ and $\zeta_0 = z_1 z^*_1 + z^*_2 z_2 = 1 - z_3 z^*_3 - z^*_4 z_4$
is the subalgebra of coinvariant. Moreover, $A(S^{4})\subseteq A(S^{7})$ is a right $A(SU(2))$-Galois extension. Let $t_{i}, t_{i}^{*}, i=1,2$ be the generators of $A(\mathbb{T}^2)$, the left $K$-coaction of $P$ is given by
\[{}^{K}\delta(z_{i})=\tau_{i}\ot z_{i},\qquad{}^{K}\delta(z_{i}^*)=\tau_{i}^*\ot z_{i}^*,\]
where $\tau_{i}=(t_{1}, t_{1}^{*}, t_{2}, t_{2}^{*})$. The left $K$-coaction on $B$ is given by
\[{}^{K}\delta(\zeta_{1})=t_{1}t_{2}^{*}\ot \zeta_{1},\quad{}^{K}\delta(\zeta_{2})=t_{1}^{*}t_{2}^{*}\ot \zeta_{1},\quad \quad{}^{K}\delta(\zeta_{0})=1\ot \zeta_{0}.\]

Define $\gamma:K\ot K\to k$ a 2-cocycle which is given by
\begin{align}\label{equ. first twist}
    \gamma(t_i, t_j)=e^{i\pi \Theta_{ij}},\qquad
\Theta_{ij}=\frac{1}{2}\begin{pmatrix}
    0 &\theta\\\theta & 0
\end{pmatrix}.
\end{align}
By \cite{ppca}, ${}_{\gamma}P=A(S^{7}_{\theta})$ and ${}_{\gamma}B=A(S_{\theta}^{4})$. 
It is given in \cite{LS04} that $A(S_{\theta}^4)\subseteq A(S_{\theta}^7)$ is a right $A(SU(2))$-Galois extension. More precisely, $A(S_{\theta}^7)$ is generated by elements
$z_a, z_a^*$, $a=1,\dots,4$, subject to relations
\begin{align}
    \label{s7t}
z_a z_b = \lambda_{a b} \, z_b z_a, \quad  z_a z_b^* = \lambda_{b a} \, z_b^* z_a,
\quad z_a^*z_b^* = \lambda_{a b} \, z_b^* z_a^* ,
\end{align}
and with the spherical relation $\sum_a z_a^* z_a=1$, where $\lambda_{a b} = e^{2 \pi i \theta_{ab}}$ and $(\theta_{ab})$ a real antisymmetric matrix given by
\beq\label{lambda7}
\lambda_{ab}=
\begin{pmatrix} 1 & 1 & \bar{\mu} & \mu \\
1 & 1 & \mu & \bar{\mu} \\
\mu & \bar{\mu} &1 & 1\\
\bar{\mu} & \mu &1 & 1
\end{pmatrix}, \quad \mu = \sqrt{\lambda} \qquad \mathrm{or} \qquad
\theta_{ab}=\frac{\theta}{2}\begin{pmatrix} 0 & 0 & -1 & 1 \\
0 & 0 & 1 & -1 \\
1 & -1 & 0 & 0 \\
-1 & 1 & 0 & 0  \end{pmatrix}.
\eeq
The generators of $A(S_{\theta}^{4})$ has the relations
\[\zeta_1 \zeta_2  = \lambda \zeta_2 \zeta_1,\quad \zeta_1 \zeta_2^* = \bar{\lambda} \zeta_2^* \zeta_1,\quad \zeta_1^* \zeta_1 + \zeta_2^* \zeta_2 = \zeta_0 (1-\zeta_0).\]
It is given by \cite{HLL23} that $L(A(S_{\theta}^7), A(SU(2)))$ has generator $p \ot 1$, $1\ot {q}$ and $V:=\Psi\ot \Psi^{\dagger}$, where 
  \begin{align}
    p = \Psi \cdot \Psi^\dagger =
\begin{pmatrix}
\zeta_0 & 0 & \zeta_1 & - \bar{\mu} \zeta_2^* \\
0 & \zeta_0 & \zeta_2  & \mu \zeta_1^* \\
\zeta_1^*& \zeta_2^* & 1-\zeta_0 & 0\\
-\mu \zeta_2 & \bar{\mu}   \zeta_1 & 0 & 1-\zeta_0
\end{pmatrix},\quad q =  \Psi \cdot_{op} \Psi^\dagger =
\begin{pmatrix}
\zeta_0 & 0 & \bar{\mu} \zeta_1 & - \zeta_2^* \\
0 & \zeta_0 &  \mu \zeta_2  & \zeta_1^* \\
 \mu \zeta_1^*& \bar{\mu} \zeta_2^* & 1- \zeta_0 & 0\\
- \zeta_2 & \zeta_1 & 0 & 1- \zeta_0
\end{pmatrix}.
\end{align}
By \Cref{thm. Ehresmann of twist is equal to the twist of Ehresmann}, we have
\begin{lem}
Let $P=A(S^7)$, $B=A(S^4)$, $H=A(SU(2))$, $K=A(\mathbb{T}^2)$ and $\gamma$ be the 2-cocycle given in \eqref{equ. first twist}. Then    $L(A(S^7), A(SU(2)))^{\Gamma}=L(A(S^{7}_{\theta}), A(SU(2)))$, where $\Gamma$ is an invertible left 2-cocycle given by \Cref{lem. gamma induce Gamma}.
\end{lem}

      \subsubsection{Quantum homogeneous spaces of quantum groups $SO_\theta(2n+1, \mathbb{R})$\cite{CL01,Var01}}

 Let $P=A(SO(2n+1, \mathbb{R}))$, $B=A(S^{2n})$, $H=A(SO(2n, \mathbb{R}))$ and $K=A(\mathbb{T}^n)$.  We first define $A(M(2n, \mathbb{R}))$ being a bialgebra generated by elements  $\mathbf{a} = (a_{jk})$, $\mathbf{b} = (b_{jk})$,
$\mathbf{a}^* = (a_{jk}^*)$, $\mathbf{b}^* = (b_{jk}^*), i, j,k=1,\dots n$. In matrix notation $A(M(2n, \mathbb{R}))$ has coproduct and counit given by
\beq%
M = (M_{J K}) =
\begin{pmatrix}  
\mathbf{a} & \mathbf{b} \\
\mathbf{b}^* & \mathbf{a}^*
\end{pmatrix} , \qquad \Delta(M) = M \ot M , \qquad \varepsilon(M) = 1_{2n}  .
\eeq
The Hopf algebra $A(SO(2n, \mathbb{R}))$ is given by the quotient of $A(M(2n, \mathbb{R}))$ by a Hopf ideal $I_{Q}$ which is given by 
\beq%
I_Q = <M^t Q M - Q , \, M Q M^t - Q , \, {\det}(M) - 1 > , \qquad
Q = \begin{pmatrix}
0 & 1_n \\
1_n & 0
\end{pmatrix} = Q^{-1}.
\eeq
Similarly, $A(M(2n+1, \mathbb{R}))$ is generated by $\mathbf{a} = (a_{jk})$, $\mathbf{b} = (b_{jk})$,
$\mathbf{a}^* = (a_{jk}^*)$, $\mathbf{b}^* = (b_{jk}^*)$, and $u_i, u_{i}^{*}, v_{i}, v_{i}^{*}$, $i, j,k=1,\dots n$ and a hermitian scalar $x$. In matrix notation the coproduct and counit are given by
\beq%
N =
\begin{pmatrix}
\mathbf{a} & \mathbf{b} & \mathbf{u} \\
\mathbf{b}^* & \mathbf{a}^* & \mathbf{u}^*\\
\mathbf{v} & \mathbf{v}^* & x\\
\end{pmatrix} , \qquad \Delta(N) = N \ot N , \qquad \varepsilon(N) = 1_{2n+1}  .
\eeq
The Hopf algebra $A(SO(2n+1, \mathbb{R}))$ is given by the quotient of $A(M(2n+1, \mathbb{R}))$ by a Hopf ideal $J_{Q}$, which is given by  
\beq%
J_Q = <N^t Q N = Q , \quad N Q N^t = Q , \quad {\det}(N) = 1
> , \qquad
Q = \begin{pmatrix}
0 & 1_n & 0 \\
1_n & 0 & 0 \\
0 & 0 & 1
\end{pmatrix} = Q^{-1}.
\eeq
The Hopf algebra $K$ is generated by $t_{i}, t_{i}^{*}$ subject to the relation $\sum_{i}t_{i} t_{i}^{*}=1$. The coproduct and counit are
\[\Delta(T)=T\ot T,\quad\varepsilon(T)=1_{2n},\quad T:=\textup{diag}\{t_{1},\dots t_{n},t_{1}^{*},\dots t_{n}^{*} \}.\]
There is a surjective
Hopf algebra morphism $\pi:P\to H$  given by
\beq%
\pi:
\begin{pmatrix}
\mathbf{a} & \mathbf{b} & \mathbf{u} \\
\mathbf{b}^* & \mathbf{a}^* & \mathbf{u}^*\\
\mathbf{v} & \mathbf{v}^* & x\\
\end{pmatrix}  \quad\mapsto \quad 
\begin{pmatrix}
\mathbf{a} & \mathbf{b} & 0 \\
\mathbf{b}^* & \mathbf{a}^* & 0\\
0 & 0 & 1\\
\end{pmatrix}
\eeq
Therefore, $P$ is a right $H$-comodule algebra with coaction given by $\delta^{H}=(\id\ot\pi)\circ\Delta:P\to P\ot H$.
The algebra $B=A(S^{2n})$ is generated by $u_i, u_{i}^{*}$ and $x$ subject to the relation $\sum_{i}2u_i u_{i}^{*}+x^2=1$.  Moreover, $B=A(S^{2n})\subseteq P$ is a $H$-Galois extension.
There is another surjective
Hopf algebra morphism $\pi':H\to K$ given by
\[\pi': a_{ij}\mapsto \delta_{ij}t_{i},\quad a_{ij}^{*}\mapsto \delta_{ij}t_{i}^{*},\quad b_{ij}\mapsto 0, \quad b_{ij}^{*}\mapsto 0.\]
Therefore, $P$ is a left $K$-comodule algebra with coaction given by $\delta^{H}=(\pi'\circ \pi\ot\id)\circ\Delta:P\to K\ot P$.

 We can define a 2-cocycle $\gamma$ of $K$ on the generators by
 \begin{align}\label{equ. second twist}
  \gamma(t_{i}, t_{j})=e^{i\pi \theta_{ij}},\qquad \theta_{ij}=-\theta_{ji}.   
 \end{align}
There is a 2-cocycle $\sigma$ on $H$ given by $\sigma(g,h)=\gamma(\pi'(g),\pi'(h))$ for any $g,h\in H$. By \cite{ppca}, ${}_{\gamma}P_{\sigma}=A(SO_{\theta}(2n+1, \mathbb{R}))$, $H^{\sigma}=A(SO_{\theta}(2n, \mathbb{R}))$ and ${}_{\gamma}B=A(S_{\theta}^{2n})$. It is given by \cite{CL01,Var01}, $A(S_{\theta}^{2n})\subseteq A(SO_{\theta}(2n+1, \mathbb{R}))$ is a $A(SO_{\theta}(2n, \mathbb{R}))$-Galois extension. More precisely, the generators $\mathbf{a} = (a_{jk})$, $\mathbf{b} = (b_{jk})$,
$\mathbf{a}^* = (a_{jk}^*)$, $\mathbf{b}^* = (b_{jk}^*)$ of $A(SO_{\theta}(2n, \mathbb{R}))$ satisfy the relations:
\begin{align}\label{thetaCR}
a_{ij}     a_{kl} & = \lambda_{ik}\lambda_{lj} ~a_{kl}    a_{ij} ,  \qquad
a_{ij}     b^*_{kl} = \lambda_{ki}\lambda_{lj} ~b^*_{kl}    a_{ij} \nn
\\
a_{ij}     b_{kl} & = \lambda_{ik}\lambda_{jl} ~b_{kl}    a_{ij}  ,  \qquad
a_{ij}     a^*_{kl} = \lambda_{ki}\lambda_{jl} ~a^*_{kl}    a_{ij} \nn
\\
b_{ij}     b_{kl} & = \lambda_{ik}\lambda_{lj} ~b_{kl}    b_{ij}  ,  \qquad
b_{ij}     b^*_{kl} = \lambda_{ki}\lambda_{jl} ~b^*_{kl}    b_{ij}
\end{align}
together with their $*$-conjugated. The commutation relations of $A(SO_{\theta}(2n+1, \mathbb{R}))$ is given by
\beq\label{crN}
N_{I J} \, N_{K L} = \lambda_{IK} \lambda_{L J} N_{K L} \, N_{I J},
\eeq  
where $\lambda_{IK}$ is induced by the 2-cocycle $\gamma$ and $I,J,K,L=1,\dots 2n$. The generators $u_{i}, u_{i}^{*}, x$ of $S^{2n}_{\theta}$ satisfy \eqref{crN}:
\beq
u_i     u_j= \lambda_{ij} \, u_j     u_i \, , \qquad
u_i^*     u^*_j= \lambda_{ij} \, u^*_j     u^*_i \, ,
 \qquad u_i     u_j^*= \lambda_{ji} \, u_j^*  u_i \, ,
\eeq
and $x$ central. The orthogonality conditions imply the sphere
relation
$$
\sum_{j=1}^{n} 2 u_j^* u_j + x^2 = 1 .
$$
It is given by \cite{HLL23} that $L(A(SO_{\theta}(2n+1, \mathbb{R})),A(SO_{\theta}(2n, \mathbb{R})))$ is generated
by elements $1\ot \mathbf{u}, 1\ot \mathbf{u}^*, 1 \ot x$, and $\mathbf{u} \ot 1, \mathbf{u}^* \ot 1, x \ot 1$,
together with the entries of the matrix
$$
\mathbf{V} = 
\begin{pmatrix}
\mathbf{a} & \mathbf{b} \\
\mathbf{b}^* & \mathbf{a}^* \\
\mathbf{v} & \mathbf{v}^*
\end{pmatrix} \ot
\begin{pmatrix}
\mathbf{a} & \mathbf{b} \\
\mathbf{b}^* & \mathbf{a}^* \\
\mathbf{v} & \mathbf{v}^*
\end{pmatrix}^\dagger,
$$

By \Cref{cor. Ehresmann Hopf algebroid by two sides twist}, we have
 \begin{lem}
Let $P=A(SO_{\theta}(2n+1, \mathbb{R}))$, $B=A(S^{2n})$, $H=A(SO_{\theta}(2n, \mathbb{R}))$, $K=A(\mathbb{T}^{2n})$ and $\gamma$ be the 2-cocycle on $A(\mathbb{T}^{2n})$ in \eqref{equ. second twist}. Then    $L(A(SO(2n+1, \mathbb{R})), A(SO(2n, \mathbb{R})))^{\Gamma}\cong L(A(SO_{\theta}(2n+1, \mathbb{R})), A(SO_{\theta}(2n, \mathbb{R})))$, where $\Gamma$ is an invertible left 2-cocyle given by \Cref{lem. gamma induce Gamma}.
\end{lem}

\end{document}